\documentclass[opre,nonblindrev]{informs3opt} 

\usepackage{endnotes}
\let\footnote=\endnote

\usepackage{natbib}
\bibpunct[,]{(}{)}{,}{a}{}{,}%
\def\bibfont{\small}%

\usepackage{natbib}
\usepackage{titlesec,soul}
\usepackage[titletoc,toc,title]{appendix}
\usepackage{graphicx}
\usepackage{fancyhdr}
\usepackage{subfigure}

\usepackage{multirow,todonotes}

\usepackage{epstopdf}

\usepackage{xcolor}

\usetikzlibrary{shapes,decorations,arrows,calc,arrows.meta,fit,positioning}
\tikzset{
	-Latex,auto,node distance =1 cm and 1 cm,semithick,
	state/.style ={ellipse, draw, minimum width = 0.7 cm},
	point/.style = {circle, draw, inner sep=0.04cm,fill,node contents={}},
	bidirected/.style={Latex-Latex,dashed},
	el/.style = {inner sep=2pt, align=left, sloped}
}

\usepackage[hidelinks,bookmarksnumbered=true]{hyperref}
\usepackage{todonotes,bm}
\usepackage{tikz}
\usepackage{pgfplots}
\pgfplotsset{compat=newest}
\pgfdeclareplotmark{mystar}{
	\node[star,star point ratio=2.25,minimum size=6pt,
	inner sep=0pt,draw=black,solid,fill=red] {};
}
\usetikzlibrary{decorations.pathreplacing}

\usepgflibrary{arrows}

\usetikzlibrary{calc}
\usetikzlibrary{plotmarks}

\usepackage{url}
\usepackage{amssymb}
\usepackage{psfrag}
\usepackage{amsmath}
\usepackage{setspace}
\newcommand{\exclude}[1]{}
\usepackage{bm}
\usepackage{mathtools}
\usepackage{bbm}
\usepackage[flushleft]{threeparttable}
\usepackage{algorithm}
\usepackage{algpseudocode}
\algdef{SE}[DOWHILE]{Do}{doWhile}{\algorithmicdo}[1]{\algorithmicwhile\ #1}%
\algnewcommand{\Or}{\textbf{or}}
\algnewcommand{\And}{\textbf{and}}

\usepackage{thmtools} 
\usepackage{thm-restate}

\usepackage{cleveref}

\declaretheorem[name=Proposition]{proposition}

\declaretheorem[name=Claim]{claim}
\declaretheorem[name=Corollary]{corollary}

\declaretheorem[name=Example]{example}

\def\E{{\mathbb E}}

\def\Pr{{\mathbb{P}}}

\def\Re{\mathbb{R}}

\def\I{\mathbb{I}}

\def\hat{\widehat}

\def \bxi{\boldsymbol{\xi}}

\def \P{\mathcal{P}}
\def \F{\mathcal{F}}

\def \Ze{{\mathbb{Z}}}

\def\one{\I}

\def\O{{\mathcal O}}

\def\N{{\mathcal N}}
\def\F{{\mathcal F}}

\def\P{{\mathcal P}}

\def\X{{\mathcal X}}

\def\C{{\mathcal C}}

\def\Re{{\mathbb R}}

\def\O{{\mathcal O}}

\newcommand{\rxi}{\bm{\xi}}
\newcommand{\rzeta}{\bm{\zeta}}

\newcommand{\e}{\mathbf{e}}

\def\alsox{{\mathrm{ALSO}{\text-}\mathrm{X}}}
\def\DC{{\mathrm{DC}}}
\def\AM{{\mathrm{AM}}}

\def\CVaR{{\mathrm{CVaR}}}
\def\VaR{{\mathrm{VaR}}}

\def\alsoxt{{\mathbf{ALSO}{\text-}\mathbf{X}}}
\def\CVaRt{{\mathbf{CVaR}}}
\newcommand{\trxi}{\tilde{\bm{\xi}}}

\newcommand{\trzeta}{\tilde{\bm{\zeta}}}

\DeclarePairedDelimiter\ceil{\lceil}{\rceil}
\DeclarePairedDelimiter\floor{\lfloor}{\rfloor}
\DeclareMathOperator{\sign}{sign}

\usepackage{comment}

\usepackage{enumitem}
\usepackage{makecell}

\usepackage{wrapfig}

\usepackage[outline]{contour}

\def\x{\vect{x}}

\newcommand{\vect}[1]{\boldsymbol{\bm{#1}}}

\renewcommand{\proof}{\noindent{\em Proof. }}
\newcommand*{\QEDA}{\hfill\ensuremath{\square}}
\newcommand*{\QEDB}{\hfill\ensuremath{\diamond}}

\definecolor{mygreen}{RGB}{28,172,0} 
\definecolor{mypurple}{RGB}{170,55,241}
\definecolor{myorange}{rgb}{0.77,0.38,0.06}

\allowdisplaybreaks

\begin{document}
	\TITLE{ALSO-X and ALSO-X+: \\
		Better Convex Approximations for Chance Constrained Programs}
	
	\ARTICLEAUTHORS{%
		\AUTHOR{Nan Jiang}
		\AFF{Department of Industrial \& Systems Engineering, Virginia Tech, Blacksburg, VA 24061, \EMAIL{jnan97@vt.edu} \URL{}}
		\AUTHOR{Weijun Xie}
		\AFF{Department of Industrial \& Systems Engineering, Virginia Tech, Blacksburg, VA 24061, \EMAIL{wxie@vt.edu} \URL{}}
	} 
	\RUNAUTHOR{Nan Jiang and Weijun Xie}
	
	\RUNTITLE{ALSO-X and ALSO-X+}
	\ABSTRACT{%
		{In a chance constrained {program} (CCP), {decision-makers seek} the best decision whose probability of violating the uncertainty constraints is within the prespecified risk level.}
		{As a CCP is often nonconvex and is difficult to solve to optimality,} much effort has been devoted to developing convex inner approximations for a CCP, among which the conditional value-at-risk ($\CVaR$) has been known to be the best for more than a decade. This paper studies and generalizes the $\alsox$, originally proposed by \textbf{A}hmed, \textbf{L}uedtke, \textbf{SO}ng, and \textbf{X}ie (2017), for solving a CCP. We first show that the $\alsox$ resembles a bilevel optimization, where the upper-level problem is to find the best objective function value and enforce the feasibility of a CCP for a given decision from the lower-level problem, and the lower-level problem is to minimize the expectation of constraint violations subject to the upper bound of the objective function value provided by the upper-level problem. This interpretation motivates us to prove that when uncertain constraints are convex in the decision variables, $\alsox$ always outperforms the $\CVaR$ approximation. We further show (i) sufficient conditions under which $\alsox$ can recover an optimal solution to a CCP; (ii) an equivalent bilinear programming formulation of a CCP, inspiring us to enhance $\alsox$ with a convergent alternating minimization method ($\alsox +$); (iii) {an extension of $\alsox$ and $\alsox +$ 
			to distributionally robust chance constrained programs (DRCCPs) under $\infty-$Wasserstein ambiguity set.} Our numerical study demonstrates the effectiveness of the proposed methods.
	}%
	
	
	
	\maketitle
	
	\section{Introduction}
	Let us consider a chance constrained program (CCP) of the form
	\begin{align}
	\label{eq_ccp}
	v^* =\min_{\bm x \in \mathcal{X} } \left\{ \bm{c}^\top\bm x\colon\Pr\left\{\tilde{\bm\xi} \colon g(\bm x,\tilde{\bm\xi})\leq 0\right\} \geq1-\varepsilon\right \}.
	\end{align}
	The goal of CCP \eqref{eq_ccp} is to find a solution $\bm x\in \mathcal{X}$ that minimizes the objective $\bm{c}^\top\bm x$ and is subject to the uncertain constraints $g(\bm x,\tilde{\bm\xi})\leq 0$ satisfied with probability $1-\varepsilon$, where $\varepsilon\in (0,1)$ is a preset risk level. In this paper, we focus on the convex setting, i.e., throughout this paper, we make the following assumptions
	\begin{enumerate}[label={A\arabic*}]
		\item \label{A_1} {Given a probability space $(\Omega,\F,\Pr)$ where the probability measure $\Pr$ is defined on the measurable space $(\Omega,\F)$ equipped with the sigma algebra $\mathcal{F}$, the random vector $\trxi:\Omega\rightarrow\Xi$ is a measurable mapping from $\Omega$ {to} $\Re^{m}$ {with support set} $\Xi\subseteq\Re^m$.} Function $ g(\bm{x},\bm{\xi})= \max_{i\in[I]} g_i(\bm x,\bm\xi)$, where $g_i(\bm x,\bm\xi)\colon \Re^n\times \Xi\rightarrow {\Re}$ for all $i\in[I]: =\{1,\dots,I\}$ and $g_i(\bm x,\bm {\xi})$ is convex and {lower semi-continuous} in $\bm x$ for {almost every} $\bm {\xi}\in\Xi$; 
	{\item \label{A_2} Set $\X$ is nonempty, closed, convex, contained in a closed convex cone $\C$. That is, $\emptyset\neq\X\subseteq \C$, where $\C$ is a closed pointed convex cone; and
		\item   \label{A_3}   The feasible region of CCP \eqref{eq_ccp} is nonempty and the objective cost vector $\bm c\in \mathrm{int}(\C^*)\cup\{\bm 0\}$, where $\C^*$ is the dual cone of $\C$ and  $\mathrm{int}(\cdot)$ denotes the interior of a set.}
	\end{enumerate}
	If $I = 1$, CCP \eqref{eq_ccp} involves a single chance
	constraint, and otherwise, it contains a joint chance constraint. 
	Note that Assumption~\ref{A_1} follows from existing chance constraint literature (see, e.g., \citealt{nemirovski2007convex,ahmed2017nonanticipative,ahmed2018relaxations}) {and the lower semi-continuity assumption in Assumption~\ref{A_1} and the closedness of set $\X$ in Assumption~\ref{A_2}  together  guarantee that the feasible region of CCP \eqref{eq_ccp} is closed (see Proposition~\ref{prop_lower_continuous_closed} in Appendix~\ref{sec_append_lower_continuous_closed}),} while { Assumption~\ref{A_2} is quite standard for convex analysis (see, e.g., section 5 in \citealt{nemirovski2001lectures}), and Assumption~\ref{A_3} guarantees that there exists an optimal solution in CCP \eqref{eq_ccp} (see also the discussions in \citealt{xie2020bicriteria}), 
for the sake of simplicity. }  
It is worthy of mentioning that if Assumption~\ref{A_2} does not hold, for example, if set $\mathcal{X}$ is mixed-integer, then the main result that $\alsox$ is better than $\CVaR$ {approximation} does not hold (see \Cref{cvar_better_nonconvex}). Thus, {the convexity assumption in Assumption~\ref{A_2}} is crucial to this key result.
	
	\subsection{Relevant Literature}
	
	Since its first appearance to tackle uncertain constraints in the decision-making problems \citep{charnes1963deterministic,charnes1958cost}, CCPs have been studied and applied in many areas. For example, \cite{pagnoncelli2009sample} considered a portfolio selection problem, where the decision-makers plan to achieve the targeted return rate with high probability. Chance constraints have also been employed to ensure a high level of service in transportation assignment problems \citep{dentcheva2000concavity} or facility location problems \citep{lejeune2016solving}. In power systems (see, e.g., \citealt{bienstock2014chance,shiina1999numerical, xie2017distributionally,zhang2016distributionally}), the decision-makers would like to restrict the probability of capacity violations of transmission lines within a small risk level. \cite{deng2016decomposition} studied a scheduling problem in healthcare, where planners want to have a low level of overtime servers. Interested readers are referred to the work by \cite{ahmed2008solving} for more CCP applications. Albeit important, a CCP encounters two main difficulties --- its feasible region is often nonconvex, and checking the feasibility of a CCP for a given solution, in general, is challenging.

	To address the aforementioned difficulties, there are several approaches proposed in the literature to solve CCP \eqref{eq_ccp}. One method is to investigate the conditions where the feasible region in CCP \eqref{eq_ccp} is convex. For example, as demonstrated in \cite{prekopa2013stochastic}, if the random vector $\tilde{\bm\xi} $ follows a log-concave probability distribution and $g(\bm x,\tilde{\bm\xi})$ is quasi-convex, the feasible region defined in CCP \eqref{eq_ccp} is convex. More convexity results can be found in {\cite{henrion2007structural,henrion2008convexity,lagoa2005probabilistically,henrion2011convexity}}. However, it may still be hard to evaluate the probabilistic constraint in CCP \eqref{eq_ccp} precisely even if it is convex. The second method is to consider approximations of chance constraints using the Monte Carlo approach, e.g., sampling average approximation (SAA) proposed by \cite{luedtke2008sample}. The advantage of SAA is to approximate a chance constraint with the one under finite support with arbitrary accuracy, and the latter can be recast as a mixed-integer program {\citep{ruszczynski2002probabilistic}}. The third method is to propose convex inner approximations of the nonconvex chance constraint (see, e.g., \citealt{nemirovski2006scenario,calafiore2006scenario,nemirovski2007convex}). The best-known convex approximation is to replace the chance constraint in CCP \eqref{eq_ccp} with the conditional value-at-risk ($\CVaR$) approximation proposed by \cite{nemirovski2007convex}. The $\CVaR$ approximation usually returns a feasible yet sub-optimal solution. Other nonlinear programming approaches have been developed recently, such as difference-of-convex functions approximation \citep{hong2011sequential}, a smooth sampling-based approximation \citep{pena2020solving}. These approaches often find stationary points of a CCP and thus are not known whether they can be more effective than $\CVaR$ approximation or not. In \citeyear{ahmed2017nonanticipative}, \textbf{A}hmed, \textbf{L}uedtke, \textbf{SO}ng, and \textbf{X}ie \citep{ahmed2017nonanticipative} proposed a heuristic scheme, called ``$\alsox$'' in this paper, which could effectively solve all of their testing instances within the 4\% optimality gap. Albeit numerically promising, its theoretical performances are not clear. This paper {fills} this gap. One {main} result in this paper is that $\alsox$ outperforms the $\CVaR$ approximation.
	
	When the distributional information is limited, as a better alternative to the conventional CCPs, distributionally robust chance constrained programs (DRCCPs) have attracted much attention (see, e.g., \citealt{hanasusanto2015distributionally,hanasusanto2017ambiguous,xie2018deterministic,zymler2013distributionally,chen2018data,xie2019distributionally}), where the latter is shown to be {effective for} decision-making under uncertainty without fully knowing the probability distribution. Interested readers are referred to the work \citep{rahimian2019distributionally} for a comprehensive review. The particular ambiguity set we focus on in this paper is {type $\infty-$}Wasserstein ambiguity set. 

	\subsection{Summary of Contributions}
	In this paper, we study and generalize $\alsox$ for solving a CCP and its distributionally robust counterpart (i.e., DRCCP). Our main contributions are summarized below.
	\begin{enumerate}[label=(\roman*)]
		\item We show that when the uncertain constraints are convex, $\alsox$ always outperforms $\CVaR$, the well-known best convex approximation, and provide sufficient conditions under which $\alsox$ can return an optimal solution to CCP \eqref{eq_ccp}.
		\item We derive an equivalent bilinear programming formulation of CCP \eqref{eq_ccp}, which inspires us to improve $\alsox$ with a convergent Alternating Minimization ($\AM$) method, termed ``$\alsox +$.'' We show that {the solution from the $\AM$ method is at least as good as that from} the difference-of-convex ($\DC$) approach.
		\item {We extend $\alsox$ and $\CVaR$ approximation to solve DRCCPs under $\infty-$Wasserstein ambiguity set, termed ``the worst-case $\alsox$,''  and ``the worst-case $\CVaR$ approximation,'' respectively. We show that under $\infty-$Wasserstein ambiguity set, the worst-case $\alsox$ outperforms the worst-case $\CVaR$ approximation.}
	\end{enumerate}
	The roadmap of contributions of our paper is shown in Figure~\ref{Roadmap}.
	\begin{figure}[htbp]
		\begin{center}
			\centering
			\includegraphics[width=0.8\textwidth]{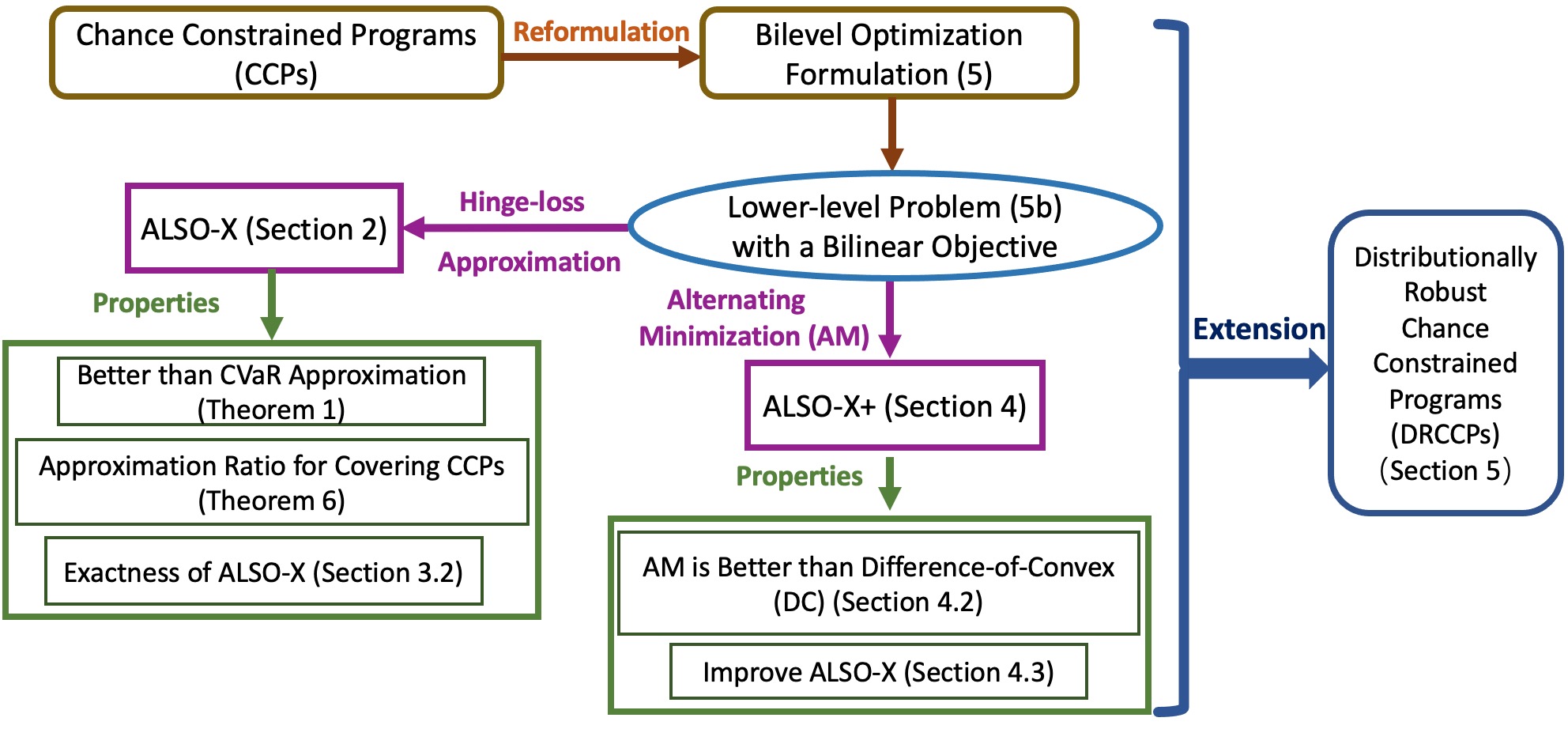}
			\caption{A Roadmap of the Main Results in This Paper.}
			\label{Roadmap}
		\end{center}
	\end{figure}

	\noindent\textbf{Organization.}
	The remainder of the paper is organized as follows. Section~\ref{alsox_property} details the properties of $\alsox$. Section~\ref{alsox_strengths} describes the strengths of $\alsox$. Section~\ref{ap} provides the formulation and properties of $\alsox +$. {Section~\ref{drccp} extends and studies $\alsox$ and $\alsox+$ to solve DRCCPs under $\infty-$Wasserstein ambiguity set.} 
	Section~\ref{numerical_study} numerically illustrates the proposed methods. Section~\ref{sec_conclusion} concludes the paper.

	\noindent\textbf{Notation.} The following notation is used throughout the paper. We use bold-letters (e.g., $\vect{x},\vect{A}$) to denote vectors and matrices and use corresponding non-bold letters to denote their components. Given a vector or matrix $\bm{x}$, its zero norm $\left \| \bm{x}\right \|_0$ denotes the number of its nonzero elements. {We let $\|\cdot\|_*$ denote the dual norm of a general norm $\|\cdot\|$.} We let $\bm{e}$ be the vector or matrix of all ones, and let $\bm{e}_i$ be the $i$th standard basis vector.
	Given an integer $n$, we let $[n]:=\{1,2,\ldots,n\}$, and use $\Re_+^n:=\{\bm {x}\in \Re^n:x_i\geq0, \forall i\in [n]\}$. Given a real number $t$, we let $(t)_+:=\max\{t,0\}$. Given a finite set $I$, we let $|I|$ denote its cardinality. We let $\tilde{\bm\xi}$ denote a random vector and denote its realizations by $\bm\xi$. Given a vector $\bm{x}\in \Re^n$, let $\mathrm{supp}(\bm{x})$ be its support, i.e., $\mathrm{supp}(\bm{x}):=\{i\in [n]: x_i\neq0\}$. Given a probability distribution $\Pr$ on $\Xi$, we use $\Pr\{A\}$ to denote $\Pr\{\bm\xi:\text{condition} \ A(\bm\xi) \ \text{holds}\}$ when $A(\bm\xi)$ is a condition on $\bm\xi$, and to denote $\Pr\{\bm\xi\colon \bm\xi \in A\}$ when $A \subseteq \Xi$ is $\Pr-$measurable. {We follow the convention of CCP literature (see, e.g., \citealt{ruszczynski2002probabilistic,nemirovski2007convex,luedtke2008sample}) for the definition of an indicator function, that is, given a set $R$,
		its normal cone at point $\x\in R$ is denoted {by} $\N_R(\x):=\{\bm{h}: \bm{h}^\top(\hat{\x}-\x)\leq 0, \forall \hat{\x}\in R\}$, {and $\emptyset$ if $\x\notin R$;} and the indicator function $\I(\bm x \in R) =1$ if $\bm x \in R$, and $0$, otherwise.} We use $\floor{x}$ to denote the largest integer $y$ satisfying $y\leq x$, for any $x\in \Re$. We use the phrase ``Better Than" to indicate ``at least as good as."	Additional notation will be introduced as needed.
	
	\section{Developments and Properties of $\alsoxt$ }
	\label{alsox_property}
	In this section, we present equivalent formulations of CCP \eqref{eq_ccp}, derive its hinge-loss approximation, show its connection to the $\alsox$, and we also derive two special cases of $\alsox$. 
	\subsection{Equivalent Formulations}
	\label{bp}
	The fact that $\Pr\{\tilde{\bm\xi} \colon g(\bm x,\tilde{\bm\xi})\leq 0\} = \E_\Pr[\I(g(\bm x,\tilde{\bm\xi})\leq 0)]$ inspires us to introduce a binary functional variable {$z(\cdot):\Xi\to\Re $ defined on $(\Omega,\F,\Pr)$}  
	to represent the indicator function {$\I(\cdot)$}. Thus, we have the following equivalent formulation of CCP \eqref{eq_ccp},
	\begin{align}
	\label{eq_bilinear_prop}
	v^*=\min _{\bm {x}\in \mathcal{X},{z({\cdot})}}\left\{ \bm{c}^\top \bm{x}\colon \I( g(\bm x,\tilde{\bm\xi})\leq 0)\geq z(\tilde{\rxi}), \E[{z}(\tilde{\rxi})]\geq 1-\varepsilon, z(\tilde{\rxi})\in\{ 0,1\} \right\},
	\end{align}
	where for the sake of simplicity, we suppose that all the random constraints are satisfied almost surely throughout the paper. 
	Next, we observe that CCP \eqref{eq_bilinear_prop} can be equivalently reformulated as a bilinear constrained program by introducing another auxiliary nonnegative functional variable {$s(\cdot):\Xi\to\Re $ defined on $(\Omega,\F,\Pr)$} 
	to denote uncertain constraint violations and {subsequently} replacing random constraint $\I( g(\bm x,\trxi)\leq0)\geq z(\trxi)$ by $\E[{z}(\trxi){s}(\trxi)]=0$. 
	
	\begin{restatable}{proposition}{alsoxformulation}\label{also_+_formulation_1} 
		The CCP \eqref{eq_ccp} can be viewed as the following equivalent form 
		\begin{align}
		\label{eq_bilinear}
		v^*=\min _{\bm {x}\in \mathcal{X}, {{z}(\cdot), {s}(\cdot)}} \left\{ \bm{c}^\top \bm{x}\colon g(\bm x,\tilde{\bm {\xi}})\leq s(\tilde{\rxi}), \E[{z}(\tilde{\rxi})]\geq 1-\varepsilon,{z}(\tilde{\rxi})\in [0,1], \E[{z}(\tilde{\rxi}){s}(\tilde{\rxi})]=0,{s}(\tilde{\rxi})\geq 0 \right\}.
		\end{align}
	\end{restatable}
	\proof
	See Appendix~\ref{proof_also_+_formulation_1}.
	\QEDA
	
	We remark that in the bilinear formulation \eqref{eq_bilinear}, the functional variable {$z(\cdot) $} can be either binary or continuous, and this property {is} useful for deriving $\alsox+$ in \Cref{ap}.
	
	Replacing the objective function with an
	auxiliary variable $t$, we can rewrite CCP \eqref{eq_bilinear} as 
	\begin{align}
	\label{alsox_aug}
	v^* =\min _{ \begin{subarray}{c}
		\bm {x}\in \mathcal{X},t,\\
		{{z}({\cdot}),s({\cdot})}
		\end{subarray} } \left\{t\colon g(\bm x,\tilde{\bm {\xi}})\leq s(\tilde{\rxi}),\E[{z}(\tilde{\rxi})]\geq 1-\varepsilon, \E[{z}(\tilde{\rxi}){s}(\tilde{\rxi})]=0,{z}(\tilde{\rxi})\in [0,1],{s}(\tilde{\rxi}) \geq 0,\bm{c}^\top \bm{x} \leq t\right\}.
	\end{align}
	Formulation \eqref{alsox_aug} implies that in a CCP, one can search the smallest possible $t$ such that the constraint system remains feasible. Thus, this motivates us to convert CCP \eqref{alsox_aug} into an equivalent simple bilevel optimization problem.
	\begin{restatable}{proposition}{alsoxformulationbi}\label{also_+_formulation_2} 
		CCP \eqref{alsox_aug} is equivalent to
		\begin{subequations}
			\label{alsox_bilnear}
			\begin{align}
			v^* =\min _{ {t}} \ & t,\label{alsox_bilneara}\\
			\text{s.t.}\ &\left(\bm x^*,{{s}^*(\cdot),{z}^*(\cdot)}\right)\in\argmin _{ \begin{subarray}{c}
				\bm {x}\in \mathcal{X}, \\
				{{z}({\cdot})}\in [0,1],\\
				{{s}({\cdot})} \geq 0
				\end{subarray} }\left\{\E[{z}(\tilde{\rxi}){s}(\tilde{\rxi})]\colon g(\bm x,\tilde{\bm {\xi}})\leq s(\tilde{\rxi}),\E[{z}(\tilde{\rxi})]\geq 1-\varepsilon, \bm{c}^\top \bm{x} \leq t\right\},\label{alsox_bilnearb}\\
			&\Pr\left\{\tilde{\rxi}\colon g(\bm x^*,\tilde{\bm\xi})\leq 0\right\} \geq 1-\varepsilon.\label{alsox_bilnearc} 
			\end{align}
		\end{subequations}
	\end{restatable}
	\proof
	See Appendix~\ref{proof_also_+_formulation_2}.
	\QEDA
	
	In the bilevel optimization Formulation \eqref{alsox_bilnear}, the problem defined in \eqref{alsox_bilneara} and \eqref{alsox_bilnearc} is known as an upper-level (or leader's) problem with decision variable $t$, and the one appearing in \eqref{alsox_bilnearb} can be regarded as the lower-level (or follower's) problem with decision variables $\bm{x}$, {${s}(\cdot)$}, and  {${z}(\cdot)$}. Given the value of upper-level decision $t$, we can solve the lower-level problem \eqref{alsox_bilnearb} and then check whether the solution satisfies condition \eqref{alsox_bilnearc} or not --- if the answer is YES, we can reduce the value of $t$; otherwise, we have to increase $t$. 
	Besides, the optimal value of the lower-level problem \eqref{alsox_bilnearb} is monotone nonincreasing with respect to $t$. Therefore, if the lower-level problem \eqref{alsox_bilnearb} were easy to solve, it would be trivial to find the optimal value of CCP \eqref{alsox_bilnear}. For example, the binary search can be used to find the optimal upper-level decision $t^*$. However, {for a given $t$, the lower-level problem \eqref{alsox_bilnearb} is
		a two-stage bilinear program, which is known to be challenging to solve. In fact, solving the original CCP \eqref{eq_ccp} is polynomial-time reducible to problem \eqref{alsox_bilnearb} since the binary search can take a polynomial number of iterations in the size of the problem.} Note that solving a CCP is NP-hard as shown in \cite{luedtke2010integer}, so is the lower-level problem \eqref{alsox_bilnearb}.
	
	\subsection{What is $\alsoxt$?}\label{sec_alsox_sub}
	As discussed in the previous subsection, the lower-level problem \eqref{alsox_bilnearb} can be challenging to solve. Thus, in this subsection, we study its hinge-loss approximation, which is much easier to handle 
	in many cases. 
	Particularly,  {instead of solving a difficult bilinear program, we adopt a simple strategy by letting} the functional variable  ${z}(\rxi)=1$ in the lower-level problem \eqref{alsox_bilnearb}, i.e., the following form
	\begin{align*}
	&v^A(t):=\min _{\bm {x}\in \mathcal{X},{ {s}({\cdot})} \geq 0}\left\{\E[{s}(\tilde{\rxi})]\colon g(\bm x,\tilde{\bm {\xi}})\leq s({\tilde{\bxi}}),\bm{c}^\top \bm{x} \leq t\right\}.
	\end{align*}
	Next, projecting out continuous variable {${s}(\cdot)$}, we arrive at the following hinge-loss approximation
	\begin{align}
	\label{eq_project_out}
	v^A(t)=\min_{\bm {x}\in \mathcal{X}} & \left\{\E[g(\bm x,\tilde{\bm {\xi}})_+]\colon \bm{c}^\top \bm{x} \leq t \right\}.
	\end{align}
	The objective {function} in \eqref{eq_project_out}, {termed ``the hinge-loss function,''} can be viewed as an expectation of the nonnegative part of the random function $g(\bm x,\tilde{\bm {\xi}})$,  {which has been} widely {applied} in machine learning methods such as SVM (see, e.g., \citealt{suykens1999least}) and LASSO (see, e.g., \citealt{tibshirani1996regression}).  The goal of the hinge-loss approximation \eqref{eq_project_out} is to minimize the expectation of infeasibilities given the upper-level problem's decision $t$. 
	
	The proposed $\alsox$ is to replace the lower-level problem \eqref{alsox_bilnearb} by the hinge-loss approximation \eqref{eq_project_out}, which
	admits the following form 
	\begin{subequations}
		\label{eq_also_x_2}
		\begin{align}
		v^A =\min _{ {t}}\quad & t,\label{eq_also_x_2a}\\
		\text{s.t.}\quad&\left(\bm x^*,{{s}^*(\cdot)}\right)\in\argmin _{\bm {x}\in \mathcal{X}, {{s}({\cdot})} \geq 0}\left\{\E[{s}(\tilde{\rxi})]\colon g(\bm x,\tilde{\bm {\xi}})\leq s(\tilde{\rxi}),\bm{c}^\top \bm{x} \leq t\right\},\label{eq_also_x_2b}\\
		& \Pr\left\{\tilde{\bm\xi} : s^*(\tilde{\rxi})=0\right\} \geq 1-\varepsilon.\label{eq_also_x_2c} 
		\end{align}
	\end{subequations}
	Letting $(\bm x^*,{s^*(\cdot)})$ denote an optimal solution of the hinge-loss approximation \eqref{eq_also_x_2b}, if its optimal solution $\bm x^*$ is feasible to CCP \eqref{eq_ccp}, then we have $\Pr\{\tilde{\bm\xi} : g(\bm x^*,\tilde{\bm\xi})\leq 0 \}\geq1-\varepsilon$, which is equivalent to
	$\Pr\{\tilde{\bm\xi} : s^*(\tilde{\rxi})=0\} \geq1-\varepsilon$. {Increasing $t$ in \eqref{eq_also_x_2b} would drive down the optimal $s^*(\cdot)$ and hence make \eqref{eq_also_x_2c} more likely to be satisfied.} This ensures that 
	\begin{proposition}
		\label{alsox_feasible}
		{The proposed $\alsox$ \eqref{eq_also_x_2} is a convex approximation of CCP \eqref{eq_ccp}}, i.e., $v^A\geq v^*$. 
	\end{proposition}
	{In general, it is difficult to quantify the difference between $v^A$ and $ v^*$, and similar to many conservative approximation methods, under some extreme cases, $\alsox$ \eqref{eq_also_x_2} may be infeasible (please see \Cref{alsox_infeasible_example} in Appendix~\ref{sec_append_example2}).} {However, in the next section, we {show} that for the finite-support covering CCPs, $v^A$ is within a factor (greater than one) of $v^*$ and this factor is tight.}  
	
	Correspondingly, we develop the $\alsox$ \Cref{alg_alsox}, which generalizes the heuristic algorithm in section 6.1 \citep{ahmed2017nonanticipative}.
	Particularly, for a given value $t$ of the upper-level problem, we solve the hinge-loss approximation \eqref{eq_also_x_2b} with an optimal solution $(\bm x^*,{s^*(\cdot)})$ and check whether $\bm x^*$ is feasible to CCP \eqref{eq_ccp} or not, i.e., check if $\Pr\{\tilde{\bm\xi} \colon s^*(\tilde{\rxi})=0\} \geq1-\varepsilon$ or not.
	If the answer is YES, we {reduce} the value of $t$. Otherwise, increase it. In the implementation, we search the optimal $t$ by using the binary search method {with a proper stopping tolerance $\delta_1$ (e.g., we choose $\delta_1=10^{-2}$ in numerical study)}, which is detailed in \Cref{alg_alsox}.
	
	\begin{algorithm}[htbp]
		\caption{The Proposed $\alsox$ Algorithm}
		\label{alg_alsox}
		\begin{algorithmic}[1]
			\State \textbf{Input:} Let $\delta_1$ denote the stopping tolerance parameter, $t_L$ and $t_U$ be the known lower and upper bounds of the optimal value of CCP \eqref{eq_ccp}, respectively 
			\While {$t_U-t_L>\delta_1$}
			\State Let $t=(t_L+t_U)/2$ and {$(\bm{x}^*,{s}^*(\cdot))$} be an optimal solution of the hinge-loss approximation \eqref{eq_also_x_2b}
			\State Let $t_L=t$ if $\Pr\{\tilde{\bm\xi} \colon{{s}^*}(\tilde{\rxi})>0\} >\varepsilon$; otherwise, $t_U=t$
			\EndWhile
			\State \textbf{Output:} A feasible solution $\x^*$ {and its objective value $\bar v^A$} to CCP \eqref{eq_ccp} 
		\end{algorithmic}
	\end{algorithm}
	
	We make the following remarks about $\alsox$ \Cref{alg_alsox}.
	\begin{enumerate} [label=(\roman*)]
		\item   At Step 4 in  \Cref{alg_alsox}, we let $t_L=t$ if the solution $(\bm{x}^*,{s}^*(\cdot))$ is infeasible to the CCP \eqref{eq_ccp}, i.e., $\Pr\{\tilde{\bm\xi} \colon{{s}^*}(\tilde{\rxi})>0\} >\varepsilon$; otherwise, we decrease the current upper bound of $t$ by letting $t_U=t$;
		\item For the linear CCPs, i.e., $g_i(\bm x,{\rxi}) ={\rxi}^\top\bm a_i(\bm x) -b_i(\bm x)$ with affine functions $\bm a_i(\bm x),b_i(\bm x)$ and set $\mathcal{X}$ is a polyhedron, we can { use  parametric} linear programming techniques (see, e.g., \citealt{adler1992geometric}) to decrease the number of bisection needed. That is, at each iteration, if the current solution is not feasible to CCP \eqref{eq_ccp}, then we update $t$ to be the upper bound of its allowable range; otherwise, we let $t$ be equal to the lower bound of the allowable range. Then, we continue the binary search procedure;
		\item {Note that if the probability distribution $\Pr$ is finite-support or the probability distribution is elliptical and CCP \eqref{eq_ccp} solely involves a single linear chance constraint, then the hinge-loss approximation \eqref{eq_also_x_2b} can be efficiently solvable under mild conditions (see \Cref{polynomial_support_tract} and \Cref{el_tract} in Appendix~\ref{alsox_tractable}).} Otherwise, one can first use the sampling average approximation results from \cite{luedtke2008sample} to tightly approximate a CCP to the one with finite-support and then solve the $\alsox$ \eqref{eq_also_x_2};
		\item For nonlinear chance constraints, we can use subgradient descent algorithm to solve the hinge-loss approximation \eqref{eq_project_out}, which {is} detailed in the next subsection;
		\item At each iteration, we can warm-start the process with the solution found in the previous iteration; 
		\item High-quality upper and lower bounds can help reduce the number of iterations needed. For instance, we can use the quantile bound proposed in \cite{ahmed2017nonanticipative,song2014chance} as a promising lower bound $t_L$, and use the objective value from $\CVaR$ approximation or other heuristics as a potential upper bound $t_U$; and
		\item {	As long as the encoding length of $t_U$, $t_L$, and $\delta_1$ are polynomial in the input size of CCP \eqref{eq_ccp}, the number of bisection iterations needed in $\alsox$ \Cref{alg_alsox} is proportional to $\log((t_U-t_L)/\delta_1)$, i.e., polynomial in the input size of CCP \eqref{eq_ccp} as well.}
	\end{enumerate}
	
	\subsection{Subgradient Descent (SD) Algorithm }
	\label{sdm}
	Note that the hinge-loss approximation \eqref{eq_project_out} is a convex minimization problem and can be efficiently solved by the Subgradient Descent (SD) algorithm if the underlying distribution is finite-support. For ease of notation, we first define the feasibility set of $\bm{x}$ as $\mathcal{S}:=\{ \bm{x}\colon\bm {x}\in \mathcal{X} \cap \bm{c}^\top \bm{x} \leq t \}$. The SD method proceeds as follows: (i) Given a solution $\bm x\in \mathcal{S}$, we first find its subgradient; and (ii) then project the difference of this solution and scaled subgradient descent direction into set $\mathcal{S}$, which generates a new solution. We continue this process until invoking a stopping criterion. The detailed SD method for solving the hinge-loss approximation \eqref{eq_project_out} can be found in \Cref{alg_SD}.
	
	\begin{algorithm}[htbp]
		\caption{Subgradient Descent (SD) Algorithm to Solve the Hinge-loss Approximation \eqref{eq_project_out}}
		\label{alg_SD}
		\begin{algorithmic}[1]
			\State 	Find an initial feasible solution $\bm x_0$ such that $\bm x_0 \in \mathcal{S}: = \{ \bm{x}\colon\bm {x}\in \mathcal{X} \cap \bm{c}^\top \bm{x} \leq t \}$ {and let $k=0$}
			\Do
			\State {At iteration $k$, compute subgradient $\hat h_{\bm x_k}$ for $\bm{x}_{k}$, i.e., $\hat h_{\bm x_k} = \partial _{\bm x_k}\E[g(\bm x_k,\tilde{\bm {\xi}})_+]$}
			\State {Update $\bm{x}_{k+1}$ by $\mathrm{\Pi}_{\mathcal{S}}(\bm{x}_{k}-\gamma_k \hat h_{\bm x_k})$, where $\gamma_k$ is the step size}
			\State {$k=k+1$}
			\doWhile{{Invoking} a stopping criterion} 
		\end{algorithmic}
	\end{algorithm}
	
	We make the following remarks about SD \Cref{alg_SD}.
	\begin{enumerate} [label=(\roman*)]
		\item According to Assumption \ref{A_1} and theorem 1 in \cite{rockafellar1982interchange}, 
		we can interchange the subdifferential operator and expectation when updating $\bm x_k$ at $k$th iteration. That is, { $\bm{x}_{k+1}= \mathrm{\Pi}_{\mathcal{S}}(\bm{x}_{k}-\gamma_k \E[\partial_{\bm x_k}g(\bm x_k,\tilde{\bm {\xi}})_+])$};
		\item The proposed \Cref{alg_SD} works well when the underlying probability is finite-support and the random function $g_i(\bm x,\tilde{\bm {\xi}})$ is convex and nonlinear. For the continuous probability distribution $\Pr$, one can approximate it with the one of finite support, according to the sampling average approximation results in \cite{luedtke2008sample}, or one can use the stochastic subgradient descent method \citep{Nemirovski2009robust};  
		\item At {Step 4} in \Cref{alg_SD}, the projection onto the feasible set $\mathcal{S}$ can be solved efficiently by adopting Dykstra's projection algorithm \citep{boyle1986method}, i.e., projecting onto set $\mathcal{X}$ and the set induced by the knapsack constraint (i.e., set $\{\bm x\in \Re^n: \bm{c}^\top \bm x\leq t\}$) alternatively \citep{tavakoli2016coupled}; and 
		{	\item In the numerical study, we select the step size at iteration $k$ as $\gamma_k=1/(k+1)$, and its corresponding convergence rate is $\O(1/\mathrm{\log}(T))$. Other step size choices and convergence rate results can be found in chapter 3 \citep{nesterov2003introductory}.}
	\end{enumerate}
	
	
	\subsection{Special Cases}
	\label{alsox_sp}
	In this subsection, we discuss two special cases under which the hinge-loss approximation \eqref{eq_project_out} can be computed efficiently: (i) under discrete support; or (ii) if $I=1$, and function $g(\bm x,{\rxi})$ is biaffine in $\bm x$ and ${\rxi}$ and the random vector $\tilde{\rxi}$ follows an elliptical distribution.
	
	\noindent\textbf{Special Case 1: Discrete Support}\\
	If the underlying probability distribution is finite-support with $N$ equiprobable scenarios, i.e., the random vector $\tilde{\rxi}$ has a finite support $\Xi = \{\bm \xi^1,\dots, \bm \xi^N\}$ with $ \Pr\{\tilde{\rxi}=\bm\xi^i\}=1/N$ for all $i\in [N]$, then CCP \eqref{eq_ccp} reduces to 
	\begin{align*}
	v^* = \min_{\bm {x}\in \mathcal{X}}\left\{\bm{c}^\top \bm{x} \colon \sum_{i\in [N]}\I(g(\bm x,\bm\xi^i)\leq 0) \geq N-\floor{N\varepsilon}\right\},
	\end{align*}
	and $\alsox$ \eqref{eq_also_x_2} admits the following form
	\begin{align}
	v^A =\min _{ {t}}\quad & t,\nonumber\\
	\text{s.t.}\quad&(\bm x^*,\bm s^*)\in\argmin_{\bm {x}\in \mathcal{X},\bm{s}\geq \bm{0}}\left\{\frac{1}{N} \sum_{i\in[N]}s_i\colon \bm{c}^\top \bm{x} \leq t, g(\bm x,\bm\xi^i)\leq s_i, \forall i\in [N] \right\},\label{eq_also_x_discrete}\\
	& \sum_{i\in[N]}	\I({s^*_i}=0) \geq N-\floor{N\varepsilon}.\nonumber
	\end{align}
	Note that for this special case, the condition that $\Pr\{\tilde{\bm\xi} \colon {s}^*(\tilde{\rxi})=0\} \geq 1-\varepsilon$ in $\alsox$ \eqref{eq_also_x_2c} reduces to that $\sum_{i\in[N]}	\I({s^*_i}=0) \geq N-\floor{N\varepsilon}$, where the left-hand side of the inequality is equal to the support size of $\bm{s}^*$ (i.e., $|\mathrm{supp}(\bm{s}^*)|$). Thus, if the support size of current solution $\bm{s}^*$, or equivalently, the number of violations is strictly larger than $\floor{N\varepsilon}$, then the current solution $\bm x^*$ is not feasible to CCP \eqref{eq_ccp}.  {The $\alsox$ \eqref{eq_also_x_discrete} can be generalized to the case when the probability mass is not uniform (see \Cref{alsox_discrete_tractable} in Appendix~\ref{alsox_tractable} for the detailed formulation). 
	}


	\noindent\textbf{Special Case 2: Elliptical Distributions}\\
	{Elliptical distributions have been widely used in risk management (see, e.g., \citealt{landsman2003tail, jaworski2010copula,embrechts2002correlation,kamdem2005value}). For instance, \citealt{kamdem2005value} used elliptical distributions to model portfolio risk factors.  An elliptical distribution $\Pr_{\mathrm{E}}(\bm{\mu},\bm{\mathrm{{\Sigma}}},\hat g)$ is described by three parameters, a location parameter $\bm{\mu}$, a positive definite matrix $\bm{\mathrm{{\Sigma}}}$, and a generating function $\hat g$. The name of an elliptical distribution is based on the fact that the contours of its density are ellipsoids in $\bm x\in\Re^n$, and therefore, its probability density function $\hat{f}$ is defined as 
		\begin{equation*}
		\hat{f}(\bm{x})=k\cdot \hat g\left(\frac{1}{2}(\bm{x}-\bm{\mu})^\top\bm{\mathrm{{\Sigma}}}^{-1}(\bm{x}-\bm{\mu})\right) 
		\end{equation*}
		with a positive normalization scalar $k$.\\
		The probability density function of the standard univariate elliptical distribution $\Pr_{\mathrm{E}}(0,1,\hat g)$ is $\varphi(z)=k\hat g(z^2/2)$, and the corresponding cumulative distribution function is $\mathrm{\Phi}(a)=\int_{-\infty}^ak\hat g(z^2/2)dz$. For the single linear CCP (1), i.e., $I = 1$ and $g(\bm x,{{\rxi}}) =  \rxi^\top \bm a_1(\bm x)-b_1(\bm x)$ with affine functions $\bm a_1(\bm x)$, $ b_1(\bm x)$, if the random parameters $\tilde{\bm {\xi}}$ follow a joint elliptical distribution with $\tilde{\bm {\xi}} \thicksim \Pr_{\mathrm{E}}(\bm{\mu},\bm{\mathrm{{\Sigma}}},\hat g)$, the objective function in the hinge-loss approximation \eqref{eq_project_out} can be much simplified.}
	In this special case, CCP \eqref{eq_ccp} reduces to the following conic program {(see, e.g., problem 2.1 in \citealt{kataoka1963stochastic} or theorem 3 in \citealt{prekopa1974programming})}
	\begin{align}
	v^*=	\min_{\bm{x}\in \mathcal{X}} \left\{\bm{c}^\top \bm{x} \colon b_1(\bm x)-\bm{\mu}^\top\bm{a}_1(\bm {x})\geq \mathrm{\Phi}^{-1}(1-\varepsilon) \sqrt{\bm{a}_1(\bm{x})^\top\bm{\mathrm{{\Sigma}}}\bm{a}_1(\bm {x})}\right\}. \label{ccp_el}
	\end{align}
	This notable simplification \eqref{ccp_el} {is useful to show the exactness of $\alsox$ \eqref{eq_also_x_2}.} The following proposition shows an equivalent reformulation of $\alsox$ \eqref{eq_also_x_2} .
	\begin{restatable}{proposition}{hingelosstruncatedel}\label{hinge_loss_truncated_el} 
		For any elliptical distribution {$\Pr_{\mathrm{E}}(\bm{\mu},\bm{\mathrm{{\Sigma}}},\hat g)$}, $\alsox$ \eqref{eq_also_x_2} corresponding to the single linear CCP admits the following form
		\begin{subequations}
			\label{eq_alsox_el}
			\begin{align}
			v^A =& \min _{ {t}}\quad t,\label{eq_alsox_ela}\\
			&	\begin{aligned}
			\text{s.t.}\quad(\bm x^*,\alpha^*)\in\argmin_{\bm{x}\in \mathcal{X},\alpha } \biggl\{ & \sqrt{\bm{a}_1(\bm{x})^\top\bm{\mathrm{{\Sigma}}}\bm{a}_1(\bm{x})} \left(\overline{G}(\alpha^2/2)-\alpha+\alpha\mathrm{\Phi}(\alpha) \right)\colon \\
			&\bm{c}^\top \bm{x} \leq t, \frac{b_1(\bm x)-\bm\mu^\top \bm{a}_1(\bm {x})}{\sqrt{\bm{a}_1(\bm{x})^\top\bm{\mathrm{{\Sigma}}}\bm{a}_1(\bm {x})}} =\alpha \biggr\},\label{eq_alsox_elb}
			\end{aligned}\\
			&\ \quad b_1(\bm x^*)-\bm{\mu}^\top\bm{a}_1(\bm {x}^*)\geq \mathrm{\Phi}^{-1}(1-\varepsilon) \sqrt{\bm{a}_1(\bm{x}^*)^\top\bm{\mathrm{{\Sigma}}}\bm{a}_1(\bm {x}^*)},\label{eq_alsox_elc} 
			\end{align}
		\end{subequations}
		where  {$\overline{G}(\tau)=G(\infty)-G(\tau)$ and $G(\tau)=k\int_0^\tau\hat g(z)dz$}. By default, we let $\frac{0}{0}=0$ and $\frac{c}{0}=\sign(c)\infty$ if $c\neq 0$.
	\end{restatable}
	\proof
	See Appendix~\ref{proof_hinge_loss_truncated_el}.
	\QEDA
	
	Note that one might need to project out variable $\alpha$ to ensure the convexity of the objective function in the lower-level problem \eqref{eq_alsox_elb}. The purpose of introducing variable $\alpha$ is to simplify the formula and is convenient to prove the monotonicity of the objective function, which {is} elaborated in the next {section}.
	
	The following result shows under a Gaussian distribution, the objective function of the hinge-loss approximation \eqref{eq_project_out} can be further simplified according to Proposition \ref{hinge_loss_truncated_el}.
	\begin{corollary}\label{cor_truncated_nor}
		When $\tilde{\bm {\xi}}$ follows Gaussian distribution (i.e., a special elliptical distribution with {$\hat g(\mu)=e^{-\mu}$} and $k=(\sqrt{2\pi})^{-1}$), the hinge-loss approximation \eqref{eq_alsox_elb} corresponding to the single linear CCP reduces to
		\begin{equation}
		\label{eq_alsox_normal}
		v^A(t) =\min_{ \bm x\in \mathcal{X},\alpha } \left\{ \sqrt{\bm{a}_1(\bm{x})^\top\bm{\mathrm{{\Sigma}}}\bm{a}_1(\bm{x})} \left(\varphi(\alpha)-\alpha+\alpha\mathrm{\Phi}(\alpha) \right) \colon \bm{c}^\top \bm{x} \leq t,\frac{b_1(\bm x)-\bm\mu^\top \bm{a}_1(\bm {x})}{\sqrt{\bm{a}_1(\bm{x})^\top\bm{\mathrm{{\Sigma}}}\bm{a}_1(\bm {x})}} =\alpha \right\}.
		\end{equation} 
	\end{corollary}
	Finally, we re-stress that \Cref{hinge_loss_truncated_el} and \Cref{cor_truncated_nor} {are} useful to derive the exactness of $\alsox$.
	
	\section{Strengths of $\alsoxt$ }
	\label{alsox_strengths}
	In this section, we present three strengths of $\alsox$. First, we demonstrate $\alsox$ \eqref{eq_also_x_2} always outperforms $\CVaR$ approximation under Assumptions~\ref{A_1} and \ref{A_2}. Next, we show sufficient conditions under which $\alsox$ \eqref{eq_also_x_2} returns an optimal solution to CCP \eqref{eq_ccp}. Finally, we provide a provable performance guarantee when applying the proposed $\alsox$ \eqref{eq_also_x_2} to solve the covering CCPs.
	
	\subsection{A Comparison between $\alsoxt$ and $\CVaRt$ approximation}
	\label{alsox_is_better}
	Given a random variable $\tilde{\bm{X}}$, let $\Pr$ and $F_{\tilde{\bm{X}}}(\cdot)$ be its probability distribution and cumulative distribution function, respectively. For a given risk level $\varepsilon$, $(1-\varepsilon)-$Value at risk ($\VaR$) of $\tilde{\bm{X}}$ is 
	\begin{align*}
	\VaR_{1-\varepsilon}(\tilde{\bm{X}}):=\min_s \left\{ s: F_{\tilde{\bm{X}}}(s)\geq 1-\varepsilon\right\},
	\end{align*}
	and the corresponding conditional value-at-risk ($\CVaR$) is defined as 
	\begin{align*}
	\CVaR_{1-\varepsilon}(\tilde{\bm{X}}):=\min_\beta \left\{ \beta+\frac{1}{\varepsilon}\E_{\Pr}[ \tilde{\bm{X}}-\beta ]_+\right\}.
	\end{align*}
	According to \cite{nemirovski2007convex}, the $\CVaR$ approximation of CCP \eqref{eq_ccp} can be written as
	\begin{subequations}
		\begin{equation}
		\label{ccp_cvar}
		v^{\CVaR}=\min_{\bm x\in \mathcal{X}} \left\{ \bm{c}^\top\bm{x}\colon \min_{\beta\leq 0}\left\{\beta+\frac{1}{\varepsilon} \E\{ g(\bm x,\tilde{\rxi})-\beta\}_+\right\}\leq 0 \right\}.
		\end{equation}
		Letting $s({\rxi}):=\max\{g(\bm x,{\bm {\xi}}),\beta\}$ and linearizing it, the $\CVaR$ approximation \eqref{ccp_cvar} is equivalent to 
		\begin{align}
		\label{eq_cvar}
		v^{\CVaR}=\min_{{\bm x\in \mathcal{X}},\beta\leq 0,{s(\cdot)} } \left\{ \bm{c}^\top \bm{x} \colon g(\bm x,\tilde{\bm {\xi}})\leq s(\tilde{\rxi}), \E[{s}(\tilde{\rxi})]-(1-\varepsilon)\beta \leq 0, s(\tilde{\rxi})\geq \beta\right\}.
		\end{align}
		Next, augmenting the objective function by adding an auxiliary variable $t$,	the $\CVaR$ approximation \eqref{eq_cvar} can be further formulated as
		\begin{align}
		\label{eq_cvar_eq}
		v^{\CVaR}=\min_{{\bm x\in \mathcal{X}},\beta\leq 0,{s(\cdot)} ,t} \left\{ t \colon g(\bm x,\tilde{\bm {\xi}})\leq s(\tilde{\rxi}), \E[{s}(\tilde{\rxi})]-(1-\varepsilon)\beta \leq 0, s(\tilde{\rxi})\geq \beta, \bm{c}^\top \bm{x}\leq t\right\}.
		\end{align}
	\end{subequations}
	Similar to $\alsox$ \eqref{eq_also_x_2}, in the $\CVaR$ approximation \eqref{eq_cvar_eq}, one can search the smallest possible $t$ such that the constraint $\E[s(\tilde{\rxi})]-(1-\varepsilon)\beta$ remains feasible. Thus, the $\CVaR$ approximation \eqref{eq_cvar_eq} can be rewritten as the following simple bilevel program
	\begin{subequations}
		\label{cvar_bilevel}
		\begin{align}
		v^{\CVaR} & =\min _{ {t}}\ t,\label{cvar_bilevela}\\
		\text{s.t.}\quad&\left(\bm x^*, {s^*(\cdot)},\beta^*\right)\in\argmin _{\bm {x}\in \mathcal{X},\beta\leq 0, {s(\cdot)} }\left\{\E[{s}(\tilde{\rxi})] - (1-\varepsilon)\beta\colon g(\bm x,\tilde{\bm {\xi}})\leq s(\tilde{\rxi}),s(\tilde{\rxi})\geq \beta, \bm{c}^\top \bm{x} \leq t\right\},\label{eq_cvar_t}\\
		&\E[{s}^*(\trxi)]-(1-\varepsilon)\beta^*\leq 0.\label{cvar_bilevelc} 
		\end{align}
	\end{subequations}
	Above, for any given $t$, let $ ({\bm{x}^*},{{s}^*(\cdot)},\beta^*)$ be an optimal solution of the lower-level problem \eqref{eq_cvar_t} with an optimal objective value $	v^{\CVaR}(t)$, { where
		\begin{align*}
		v^{\CVaR}(t) = \min_{\bm {x}\in \mathcal{X},\beta\leq 0, s(\cdot) }\left\{\E[{s}(\tilde{\rxi})] - (1-\varepsilon)\beta\colon g(\bm x,\tilde{\bm {\xi}})\leq s(\tilde{\rxi}),s(\tilde{\rxi})\geq \beta, \bm{c}^\top \bm{x} \leq t\right\}.
		\end{align*}
		If $v^{\CVaR}(t)>0$, then constraint \eqref{cvar_bilevelc} would be violated and $ ({\bm{x}^*},{{s}^*(\cdot)},\beta^*, t)$ would be infeasible to the $\CVaR$ approximation \eqref{eq_cvar_eq}. As a result, we must have $t< v^\CVaR$. Otherwise, we would have $t\geq  v^\CVaR$.}
	
	Notice that letting the variable $\beta=0$ in the lower-level problem \eqref{eq_cvar_t} recovers the hinge-loss approximation \eqref{eq_also_x_2b}. This observation motivates us to compare $\CVaR$ approximation \eqref{cvar_bilevel} and $\alsox$ \eqref{eq_also_x_2}; namely, for a given $t$, if an optimal solution of the lower-level problem \eqref{eq_also_x_2b} violates the chance constraint \eqref{eq_also_x_2c}, so does the $\CVaR$ approximation. In fact, we can prove that the optimal values of both lower-level problems coincide under this premise. Therefore, the $\alsox$ \eqref{eq_also_x_2} outperforms the $\CVaR$ approximation \eqref{cvar_bilevel}, since the feasibility-checking condition (i.e., constraint \eqref{cvar_bilevelc}) of the upper-level problem in the $\CVaR$ approximation \eqref{cvar_bilevel} is more restrictive.
	
	\begin{restatable}{theorem}{alsoxbettercvarccp}\label{the1} 
		Let $v^A$, $v^{\CVaR}$ denote the optimal value of the $\alsox$ \eqref{eq_also_x_2} and the $\CVaR$ approximation \eqref{cvar_bilevel}, respectively. Then, under Assumptions \ref{A_1}-\ref{A_2}, we must have $v^A\leq v^{\CVaR}$.
	\end{restatable}
	\proof
	See Appendix~\ref{proof_the1}.
	\QEDA
	
	Interested readers are referred to \Cref{example_alsox_cvar} in Appendix~\ref{sec_append_example} for an illustration of the correctness of \Cref{the1}.
	
	{We note that according to the stopping criterion in  \Cref{alg_alsox}, the output objective value $\bar{v}^{A}$ might not be equal to $v^A$ and at most $\delta_1$ larger than $v^A$, the optimal value of $\alsox$. According to \Cref{the1}, we must have $v^A\leq v^\CVaR$. Thus, the output objective value $\bar v^A$ from  \Cref{alg_alsox} is no larger than $v^\CVaR+\delta_1$. This result is summarized below.}
	\begin{corollary}\label{cor_alsox_algorithm_alsox}
		Under Assumptions \ref{A_1}-\ref{A_2}, $\alsox$ \Cref{alg_alsox} returns a feasible solution with 
		{$\bar v^A\leq v^\CVaR+\delta_1 $}, where {$\bar v^A$ is the output objective value} and $\delta_1 $ is the chosen stopping tolerance parameter in \Cref{alg_alsox}.
	\end{corollary}
	{It is worth noting that using a proper stopping tolerance parameter, $\alsox$ \Cref{alg_alsox} in general returns a better solution than the $\CVaR$ approximation; however, since the binary search procedure requires solving many similar hinge-loss approximations, it might be slower than the $\CVaR$ approximation. Our numerical study shows that for the linear CCP, the computational time of $\alsox$ \Cref{alg_alsox} is usually longer than the $\CVaR$ approximation since the off-the-shelf solvers excel in solving large-scale linear programs, while for the nonlinear CCP, $\alsox$ \Cref{alg_alsox} takes a shorter time than the CVaR approximation using the subgradient method. Nevertheless, in both cases, the solution quality of $\alsox$ \Cref{alg_alsox} is found to be consistently better than that of the $\CVaR$ approximation.} 
	
	Notably, the convexity assumption of set $\mathcal{X}$ in Assumption \ref{A_2} is of a necessity to the result of \Cref{the1}. In the non-convex setting, the $\alsox$ \eqref{eq_also_x_2} cannot be guaranteed to be better than the $\CVaR$ approximation \eqref{cvar_bilevel}. Particularly,  {the following two examples show that the $\alsox$ can return a better solution than the $\CVaR$ approximation and vice versa when set $\X$ is nonconvex. }
	
	\begin{example}
		\rm
		\label{alsox_better_nonconvex}
		Let us revisit \Cref{example_alsox_cvar} in Appendix~\ref{sec_append_example} with an additional restriction that $x$ is an integer, i.e, set $\mathcal{X}=\mathbb{Z}_+$. In this case, 
		we have $v^*=2$, $v^A=2$, and $v^{\CVaR}=3$. 	Thus, the $\alsox$ outperforms the $\CVaR$ approximation. 	\QEDB	
	\end{example}
	\begin{subequations}
		\begin{example}
			\label{cvar_better_nonconvex}
			\rm
			Consider a CCP with 4 equiprobable scenarios (i.e., $N=4$, $\Pr\{\tilde{\rxi}=\bm\xi^i\}=1/N$), risk level $\varepsilon=1/2$, set $\mathcal{X}=\{0,1\}$, function $g(\bm{x},{\rxi})={\xi}_1 x - {\xi}_2$, 
			$\xi_1^1=-49$, $\xi_1^2=\xi_1^3=\xi_1^4=101$, $\xi_2^1=-50$, and $\xi_2^2=\xi_2^3=\xi_2^4=99$. Under this setting, CCP \eqref{eq_ccp} becomes
			\begin{equation*}
			v^*=\min_{x\in \{0,1\}}\left\{-x\colon 	 \I(49x\geq 50)+\I(101x\leq 99)+\I( 101x\leq 99)+\I( 101x\leq 99)\geq 2 \right\}.
			\end{equation*}
			
			The $\CVaR$ approximation of this CCP is
			\begin{align*}
			v^\CVaR=\min_{x\in\{0,1\},\beta\leq 0,\bm s}\biggl\{-x\colon  {\begin{array}{l}
				\displaystyle	-49x+50\leq s_1,101x-99\leq s_2, 101x-99\leq s_3,\\ 
				\displaystyle101x-99\leq s_4, \frac{1}{4}\sum_{i\in[4]}s_i-\beta/2\leq 0, s_i\geq \beta,\forall i\in[4] 
				\end{array}}\biggr\}.
			\end{align*}
			By the simple calculation, the optimal value is $v^*=0$ with the optimal solution $x^*=0$, and $v^{\CVaR}=0$ for the $\CVaR$ approximation \eqref{eq_cvar} with the optimal solution $x^*=0, \beta^*=-123.5, s_1^*=50, s_2^*=s_3^*=s_4^*=-99$. However, the $\alsox$ \eqref{eq_also_x_2} of this example is infeasible, which can be formulated as
			{
				\begin{align*}
				&	v^A =\min _{ {t}}\, \biggl\{t\colon  \sum_{i\in[4]}	\I({s^*_i}=0) \geq 2,\\
				&(x^*,\bm s^*)\in\argmin_{\begin{subarray}{c}
					x\in \{0,1\}, \bm s\geq \bm 0\end{subarray}}\biggl\{ \frac{1}{4} \sum_{i\in[4]}s_i\colon \begin{array}{l}
				\displaystyle-49x+50\leq s_1, 101x-99\leq s_2, \\
				\displaystyle 101x-99\leq s_3,101x-99\leq s_4,-x\leq t \end{array}\biggr\}
				\biggr\}.
				\end{align*}}
			Particularly,	for any $t\geq-1$, the hinge-loss approximation returns a solution with $s_1^*=1$, $s_2^*=s_3^*=s_4^*=2$, $ x^*=1$, and the support size of $\bm{s}^*$ is greater than $2$, then we have to increase the objective bound $t$ to the infinity. Therefore, in this example, the $\CVaR$ approximation always returns the optimal solution, but the $\alsox$ {fails to} find any feasible solution.
			\QEDB		
		\end{example}
	\end{subequations}

	\subsection{Exactness of $\alsoxt$ }
	\label{alsox_exactness}
	In this subsection, we show sufficient conditions under which $\alsox$ \eqref{eq_also_x_2} can provide an exact optimal solution to CCP \eqref{eq_ccp}.  {To begin with, the following example shows that $\alsox$ \eqref{eq_also_x_2} may not be able to find the exact solution of CCP \eqref{eq_ccp}  even under Assumptions \ref{A_1}-\ref{A_2}. Thus, in this subsection, we explore the conditions under which the $\alsox$ \eqref{eq_also_x_2}  returns an exact optimal solution.}
	
	\begin{example}
		\label{am_dc_example}
		\rm
		Consider a CCP with 3 equiprobable scenarios (i.e., $N=3$, $\Pr\{\tilde{\rxi}=\bm\xi^i\}=1/N$), risk level $\varepsilon=1/3$, set $\mathcal{X}=\Re_+^2$, function $g(\bm{x},{\rxi})=-{\rxi}^\top \bm{x} +1$, and 
		$\bm\xi^1 =(2,3)^\top$, $\bm\xi^2 =(2,1)^\top$, $\bm\xi^3 =(1,2)^\top $. 
		The optimal value of this CCP can be found by solving the following mixed-integer linear program
		\begin{equation*}
		v^*=	\min_{\bm{x}\in\Re_+^2,\bm{z}\in \{0,1\}^3}\left\{ x_1+x_2\colon 2x_1+3x_2\geq z_1,2x_1+x_2\geq z_2,x_1+2x_2\geq z_3,\sum_{i\in[3]}z_i\geq 2 \right\},
		\end{equation*}	
		i.e., $v^*=0.5$. \\
		The corresponding $\alsox$ \eqref{eq_also_x_2} is
		{
			\begin{align*}
			&v^A =\min _{ {t}}\,\biggl\{ t\colon \sum_{i\in[3]}	\I({s^*_i}=0) \geq 2,\\
			&
			( \bm x^*,\bm s^*)\in\argmin_{\bm{x}\in\Re^2_+,\bm{s}\in\Re^3_+}\biggl\{\frac{1}{3} \sum_{i\in[3]}s_i\colon  
			\begin{array}{l}
			\displaystyle 2x_1+3x_2\geq 1-s_1,2x_1+x_2\geq 1-s_2, \\
			\displaystyle x_1+2x_2\geq 1-s_3, x_1+x_2\leq t  \end{array}
			\biggr\} \biggr\}.
			\end{align*}}
		Simple calculations show that $\alsox$ has an optimal $v^{A}=2/3>v^*$.
		\QEDB
	\end{example}
	
	{\Cref{am_dc_example} motivates us to find the special cases of CCP \eqref{eq_ccp} under which $\alsox$ \eqref{eq_also_x_2}  can provide an
		optimal solution.}
	
	
	\noindent\textbf{Special Case I of Exactness: CCPs with Equality Constraint}\\
	{This special case of CCP \eqref{eq_ccp} consists of a linear uncertain equality constraint. This special case is motivated by the following two distinct applications:
		\begin{itemize}
			\item 
			The first one is to find a feasible subsystem of linear equalities studied by \cite{amaldi1995complexity}. That is, given a possibly infeasible linear system $\bm A\bm x=\bm b$ with matrix $\bm A \in\Re^{m\times n}$ and vector $\bm{b}\in \Re^m$, and a positive integer $K\in [m]$, the goal of the problem is to seek a solution $\bm x\in\Re^n$ such that $\bm x$ satisfies at least $K$ linear equalities of the system; and
			\item The second one is to find a sparse solution from the linear system studied by \cite{nemirovski2001lectures}. That is, given a linear system $\bm A\bm x=\bm b$ with matrix $\bm A \in\Re^{m\times n}$ and vector $\bm{b}\in \Re^m$ and an integer $K\in [n]$, the goal of the problem is to seek a sparse solution  $\bm x$ such that the support size of $\bm x$ is no larger than $K$.
	\end{itemize}}{In this special case, we assume set $\X=\{\bm x\in \Re^n:\bm U^\top \bm x=\bm h\}$ with matrix $\bm U\in \Re^{m\times n}$ and vector $\bm h\in \Re^{n}$},  
	$g_1(\bm x,{{\rxi}}) ={ {\rxi}}^\top (\bm A\bm{x})+ a( {\rxi})-\bm b^\top\bm x$, and $g_2(\bm x, {\rxi}) = - {\rxi}^\top (\bm A\bm{x})- a( {\rxi})+\bm b^\top\bm x$, where $\bm A\in \Re^{m\times n},\bm b\in \Re^{n}$, and $a(\cdot):\Xi\rightarrow\Re$. Hence, CCP \eqref{eq_ccp} reduces to 
	\begin{align}
	v^* =\min_{\bm x\in \X} \left\{ \bm{c}^\top\bm x\colon\Pr\left\{\tilde{\bm\xi} \colon \tilde{\bm {\xi}}^{\top} (\bm A\bm{x}) + a(\tilde{\bm {\xi}})=\bm b^\top\bm x \right\} \geq 1-\varepsilon\right \},\label{eq_ccp_equal}
	\end{align}
	while the corresponding $\alsox$ \eqref{eq_also_x_2} can be written as
	\begin{subequations}\label{eq_alsox_equal}
		\begin{align}
		v^A =& \min _{ {t}}\quad  t,\\
		\text{s.t.}\quad&( \bm x^*, {s^*(\cdot)})\in\argmin _{\bm {x}\in \X,{s(\cdot)}} \left\{ \E\left[|{s}(\tilde{\rxi})|\right]\colon \tilde{\bm {\xi}}^{\top} (\bm A\bm{x})+ a(\tilde{\bm {\xi}})-\bm b^\top\bm x=s(\tilde{\rxi}),\bm{c}^\top \bm{x} = t \right\},\label{single_ccp_equality_1}\\
		&\Pr\left\{\tilde{\bm\xi} : s^*(\tilde{\rxi})=0\right\} \geq 1-\varepsilon.
		\end{align}
	\end{subequations}
	Above, we replace the inequality $\bm{c}^\top \bm{x} \leq t $ by the equality $\bm{c}^\top \bm{x} = t$ since at optimality, the equality must hold. 
	The sufficient condition under which ${\alsox}$ \eqref{eq_alsox_equal} returns an optimal solution of CCP \eqref{eq_ccp_equal} relies on the following property.
	\begin{quote}\it
		\textbf{Generalized Nullspace Property.} For any $(\bm{x},{s(\cdot)})$ such that ${\bm {\xi}}^{\top} (\bm A\bm{x})-\bm b^\top\bm x-s({\rxi})=0,\bm{c}^\top \bm{x} = 0, {\bm U^\top \bm x=\bm 0}, {s}({\rxi})\not=0$, then for any $\Pr-$measurable set $\mathcal{S} \subseteq\Xi $ such that $\Pr\{\tilde{\rxi}\colon\tilde{\rxi} \in \mathcal{S}\} \leq \varepsilon$, one must have $\E[|{s}(\tilde{\rxi})|\I(\tilde{\rxi} \in \mathcal{S})]<1/2\E[|{s}(\tilde{\rxi})|]$.
	\end{quote}
	
	\noindent It is worthy of mentioning that this generalized nullspace property extends the notion of the nullspace property for the sparse signal recovery (i.e., property 1.3.4 in \citealt{nemirovski2001lectures}), where the latter {is useful to characterize the uniqueness of the sparse solution satisfying a finite set of linear equations.}
	We also remark that if the probability distribution $\Pr$ consists of $N$ equiprobable scenarios, the generalized nullspace property can be simplified as
	\begin{quote}\it
		\textbf{Generalized Nullspace Property with $\boldsymbol{N}$ Equiprobable Scenarios.} For any $(\bm{x},\bm{s})\in \Re^n\times \Re^N$ such that ${{\bm \xi}^i}^\top(A\bm{x})-B\bm x-s_i=0$ for each $i\in [N]$, $\bm{c}^\top \bm{x} = 0, {\bm U^\top \bm x=\bm 0}, \bm{s}\not=\bm{0}$, then for any set $\mathcal{S} \subseteq [N] $ with $|\mathcal{S}|\leq \floor{\varepsilon N}$, we must have $\sum_{i\in\mathcal{S}}|s_i|<1/2\sum_{i\in[N]}|s_i|$.
	\end{quote}
	
	Next, we are ready to establish the exactness of $\alsox$ \eqref{eq_alsox_equal} for this special case, given that the generalized nullspace property holds. 
	\begin{restatable}{theorem}{exactnessec}\label{the4} 
		For Special Case 1, the following results must hold:
		\begin{enumerate}[label=(\roman*)]
			\item For any feasible pair of $(t, {a(\cdot)})$ under which the hinge-loss approximation \eqref{single_ccp_equality_1} has a feasible solution $(\bm{x},{s(\cdot)})$ satisfying $\Pr\{\tilde{\rxi} : s(\tilde{\rxi})=0\} \geq 1-\varepsilon$, 
			every optimal solution $(\bm{x}^*,{s^*(\cdot)})$ to the hinge-loss approximation \eqref{single_ccp_equality_1} shares the same ${s^*(\cdot)}$ and satisfies $\Pr\{\tilde{\bm\xi} : s^*(\tilde{\rxi})=0\} \geq 1-\varepsilon$, if and only if the generalized nullspace property holds; and
			\item Suppose that the generalized nullspace property holds. Then the optimal values of CCP \eqref{eq_ccp_equal} and $\alsox$ \eqref{eq_alsox_equal} coincide, i.e., $v^A=v^*$. Moreover, every optimal solution $(\bm{x}^*,{s^*(\cdot)})$ to the CCP \eqref{eq_ccp_equal} shares the same ${s^*(\cdot)}$.
		\end{enumerate}
	\end{restatable}
	
	\proof
	See Appendix~\ref{proof_the4}.
	\QEDA
	
	
	The follow example illustrates the correctness of \Cref{the4}.
	\begin{example}
		\rm
		\label{exactness_eq_example}
		Consider a CCP with 3 equiprobable scenarios (i.e., $N=3$, $\Pr\{\tilde{\rxi}=\bm\xi^i\}=1/N$), risk level $\varepsilon=1/3$, set $\mathcal{X}=\Re^2$, function $g_1(\bm{x},{\rxi})=-{\rxi}^\top \bm{x} +1$ and $g_2(\bm{x},{\rxi})={\rxi}^\top \bm{x} -1$, 
		$\bm\xi^1 =(2,3)^\top$, $\bm\xi^2 =(2,1)^\top$, $\bm\xi^3 =(1,2)^\top$. Under this setting, the optimal solution of this CCP can be obtained by solving the following mixed-integer linear program
		\begin{align*}
		v^*=\min_{\bm{x}\in\Re^2}\left\{x_1+x_2\colon
		\I(2x_1+3x_2=1)+\I(2x_1+x_2=1)+\I( x_1+2x_2=1)\geq 2 \right\}
		\end{align*}
		with optimal value $v^*=1/2$. \\
		In this example, the generalized nullspace property holds. In fact, any $(\bm{x},\bm{s})\in \Re^2\times \Re^3$ satisfies the following conditions
		$$2x_1+3x_2-s_1=0,2x_1+x_2-s_2=0,x_1+2x_2-s_3=0,x_1+x_2= 0, \bm{s}\not=\bm{0},$$
		which is equivalent to
		$$x_2=-x_1, s_1=-x_1, s_2=x_1, s_3=-x_1,x_1\not=0.$$ 
		Simple calculations show that for any set $\mathcal{S} \subseteq [3] $ with $|\mathcal{S}|\leq 1$, we must have $\sum_{i\in\mathcal{S}}|s_i|<1/2\sum_{i\in[3]}|s_i|$. 
		Therefore, according to \Cref{the4}, the optimal value for $\alsox$ \eqref{eq_also_x_2} must be $v^A=v^*=1/2$. Indeed, in this example, the $\alsox$ \eqref{eq_also_x_2} reduces to
		{
			\begin{align*}
			&v^A =\min _{t}\bigg\{ t\colon \sum_{i\in[3]}	\I({s^*_i}=0) \geq 2, \\
			&( \bm x^*,\bm s^*)\in \argmin_{\bm{x}\in\Re^2, \bm{s}\in\Re^3}\biggl\{ \frac{1}{3}\sum_{i\in[3]}|s_i|\colon
			\begin{array}{l}
			\displaystyle  2x_1+3x_2=1+s_1,2x_1+x_2= 1+s_2,\\
			\displaystyle x_1+2x_2= 1+s_3, x_1+x_2\leq t \end{array}
			\biggr\}\biggr\}.
			\end{align*}}
		We see that if $t\geq1/2$, the optimal solution of the corresponding hinge-loss approximation is $x_1^*=1/2$, $x_2 ^*= 0, s_1^*=s_2^*=0, s_3^*=-1/2$. 
		This suggests that the optimal value for $\alsox$ \eqref{eq_also_x_2} is $v^A=v^*=1/2$.
		\QEDB
	\end{example}
	
	
	\noindent\textbf{Special Case II of Exactness: CCPs with Generalized Set-covering Type of Uncertain Constraints}\\
	In this special case, we consider the function $g\colon \mathcal{X}\times \Xi\rightarrow \Re_-\cup \{M\}$, where $M\in\Re_{++}$ is a positive constant. This special case is a generalization of the chance constrained set covering problem (see, e.g., \citealt{ahmed2013probabilistic,beraldi2002probabilistic}), where $\X\subseteq \{0,1\}^n$, $M=1$, and $g(\x,{\rxi})=1-{\rxi}^\top \x$ with binary support ${\rxi}\in \{0,1\}^n:=\Xi$. It is worthy of noting that (i) this special case might violate Assumption~\ref{A_2} that set $\mathcal{X}$ is convex and (ii) when the probability distribution is finite-support, this special case has been studied in \cite{ahmed2017nonanticipative} (see proposition 12). We show that $\alsox$ \eqref{eq_also_x_2} {still provides} an optimal solution when the probability distribution is arbitrary. 
	\begin{restatable}{theorem}{exactnesssetcovering}\label{exactness_set_covering} 
		(A generalization of proposition 12 in \citealt{ahmed2017nonanticipative}) Suppose that $g(\bm x,{\bm\xi})\colon \mathcal{X}\times \Xi\rightarrow \Re_-\cup \{M\}$, where $M\in\Re_{++}$ is a positive constant, the optimal value of $\alsox$ \eqref{eq_also_x_2} coincides with that of CCP \eqref{eq_ccp}.
	\end{restatable}
	
	\proof
	See Appendix~\ref{proof_exactness_set_covering}.
	\QEDA

	
	
	Interested readers are referred to \Cref{exactness_set_covering_example} in Appendix~\ref{sec_append_example} for a demonstration of the Special Case 2.

	\noindent\textbf{Special Case III of Exactness: A Single Linear CCP under an Elliptical Distribution}\\
	Let us revisit Special Case 2 in Section~\ref{alsox_sp}, which considers a single linear CCP under an elliptical distribution. For this special case, we  {show} that under additional assumptions, $\alsox$ \eqref{eq_alsox_el} provides an optimal solution to CCP \eqref{eq_ccp}.
	\begin{restatable}{theorem}{exactel}\label{exact_el} 
		For the single linear CCP \eqref{ccp_el} under an elliptical distribution, 
		the $\alsox$ \eqref{eq_alsox_el} provides an optimal solution to CCP \eqref{ccp_el}, provided that (i) $\mathcal{X}\subseteq\{\bm x:\sqrt{\bm{a}_1(\bm{x})^\top\bm{\mathrm{{\Sigma}}}\bm{a}_1(\bm{x})}=C\} $, where $C$ is a positive constant; or (ii) $\mathcal{X}\subseteq\{\bm x:b_1(\bm x)-\bm\mu^\top \bm{a}_1(\bm{x})=C\} $, where $C$ is an arbitrary constant.
	\end{restatable}
	
	\proof
	See Appendix~\ref{proof_exact_el}.
	\QEDA
%
	
	{Please note that our analysis holds for any $\varepsilon\in(0,1)$. For the larger risk level $\varepsilon\in(0.5,1)$, the feasible region of CCP \eqref{ccp_el} can be non-convex  \citep{henrion2006some} and intractable (see \Cref{el_hard} in Appendix~\ref{proof_el_hard}).} {The two conditions in \Cref{exact_el} may not be very strong and can be found in the CCP literature or relevant application problems. For example, proposition 5.1 of \citealt{van2019eventual} studied the eventual convexity analysis of a CCP under an elliptical distribution and the condition that $\bm u=\bm 0$ and  $b_1(\bm x)$ is a nonnegative constant, which is a special case of Condition (ii). }

	{We also remark that either condition in \Cref{exact_el} can be satisfied in practical application problems, for example, let us consider the following portfolio selection problem \citep{markowitz1991foundations,pagnoncelli2009sample} as
				$
		v^*=	\min_{\bm{x}\in\Re^n_+} \{\bm{c}^\top \bm{x} \colon \Pr\{b_1\geq \trxi^\top\bm x \}\geq 1-\varepsilon,\bm e^\top \bm x=1\}, 
$
	where decision vector $\bm{x}$ denotes the investment plan, scalar $b_1$ represents the portfolio return level, $\trxi$ is the stochastic return vector of $n$ risky assets, and $\bm{c}$ is the cost vector. Suppose $\trxi$ follows a multi-variate elliptical distribution. Then using the notation in \Cref{alsox_sp},  the chance constrained portfolio selection problem is equivalent to
		\begin{align}
				v^*=	\min_{\bm{x}\in\Re^n_+} \left\{\bm{c}^\top \bm{x} \colon b_1
-\bm{\mu}^\top\bm {x}\geq \mathrm{\Phi}^{-1}(1-\varepsilon) \sqrt{\bm x^\top\bm{\mathrm{{\Sigma}}}\bm x},\bm e^\top \bm x=1\right\}.\label{example_portfolio_eq}
			\end{align}
		Let us consider the following two cases of problem \eqref{example_portfolio_eq}:
		\begin{itemize}
			\item To find an efficient portfolio, that is, to achieve the highest expected return of all the feasible portfolios with the same risk level (see the details in \citealt{fabozzi2012mean}), the overall risk level $ \sqrt{\bm x^\top\bm{\mathrm{{\Sigma}}}\bm x}$ is a constant (i.e., Condition (i) of \Cref{exact_el} is satisfied). In this case, according to \Cref{exact_el}, $\alsox$ \eqref{eq_alsox_el} returns an optimal solution; and 
			\item When all the portfolios share the same expected return (see the discussions in \citealt{chow1995portfolio}), that is, when $\bm u^\top \bm x$ is a constant, Condition (ii) of \Cref{exact_el} is satisfied. Similarly, according to \Cref{exact_el}, $\alsox$ \eqref{eq_alsox_el} returns an optimal solution. 
		\end{itemize}}	
	
	
	{Another byproduct of \Cref{exact_el} is to demonstrate that $\alsox$ can be strictly better than $\CVaR$ approximation. This is because $\CVaR$ approximation for the single linear CCP \eqref{ccp_el} under an elliptical distribution (see, e.g., theorem 9 in \citealt{chen2019sharing}) can be written as 
		\begin{align*}
		v^\CVaR=\min_{\bm x\in\X}\left\{\bm c^\top \bm x \colon
		b_1(\bm x) - \bm{\mu}^\top\bm{a}_1(\bm{x})\geq 
		\left(\eta^{\CVaR}+\mathrm{\Phi}^{-1}(1-\varepsilon)\right)\sqrt{\bm{a}_1(\bm{x})^\top\bm{\mathrm{{\Sigma}}}\bm{a}_1(\bm{x})} \right\},
		\end{align*}
		with constant
		\begin{equation*}
		\eta^{\CVaR} =\frac{1}{\varepsilon}\int_{\frac{1}{2}(\mathrm{\Phi}^{-1}(1-\varepsilon))^2}^\infty k\hat{g}(z)dz-\mathrm{\Phi}^{-1}(1-\varepsilon) = \overline{G}\left(\left(\mathrm{\Phi}^{-1}(1-\varepsilon)\right)^2/2\right)/\varepsilon -\mathrm{\Phi}^{-1}(1-\varepsilon). 
		\end{equation*}
		According to the proof of \Cref{exact_el} in Appendix~\ref{proof_exact_el}, we observe that $\eta^{\CVaR}> 0$ for all $\varepsilon\in(0,1)$. Therefore, the feasible region of the $\CVaR$ approximation is strictly contained in that of CCP \eqref{ccp_el}. As the $\alsox$ is exact under the conditions of \Cref{exact_el}, $\alsox$ can be strictly better than the $\CVaR$ approximation whenever $v^\CVaR>v^*$.}

	{We also remark that albeit the first condition guarantees the optimality of $\alsox$ \eqref{eq_alsox_el}, it might violate Assumption \ref{A_2}, namely, set $\X$ might not be convex. We show that these two conditions are the best that we can expect for the exactness of $\alsox$ \eqref{eq_alsox_el}.  In fact,
		{ \Cref{exactness_el_example} in Appendix~\ref{sec_append_example}} shows that $\alsox$ \eqref{eq_alsox_el}  might not be able to find an optimal solution to CCP \eqref{ccp_el} if it violates both conditions of \Cref{exact_el}.}
	
%
	
	\subsection{Approximation Ratio for the Covering CCPs under Discrete Support}
	In this subsection, we analyze the approximation ratio for a special case of finite-support CCP \eqref{eq_also_x_discrete} with the covering structure (i.e., covering CCPs), where $\X= \Re_+^n$, $\bm{c}\in \Re_+^n$, the constraints $g(\bm x,\bm{\xi}^i) = \bm{b}^i-\bm{A}^i\bm{x}$ with $\bm{A}^i\in \Re_+^{m\times n}$, $ \bm{b}^i\in \Re_{++}^m$, for all $ i\in [N]$. Various applications have been studied in literature (see, e.g., \citealt{shiina1999numerical,takyi1999surface,talluri2006vendor,deng2016decomposition,dentcheva2000concavity,xie2020bicriteria,qiu2014covering}) that can be formulated as covering CCPs. 
	Without loss of generality, $\bm{b}^i$ can always be normalized to $\bm{e}$, and thus covering CCP \eqref{eq_ccp} admits the following form
	\begin{equation}
	\label{covering_ccp}
	v^*=\min _{\bm{x}\in\Re^n_+,\bm{z}\in \{0,1\}^N} \left\{ \bm{c}^\top \bm{x}\colon \sum_{i\in [N]}z_i\geq N-\floor{N\varepsilon},\bm{A}^i\bm{x}\geq z_i\e, \forall i \in [N]\right\},
	\end{equation}
	and the corresponding $\alsox$ \eqref{eq_also_x_2} is
	\begin{subequations}\label{covering}
		\begin{align}
		v^A =&\min _{ {t}}\quad t,\\
		\text{s.t.}\quad&(\bm x^*,\bm s^*)\in \argmin _{\bm{x}\in\Re^n_+, \bm s\in\Re^N_+}\left\{ \frac{1}{N}\sum_{i\in[N]}s_i\colon \bm{c}^\top \bm{x} \leq t, \bm{A}^i\bm{x}\geq \bm{e}-s_i\bm{e}, \forall i\in [N] \right\}, \label{covering_hinge}\\
		& \sum_{i\in[N]}	\I({s^*_i}=0) \geq N-\floor{N\varepsilon}.
		\end{align}
	\end{subequations}
	Consider the following continuous relaxation of the covering CCP \eqref{covering_ccp} as
	\begin{equation}
	\label{covering_relax}
	v^{rel}=\min _{\bm{x}\in\Re^n_+, \bm s\in\Re^N_+} \left\{ \bm{c}^\top \bm{x}\colon \sum_{i\in[N]}s_i\leq \floor{N\varepsilon}, \bm{A}^i\bm{x}\geq \bm{e}-s_i\bm{e}, \forall i\in [N]\right\}.
	\end{equation}
	Particularly, we observe that (i) in the continuous relaxation \eqref{covering_ccp}, at optimality, we must have $s_i\in [0,1]$ for all $i\in [N]$; and (ii) the continuous relaxation value $v^{rel}$ is a lower bound for the covering CCP \eqref{covering_ccp}. 
	Next, we {show} that for any $t\geq (\floor{N\varepsilon}+1)v^{rel}$, any optimal solution to the hinge-loss approximation \eqref{covering_hinge} is feasible to the covering CCP \eqref{covering_ccp}. This implies that $v^A/v^*\leq v^A/v^{rel}\leq \floor{N\varepsilon}+1$, i.e., $\alsox$ \eqref{covering} yields 
	a $(\floor{N\varepsilon}+1)-$approximation guarantee, achieving the same best known approximation ratio for the finite-support covering CCPs \citep{ahmed2018relaxations}. {When employing $\alsox$ \Cref{alg_alsox}, if finding a feasible solution of a covering CCP is difficult,  one can use $(\floor{N\varepsilon}+1)v^{rel}$ as a valid upper bound. In the numerical study, we apply this strategy when running $\alsox$ \Cref{alg_alsox} to solve a covering CCP instance.}
	
	\begin{restatable}{theorem}{thmccp}\label{thm_ccp} 
		For the covering CCP \eqref{covering_ccp},
		the $\alsox$ \eqref{covering} yields a $(\floor{N\varepsilon}+1)-$approximation guarantee, that is, $v^A/v^*\leq \floor{N\varepsilon}+1$. 
	\end{restatable}
	
	\proof
	See Appendix~\ref{proof_thm_ccp}.
	\QEDA

	Finally, we conclude this section by showing that the approximation ratio of $\alsox$ \eqref{covering} for the covering CCP \eqref{covering_ccp} is tight.
	\begin{restatable}{proposition}{propcoveirngccp}\label{prop_coveirng_ccp} 
		For the covering CCP \eqref{covering_ccp},
		the $(\floor{N\varepsilon}+1)-$approximation ratio of $\alsox$ \eqref{covering} is tight, i.e., it is possible that $v^A/v^*= \floor{N\varepsilon}+1$. 
	\end{restatable}
	
	\begin{proof} See Appendix~\ref{proof_prop_coveirng_ccp}.
		\QEDA
	\end{proof}

	\section{$\alsoxt +$ Algorithm: Breaking the Symmetry and Improving $\alsoxt$ \Cref{alg_alsox} using Alternating Minimization Method}
	\label{ap}
	
	In \Cref{sec_alsox_sub}, recall that we derive $\alsox$ \eqref{eq_also_x_2} and its related \Cref{alg_alsox} by letting the functional variable $z(\rxi)=1$ in the CCP \eqref{alsox_bilnear}. As shown in the previous sections, $\alsox$ \eqref{eq_also_x_2} successfully provides exact optimal and approximate solutions for many special families of CCP \eqref{alsox_bilnear}. On the other hand, simply forcing the functional variable $z(\rxi)=1$ in CCP \eqref{alsox_bilnear} might not be ideal, for instance, \Cref{am_dc_example}  demonstrates that $\alsox$ \eqref{eq_also_x_2} might be fooled if the random parameters ${\trxi}$ obey a joint symmetric distribution. This motivates us to improve the $\alsox$ \Cref{alg_alsox} by optimizing the functional variable ${z(\cdot)}$ as well. In general, optimizing over both ${s(\cdot)}$ and ${z(\cdot)}$ can be difficult. Thus, we propose to enhance $\alsox$ \Cref{alg_alsox}, termed ``$\alsox+$ algorithm," by using a better Alternating Minimization ($\AM$) method, which is to optimize over $(\bm{x},{s(\cdot)})$ and ${z(\cdot)}$ of the lower-level problem \eqref{alsox_bilnearb} in an alternating fashion. The key idea of $\alsox+$ algorithm is that in Step 4 of $\alsox$ \Cref{alg_alsox}, one should run the $\AM$ method if the current solution is infeasible (i.e., the optimal solution $(\bm{x}^*,{s^*(\cdot)})$ of the hinge-loss approximation \eqref{eq_also_x_2b} violates the chance constraint, $\Pr\{\tilde{\bm\xi} \colon{{s}^*}(\tilde{\rxi})>0\} >\varepsilon$). 

	\subsection{The Proposed Alternating Minimization (\textbf{AM}) Method}
	\label{pam}
	To begin with, we  {first introduce} the $\AM$ method. As mentioned earlier, for a given $t$, when the hinge-loss approximation \eqref{eq_also_x_2b} is unable to provide a feasible solution to the CCP \eqref{eq_ccp}, i.e., its optimal solution $(\bm{x}^*,{s^*(\cdot)})$ is subject to $\Pr\{\tilde{\bm\xi} \colon{{s}^*}(\tilde{\rxi})>0\} >\varepsilon$. Under this circumstance, we run the $\AM$ method to optimize the lower-level problem \eqref{alsox_bilnearb} in hope of overcoming the infeasibility, which proceeds as follows. First, we observe that the constraint system in the lower-level problem \eqref{alsox_bilnearb} can be separated into two parts, with respect to the functional variable ${z(\cdot)}$ and with respect to variables $\bm x$ and ${s(\cdot)}$. This allows us to optimize over ${z(\cdot)}$ and $(\bm{x},{s(\cdot)})$ in an iterative way. Specifically, at iteration $k+1$, when fixing $(\bm x,{s(\cdot)})=(\bm x^{k},{s^k(\cdot)})$ with $(\bm x^{k},{s^k(\cdot)})$ from the previous iteration in the lower-level problem \eqref{alsox_bilnearb}, we solve the following optimization problem:
	\begin{subequations}
		\begin{align}
		{{z}^{k+1}(\cdot)} \in \argmin _{{z(\cdot)}} \left\{ \E\left[{z}(\tilde{\rxi}){s}^{k}(\tilde{\rxi})\right]\colon z(\tilde{\rxi})\in [0,1], \E[{z}(\tilde{\rxi})]\geq 1-\varepsilon \right\},	\label{basic_ap_z}
		\end{align}
		which can be done via sorting the values of $\{s^{k}({\bm{\xi}})\}_{{\rxi}\in\Xi}$. Next fixing the value of the functional variable ${z(\cdot)}={z^{k+1}(\cdot)}$, we solve the following convex optimization problem:
		\begin{align}
		\left(\bm{x}^{k+1},{s^{k+1}(\cdot)}\right) \in \argmin _{ \bm{x}\in \mathcal{X},{s(\cdot)}} \left\{ \E\left[z^{k+1}(\tilde{\rxi}){s}(\tilde{\rxi})\right]\colon \bm{c}^\top \bm{x} \leq t, g(\bm x,\tilde{\bm {\xi}})\leq {s}(\tilde{\rxi}), {s}(\tilde{\rxi})\geq 0 \right\}.\label{basic_ap_1}
		\end{align}
		Note that the problem \eqref{basic_ap_1} can be further simplified as
		\begin{equation*}
		\bm{x}^{k+1}\in \argmin _{\bm{x}\in \mathcal{X}} \left\{ \E\left[z^{k+1}(\tilde{\rxi}) [g(\bm x,\tilde{\bm {\xi}})]_+ \right]\colon \bm{c}^\top \bm{x} \leq t \right\},
		\end{equation*}
		which can be solved using the subgradient descent method proposed in \Cref{sdm} for instance.
		We continue this procedure until reaching the stopping criterion. The detailed implementation can be found in \Cref{alg_am}. Note that according to bilinear programming Formulation \eqref{eq_bilinear}, the output solution $	\bm{x}^{k+1}$ is feasible to CCP \eqref{eq_ccp} if and only if $\E[z^{k+1}(\tilde{\rxi})s^{k+1}(\tilde{\rxi})]=0$.
	\end{subequations}
	
	\begin{algorithm}[htbp]
		\caption{Alternating Minimization ($\AM$) Method to Solve the Lower-level Problem \eqref{alsox_bilnearb}}
		\label{alg_am}
		\begin{algorithmic}[1]
			\State Let $k=0$. Let $\delta_2$ denote the stopping tolerance parameter, $t$ be the current given value of the upper-level problem, and ${z^0(\cdot)}$ be the given initial solution of ${z(\cdot)}$, respectively
			\Do
			\State 	Solve \eqref{basic_ap_z} and \eqref{basic_ap_1} with optimal solutions ${{z}^{k+1}(\cdot)}$ and $(\bm x^{k+1},{s^{k+1}(\cdot)})$, respectively
			\State Let	$\Delta=\left\lvert 	\E[z^{k+1}(\tilde{\rxi})s^{k+1}(\tilde{\rxi})]-	\E[z^{k}(\tilde{\rxi})s^{k}(\tilde{\rxi})]\right\rvert $ and $k= k+1$
			\doWhile {$\Delta\geq\delta_2$}
			\State The output solution $\bm x^{k+1}$ is feasible to CCP \eqref{eq_ccp} if $ \E[{z}^{k+1}(\tilde{\rxi}){s}^{k+1}(\tilde{\rxi})]=0$; otherwise, it is infeasible
		\end{algorithmic}
	\end{algorithm}
	
	In $\AM$ \Cref{alg_am}, we observe that the sequence of objective values $\{\E[z^{k}(\tilde{\rxi})s^{k}(\tilde{\rxi})]\}_{k\in \Ze_+}$ converges. This demonstrates that the stopping criterion using the objective values is indeed valid.
	
	\begin{restatable}{proposition}{alsoxconvergence}\label{alsox_convergence} 	
		The sequence of objective values $\{\E[z^{k}(\tilde{\rxi})s^{k}(\tilde{\rxi})]\}_{k\in \Ze_+}$ generated by the $\AM$ \Cref{alg_am} is monotonically nonincreasing, bounded from below, and hence converges.
	\end{restatable}
	\proof
	See Appendix~\ref{proof_alsox_convergence}.
	\QEDA
	

	It is worthy of mentioning that the sequence of solutions $\{(\bm{x}^k,{s^{k}(\cdot),z^{k}(\cdot)})\}_{k\in \Ze_+}$ generated by the $\AM$ \Cref{alg_am} might not converge. If this case occurs, we {choose} a convergent subsequence of $\{(\bm{x}^k,{s^{k}(\cdot),z^{k}(\cdot)})\}_{k\in \Ze_+}$ and its accumulative point as the output.

	\subsection{$\textbf{AM}$ Method is Better Than Difference-of-Convex ($\textbf{DC}$) Approach}
	\label{better_than_DC}
	This subsection compares the $\AM$ method with the well-known difference-of-convex ($\DC$) approach to solve the lower-level problem \eqref{alsox_bilnearb}, and shows that the $\AM$ method provides a better-quality solution than that of the $\DC$ approach.
	
	{We first apply the well-known difference-of-convex ($\DC$) approach (see, e.g., section 2 of \citealt{tao1997convex})
		to solve the lower-level problem \eqref{alsox_bilnearb}. Note that this DC approach is different compared to the one studied in \citet{hong2011sequential}, where the latter directly approximated the chance constraint using difference-of-convex functions.} Observe that the objective function in \eqref{alsox_bilnearb} can be rewritten as the difference of two convex quadratic functions: 
	\begin{align*}
	\E\left[{z}(\tilde{\rxi}){s}(\tilde{\rxi}) \right]=\frac{1}{4}\E\left[\left({z}(\tilde{\rxi})+{s}(\tilde{\rxi})\right)^2 - \left({z}(\tilde{\rxi})-{s}(\tilde{\rxi})\right)^2 \right].
	\end{align*}
	Next, the $\DC$ approach proceeds as follows. At iteration $k+1$, we replace $({z}({\rxi})-{s}({\rxi}))^2 $ by its first order Taylor approximation using the solution from the previous iteration, i.e.,
	\begin{align*}
	\left({z}({\rxi})-{s}({\rxi})\right)^2 \approx\left({z}^k({\rxi})-{s}^k({\rxi})\right)^2+2{z}^k({\rxi})\left({z}({\rxi})-{z}^k({\rxi})\right)-2{s}^k({\rxi})\left({s}({\rxi})-{s}^k({\rxi})\right),
	\end{align*}
	and solve the following convex program:
	\begin{equation}\label{basic_dc_1}
	\begin{aligned}
	\left(\bm{x}^{k+1},{s^{k+1}(\cdot),z^{k+1}(\cdot)}\right) \in \argmin _{ \bm{x}\in \mathcal{X},{s(\cdot),z({\cdot})}} &\biggl\{ \frac{1}{4}\E\left[\left({z}(\tilde{\rxi})+{s}(\tilde{\rxi})\right)^2 \right]-\frac{1}{4}\E\left[\left({z}^k(\tilde{\rxi})-{s}^k(\tilde{\rxi})\right)^2\right]\\
	&-\frac{1}{4}\E\left[2{z}^k(\tilde{\rxi})\left({z}(\tilde{\rxi})-{z}^k(\tilde{\rxi})\right)-2{s}^k(\tilde{\rxi})\left({s}(\tilde{\rxi})-{s}^k(\tilde{\rxi})\right)\right]\colon\\
	& \bm{c}^\top \bm{x} \leq t, g(\bm x,\tilde{\bm {\xi}})\leq {s}(\tilde{\rxi}), {s}(\tilde{\rxi})\geq 0, {z}(\tilde{\rxi})\in [0,1]\biggr\}.
	\end{aligned}
	\end{equation}
	And repeat this process until the objective functions $\{\E[z^{k}(\tilde{\rxi})s^{k}(\tilde{\rxi})]\}_{k\in \Ze_+}$ converge.
	
	Since both $\AM$ method and $\DC$ approach {find} a stationary point, the formal comparison between the $\AM$ method and the $\DC$ approach relies on their stationary conditions. Particularly, the $\AM$ method generates a stationary point {$(\bm{x}^{\AM},{s}^{\AM}(\cdot),{z}^{\AM}(\cdot)))$} that solves problem \eqref{basic_ap_z} and \eqref{basic_ap_1} when {$(\bm{x}^{k+1},{s}^{k+1}(\cdot))=(\bm{x}^{\AM},{s}^{\AM}(\cdot))$ and ${z}^k(\cdot)={z}^{\AM}(\cdot)$} if and only if its satisfies the following stationary condition:
	\begin{equation}
	\label{alsox_condition_1}
	\begin{aligned}
	&	{	\E\left[{s}^{\AM}(\cdot)\left[{z}(\cdot) - {z}^{\AM}(\cdot)\right] \right]\geq 0, 	\E\left[{z}^{\AM}(\cdot)\left[{s}(\cdot)- {s}^{\AM}(\cdot) \right] \right ]\geq 0,}
	\\
	&\forall \left(\bm x,{s(\cdot), z(\cdot)}\right) \text{ satisfies the constraints in the lower-level problem \eqref{alsox_bilnearb}}.
	\end{aligned}
	\end{equation}
	On the other hand, the $\DC$ method generates a stationary point {$(\bm{x}^{\DC},{s}^{\DC}(\cdot),{z}^{\DC}(\cdot))$} that solves \eqref{basic_dc_1} when {$(\bm{x}^k,{s}^k(\cdot),{z}^k(\cdot))=(\bm{x}^{\DC},{s}^{\DC}(\cdot),{z}^{\DC}(\cdot))$} if and only if its satisfies the following stationary condition:
	\begin{equation}
	\label{alsox_condition_2}
	\begin{aligned}
	&	{\E\left[{s}^{\DC}(\cdot)\left[{z}(\cdot) - {z}^{\DC}(\cdot)\right] \right]+\E\left[{z}^{\DC}(\cdot)\left[{s}(\cdot)- {s}^{\DC}(\cdot) \right] \right ]}\geq 0, \\
	&\forall \left(\bm x,{s(\cdot), z(\cdot)}\right) \text{ satisfies the constraints in the lower-level problem \eqref{alsox_bilnearb}}.
	\end{aligned}
	\end{equation}
	
	Note that the set of the stationary points satisfying the condition \eqref{alsox_condition_1} of the $\AM$ method is contained in that satisfying the condition \eqref{alsox_condition_2} of the $\DC$ approach. This concludes that the $\AM$ method is better than the $\DC$ approach.
	\begin{restatable}{proposition}{amdc}\label{am_dc} 
		Given $t$, when solving the lower-level problem \eqref{alsox_bilnearb}, the $\AM$ method can find a better solution than that of $\DC$ approach.
	\end{restatable}

	Then following example demonstrates that the solution from $\AM$ method can be indeed strictly better than that of the $\DC$ approach.
	
	
	\begin{example}
		\rm
		\label{am_dc_example_1}
		\begin{subequations}
			Let us revisit \Cref{am_dc_example}. Given $t=0.5$, then the lower-level problem \eqref{alsox_bilnearb} admits the following form
			\begin{align}
			\label{am_dc_example_am}
			\min_{\bm{x}\in\Re_+^2,\bm{s}\in\Re_+^3,\bm{z}}\left\{ \frac{1}{3}\sum_{i\in[3]}z_is_i\colon 
			\begin{array}{l}
			\displaystyle	2x_1+3x_2\geq 1-s_1, 2x_1+x_2\geq 1-s_2
			x_1+2x_2\geq 1-s_3, \\
			\displaystyle {\sum_{i\in[3]}z_i}\geq 2,\bm{z}\in [0,1]^3, x_1+x_2\leq 0.5 \end{array}
			\right\}.
			\end{align}
			When running $\AM$ \Cref{alg_am} with an initial solution $\bm{z}^0=[1,1,1]$,
			the stationary solution is $\bm{z}^\AM=[1,0,1]$, $ \bm{s}^{\AM}=[0,0.5,0]$, $\bm{x}^{\AM}=[0,0.5]$. 
			%
			When using $\DC$ approach with an initial solution $\bm{z}^0=[1,1,1]$, $\bm{s}^0=[1,1,1]$, the stationary solution is $\bm{z}^{\DC}=[0,1,1]$, $\bm{s}^{\DC}=[0,0.25,0.25]$, $\bm{x}^{\DC}=[0.25,0.25]$.
			More importantly, since $\sum_{i\in[3]}	\I(s^{\AM}_i=0)=2 \geq 2$ and $\sum_{i\in[3]}	\I(s^{\DC}_i=0)=1<2$, the $\AM$ method is able to find a feasible solution to the CCP for this example, while the $\DC$ approach is unable to.
		\end{subequations}
		\QEDB
	\end{example}

	\subsection{$\alsoxt +$ Algorithm}
	\label{alsox_am}
	This subsection integrates $\alsox$ \Cref{alg_alsox} with the $\AM$ \Cref{alg_am} as $\alsox+$ algorithm, to improve the performance of $\alsox$ \Cref{alg_alsox}. In $\alsox+$ algorithm, we first execute $\alsox$ \Cref{alg_alsox}, and when Step 4 of $\alsox$ \Cref{alg_alsox} encounters an infeasible solution (i.e., the optimal solution $(\bm{x}^*,{s^*(\cdot)})$ of the hinge-loss approximation \eqref{eq_also_x_2b} violates the chance constraint $\Pr\{\tilde{\bm\xi} \colon{{s}^*}(\tilde{\rxi})>0\} >\varepsilon$), then we run the $\AM$ \Cref{alg_am} with the same $t$ and see if we are able to find a feasible solution. If YES, we {further} decrease $t_U=t$; otherwise, we {increase} $t_L=t$.
	The detailed procedure for the $\alsox+$ algorithm is shown in \Cref{alg_alsox_+}.
	
	\begin{algorithm}[htbp]
		\caption{The Proposed $\alsox+$ Algorithm}
		\label{alg_alsox_+}
		\begin{algorithmic}[1]
			\State \textbf{Input:} Let $\delta_1$ denote the stopping tolerance parameter, $t_L$ and $t_U$ be the known lower and upper bounds of the optimal value of CCP \eqref{eq_ccp}, respectively 
			\While {$t_U-t_L>\delta_1$}
			\State	Let $t=(t_L+t_U)/2$ and $(\bm{x}^*,{s(\cdot)})$ be an optimal solution of the hinge-loss approximation \eqref{eq_also_x_2b}
			\State Let $t_U=t$ if $\Pr\{\tilde{\bm\xi} \colon {{s}^*}(\tilde{\rxi})>0\} \leq 1-\varepsilon$; otherwise, run the $\AM$ \Cref{alg_am}. If the solution output from the $\AM$ \Cref{alg_am} is feasible to the CCP, let $t_U=t$; otherwise, $t_L=t$
			\EndWhile
			\State \textbf{Output:} A feasible solution $\x^*$ {and its objective value} to CCP \eqref{eq_ccp}
		\end{algorithmic}
	\end{algorithm}
	
	We make the following remarks about $\alsox +$ \Cref{alg_alsox_+}.
	\begin{enumerate} [label=(\roman*)]
		\item We can use the solutions of the hinge-loss approximation \eqref{eq_also_x_2b} as warm-starts for the $\AM$ \Cref{alg_am};
		\item For the linear CCP, i.e., $g_i(\bm x,{\rxi}) ={\rxi}^\top\bm a_i(\bm x)-b_i(\bm x)$ with affine functions $\bm a_i(\bm x),b_i(\bm x)$ and set $\mathcal{X}$ is a polyhedron. Similar to $\alsox$ \Cref{alg_alsox}, we can use parametric linear programming \citep{adler1992geometric} to decrease the number of bisection needed. That is, after Step 3, we can record the lower and upper bounds of the allowable range of the value $t$. Then at Step 4,
		if the current solution is not feasible to CCP \eqref{eq_ccp}, we can update $t$ to be the upper bound of its allowable range; otherwise, we can let $t$ be equal to the lower bound of the allowable range. We then continue the binary search procedure;
		\item Inherited from $\alsox$ \Cref{alg_alsox}, $\alsox +$ \Cref{alg_alsox_+} always provides a better solution than that of the $\CVaR$ approximation given that Assumptions \ref{A_1}-\ref{A_2} hold and the tolerance $\delta_1=0$; and
		\item For the nonconvex set $\X$, \Cref{cvar_better_nonconvex}  can be further used to demonstrate that
		the 	$\CVaR$ approximation can also outperform $\alsox +$ \Cref{alg_alsox_+}.
		That is, in \Cref{cvar_better_nonconvex}, the solution output from $\alsox +$ \Cref{alg_alsox_+} is the same as that from $\alsox$, which is not optimal, while the $\CVaR$ approximation provides the optimal solution.
	\end{enumerate}

	Besides, incorporating $\AM$ \Cref{alg_am} in $\alsox +$ \Cref{alg_alsox_+} helps break the symmetry in the hinge-loss approximation \eqref{eq_also_x_2b} by assigning different weights to the violations of uncertain constraints. Specifically, in $\AM$ \Cref{alg_am}, when fixing ${z(\cdot)=z^{k}(\cdot)}$ with $\E[z^{k}(\tilde{\rxi})]=1-\varepsilon$, the problem \eqref{basic_ap_1} tends to focus on the $1-\varepsilon$ portion of uncertain constraints rather than using all of them in the hinge-loss approximation \eqref{eq_also_x_2b}. {In the following example, we show that due to the symmetry of the random parameters, $\alsox$ \eqref{eq_also_x_2} is unable to provide an optimal solution, while $\alsox +$ \Cref{alg_alsox_+} with the tolerance $\delta_1=0$ can.} 
	
	\begin{example}
		\rm
		\label{am_better_dc_example}
		Let us revisit \Cref{am_dc_example}. Suppose that $t=0.6$, then the optimal solution provided by $\alsox$ \eqref{eq_also_x_2} is $x_1^*=x_2^*=0.3$, $s_1^*=0$, $s_2^*=s_3^*=0.1$, which violates the chance constraint. Invoking the $\AM$ \Cref{alg_am} with initial $s^0_i=s_i^*$ for each $i\in [3]$, at the second iteration of the $\AM$ \Cref{alg_am}, we {find} an optimal solution $x^2_1=0.4, x^2_2=0.2$, $ s^2_1=s^2_2=0$, $s^2_3=0.2$, and $z^2_1=z^2_2=1$, $z^2_3=0$ to the lower-level problem \eqref{alsox_bilnearb} with $t=0.6$. 
		Thus, if $t=0.6$, $\alsox +$ \Cref{alg_alsox_+} {further} reduce the $t_U=t=0.6$. In fact, in this example, $\alsox +$ \Cref{alg_alsox_+} with the tolerance $\delta_1=0$ {finds} the optimal solution of the CCP.
		\QEDB
	\end{example}

	Although improving $\alsox$ \eqref{eq_also_x_2}, $\alsox +$ \Cref{alg_alsox_+} might not always be able to find an optimal solution of CCP \eqref{eq_ccp}, as illustrated in \Cref{alsox_not_exact_example} of Appendix~\ref{sec_append_example}.
	{Interested readers are referred to Appendix~\ref{Illustration_section} for an illustration of comparisons among $\alsox$, $\CVaR$ approximation, and $\alsox+$ Algorithm.}
	
	\section{Extension to Distributionally Robust Chance Constrained Programs (DRCCP) with Wasserstein Distance}
	\label{drccp}
	In practice, the distributional information of random parameters $\tilde{\bm\xi}$ might not be fully known, making it difficult to commit to a single known probability distribution $\Pr$. Under this circumstance, to hedge against distributional ambiguity, we consider the distributionally robust chance constrained programs (DRCCPs), which require the chance constraint to be satisfied for all the probability distributions from a family of distributions, termed ``ambiguity set." That is, {following many recent works in DRCCP \citep{xie2020bicriteria,xie2019distributionally,chen2018data,ji2020data,chen2019sharing},} we consider the DRCCP of the form
	\begin{equation}
	\label{eq_drccp_org}
	\min _{\bm {x}\in \mathcal{X}} \left\{\bm{c}^\top\bm{x}\colon \inf_{\Pr\in \mathcal{P}} \Pr\left\{\tilde{\bm\xi}\colon g(\bm x,\tilde{\bm\xi})\leq 0\right\} \geq1-\varepsilon \right\},
	\end{equation}
	where ambiguity set $ \mathcal{P}$ denotes a subset of probability distributions $\Pr$ defined on ${(\Omega,\mathcal{F})}$ and induced by the random parameters $\tilde{\rxi}$, and risk level $\varepsilon \in (0,1)$. {Specifically, we study the DRCCP under $\infty-$Wasserstein ambiguity set (see, e.g., \citealt{bertsimas2018data,xie2020tractable})}. The $q-$Wasserstein ambiguity set is defined as
	\begin{equation*}
	\mathcal{P}_q^W =\left\{ \Pr\colon\Pr\left\{ \tilde{\rxi}\in {\Xi}\right\}=1,W_q(\Pr,\Pr_{\tilde\rzeta})\leq \theta \right\},
	\end{equation*}
	where for any $q\in[1,\infty]$, the $q-$Wasserstein distance is defined as 
	\begin{equation*}
	W_q(\Pr_1,\Pr_2)=\inf\left\{\left[ \int_{{\Xi}\times{\Xi}}\left\|\bm{\xi}_1-\bm{\xi}_2\right\|^q\mathbb{Q}(d\bm{\xi}_1,d\bm{\xi}_2)\right ]^\frac{1}{q}\colon
	\begin{aligned}
	& \mathbb{Q} \text{ is a joint distribution of } \tilde{\rxi}_1 \text{ and } \tilde{\rxi}_2\\
	& \text{ with marginals }\Pr_1 \text{ and } \Pr_2, \text{ respectively }
	\end{aligned}
	\right\},
	\end{equation*}
	{where $\theta\geq0$ is the Wasserstein radius, and $\Pr_{\trzeta}$ denotes the reference distribution induced by random parameters $\trzeta$. For example, $\Pr_{\trzeta}$ can be an empirical distribution with $\trzeta$ being a uniform discrete random vector.}
	Note that if $q=\infty$, the $\infty-$Wasserstein  distance is reduced to
	\begin{equation*}
	W_\infty(\Pr_1,\Pr_2)=\inf\left\{ \rm{ess.sup}\left\|\bm{\xi}_1-\bm{\xi}_2\right\|\mathbb{Q}(d\bm{\xi}_1,d\bm{\xi}_2)\colon
	\begin{aligned}
	& \mathbb{Q} \text{ is a joint distribution of } \tilde{\rxi}_1 \text{ and } \tilde{\rxi}_2\\
	& \text{ with marginals }\Pr_1 \text{ and } \Pr_2, \text{ respectively }
	\end{aligned}
	\right\}.
	\end{equation*}
	%
	
	Throughout this section, we assume that
	\begin{enumerate}[label={A\arabic*}]
		\setcounter{enumi}{2}
		\item \label{A_4} The reference distribution $\Pr_{\tilde{\rzeta}}$ is sub-Gaussian, that is, $\Pr_{\tilde{\rzeta}}\{\tilde{\rzeta}: \|\tilde{\rzeta}\|\geq t\}\leq C_1\exp(-C_2t^2)$ for some positive constants $C_1,C_2$.
	\end{enumerate}
	It is worthy of noting that 
	the sub-Gaussian assumption ensures the weak compactness of $\infty-$Wasserstein ambiguity set and thus ensures the strong duality of reformulating the worst-case expectation under $\infty-$Wasserstein ambiguity set. Particularly, this paper mainly focuses on empirical or elliptical reference distributions, which clearly satisfy Assumption~\ref{A_4}.
	
	Under this setting, DRCCP \eqref{eq_drccp_org} can be written as
	\begin{equation}
	\label{eq_drccp}
	v_\infty^* :=\min _{\bm {x}\in \mathcal{X}} \left\{\bm{c}^\top\bm{x}\colon \inf_{\Pr\in \mathcal{P}_\infty^W} \Pr\left\{\tilde{\bm\xi}\colon g_i(\bm x,\tilde{\bm\xi})\leq 0, \forall i\in[I]\right\} \geq 1-\varepsilon \right\}.
	\end{equation}
	{It turns out that DRCCP \eqref{eq_drccp} admits a neat equivalent representation.}
	{\begin{restatable}{proposition}{theinfinitydrccp}\label{the_infinity_drccp} 
			Under $\infty-$Wasserstein ambiguity set, DRCCP  \eqref{eq_drccp} is equivalent to 
			\begin{align}
			v_\infty^*  =\min _{\bm {x}\in \mathcal{X}} \left\{\bm{c}^\top\bm{x}\colon
			\Pr_{\tilde\rzeta}\left\{{\tilde\rzeta}\colon \bar g_i(\bm x,\trzeta)\leq 0, \forall i \in [I] \right\}\geq 1-\varepsilon \right\},
			\label{eq_drccp_infty}
			\end{align}
			where the convex and lower semi-continuous function $\bar g_i:\Re^n\times\Xi\to\Re$ is defined as $\bar g_i(\bm x,\rzeta):=\max_{{\rxi}}\{g_i(\bm x,\rxi)\colon \|{\rxi}-\rzeta\|\leq \theta\}$  for each $i\in[I]$.
		\end{restatable}
		\proof
		See Appendix~\ref{proof_the_infinity_drccp}.
		\QEDA}
	
	{The reformulation in \Cref{the_infinity_drccp} implies that  DRCCP \eqref{eq_drccp} under $\infty-$Wasserstein ambiguity set is equivalent to a regular CCP \eqref{eq_drccp_infty}.
		In fact, we anticipate that  the worst-case $\alsox$ and the worst-case $\CVaR$ approximation under	$\infty-$Wasserstein ambiguity set are equivalent to directly applying $\alsox$ and  $\CVaR$ approximation to solve CCP \eqref{eq_drccp_infty}. This observation motivates us to show that under $\infty-$Wasserstein ambiguity set, the worst-case $\alsox$ is better than the worst-case $\CVaR$ approximation, which {is} detailed in the next subsection.}

	
	\subsection{The Worst-case $\alsoxt$ Outperforms the Worst-case $\CVaRt$ Approximation}
	{In this subsection, we introduce the notions of the worst-case $\alsox$ and the worst-case $\CVaR$ approximation, and then demonstrate that the worst-case $\alsox$ always outperforms the worst-case $\CVaR$ approximation under $\infty-$Wasserstein ambiguity set.} 
	
	{
		Similar to $\alsox$ \eqref{eq_also_x_2}, we derive the worst-case $\alsox$ counterpart under $\infty-$Wasserstein ambiguity set. 
		That is, in the worst-case $\alsox$, we first solve the worst-case hinge-loss approximation, which is to minimize the least-favorable expectation of function $\max_{i\in [I]}g_i(\bm x,\tilde{\bm\xi})_+$, i.e., the objective function of \eqref{drccp_alsox} is to minimize the worst-case objective function of the hinge-loss approximation \eqref{eq_also_x_2b},  and check if its optimal solution $\bm{x}^*$ satisfies the distributionally robust chance constraint in \eqref{eq_drccp} or not.}
	If the answer is YES, we {continue} reducing the upper bound of the objective value $t$, and otherwise, we {increase} $t$. This procedure can be formally formulated as
	\begin{subequations}
		\label{drccp_alsox_formulation}
		\begin{align}
		v_\infty^A =\min _{ {t}}\quad & t,\\
		\text{s.t.}\quad&\bm x^*\in\argmin _{\bm {x}\in \mathcal{X}}\, \sup _{\Pr\in\mathcal{P}_\infty^W} \left\{ \E_\Pr\left[\max_{i\in [I]}g_i(\bm x,\tilde{\bm\xi})_+\right]\colon \bm{c}^\top \bm{x} \leq t \right\},\label{drccp_alsox}\\
		&  \inf_{\Pr\in \mathcal{P}_\infty^W} \Pr\left\{\tilde{\bm\xi}\colon g_i(\bm x^*,\tilde{\bm\xi})\leq 0, \forall i\in[I]\right\} \geq 1-\varepsilon.\label{alsox_drccp_formualtion}
		\end{align}
	\end{subequations} 
	{For DRCCP \eqref{eq_drccp}, the worst-case $\CVaR$ approximation is defined as}
	\begin{equation}
	\label{worse_case_cvar_q}
	v_\infty^\CVaR=\min _{\bm {x}\in \mathcal{X}}\left\{\bm c^\top\bm x\colon \sup _{{\Pr}\in\mathcal{P}_\infty^W} \inf_{\beta}\left[ \beta+ \frac{1}{\varepsilon}\E_{\Pr}\left[ \max_{i\in[I]}\left(  g_i(\bm x,\tilde{\bm\xi})-\beta \right)_+ \right] \right] \leq 0 \right\}.
	\end{equation}
	{The next proposition formally derives the equivalent reformations of the worst-case $\alsox$ \eqref{drccp_alsox_formulation} and  the worst-case $\CVaR$ approximation \eqref{worse_case_cvar_q}, respectively.}
	{\begin{restatable}{proposition}{theinfinityalsox}\label{the_infinity_alsox} 
			Under $\infty-$Wasserstein ambiguity set, we have
			\begin{enumerate}[label=(\roman*)]
				\item the worst-case $\alsox$ \eqref{drccp_alsox_formulation} is equivalent to 
				\begin{align}
				v_\infty^A =\min _{ {t}}\quad & t\nonumber,\\
				\textup{s.t.}\quad&\bm x^*\in\argmin _{\bm {x}\in \mathcal{X}}\left\{ \E_\Pr\left[\max_{i\in [I]}\bar g_i(\bm x,\tilde{\bm\xi})_+\right]\colon \bm{c}^\top \bm{x} \leq t \right\}, \label{drccp_alsox_infity_eq} \\
				& 	\Pr_{\tilde\rzeta}\left\{{\tilde\rzeta}\colon \bar g_i(\bm x^*,\trzeta)\leq 0, \forall i \in [I] \right\}\geq 1-\varepsilon; \nonumber
				\end{align}
				\item the worst-case $\CVaR$ approximation \eqref{worse_case_cvar_q} is equivalent to  
				\begin{align}
				v_\infty^\CVaR=\min _{\bm {x}\in \mathcal{X}}\left\{\bm c^\top\bm x\colon \min_{\beta}\left[ \beta+\frac{1}{\varepsilon} \E_{\Pr_{\trzeta}}\left\{\max_{i\in[I]}	  \left\{\bar g_i(\bm x,\trzeta)\right\}-\beta\right\}_+\right]\leq 0 \right\}.
				\label{drccp_cvar_infity_eq} 
				\end{align}
			\end{enumerate}
		\end{restatable}
		\proof
		See Appendix~\ref{proof_the_infinity_alsox}.
		\QEDA}
	
	
	{We remark that both the worst-case $\alsox$ \eqref{drccp_alsox_infity_eq} and the worst-case $\CVaR$ approximation \eqref{worse_case_cvar_q} under $\infty-$Wasserstein ambiguity set can be interpreted as applying $\alsox$ and $\CVaR$ approximation of the regular CCP \eqref{eq_drccp_infty}, respectively. 
		Therefore, the results in previous sections hold for DRCCP \eqref{eq_drccp_infty} including that we can simply apply $\alsox+$ to CCP \eqref{eq_drccp_infty}. More importantly,  following the spirit of \Cref{alsox_is_better}, we can conclude that the worst-case $\alsox$ is better than the worst-case $\CVaR$ approximation under $\infty-$Wasserstein ambiguity set. }

	\begin{restatable}{theorem}{theinfinitytheorem}\label{the_infinity} 
		For DRCCP with $\infty-$Wasserstein ambiguity set, the worst-case $\alsox$ outperforms the worst-case $\CVaR$ approximation.
	\end{restatable}
	\proof
	See Appendix~\ref{proof_the_infinity}.
	\QEDA
	
	{We also remark that under some additional assumptions of the functions $\{g_i(\cdot,\cdot)\}$, 
		their robust counterparts $\{\bar g_i(\cdot,\cdot)\}$ admit simple representations. Interested readers are referred to the work \citep{ben2009robust} for many different choices of  functions $\{g_i(\cdot,\cdot)\}$. Below, we list two classes of functions:
		\begin{enumerate}[label=(\roman*)]
			\item 
			When the functions are bi-affine, i.e., $g_i(\bm x,{\rxi}) ={\rxi}^\top\bm a_i(\bm x)-b_i(\bm x)$ with affine functions $\bm a_i(\bm x),b_i(\bm x)$ for each $i\in [I]$, we have
			\begin{align*}
			\bar g_i(\bm x,\rzeta)  = \theta\left\|\bm{a}_i^\top(\bm{x})\right\|_*+\tilde\rzeta ^\top\bm{a}_i(\bm{x})-b_i(\bm{x}),\forall i\in [I].
			\end{align*}
			Note that the bi-affinity assumption of $\{g_i(\cdot,\cdot) \}_{i\in [I]}$ has been commonly used in many DRCCP literature (see, e.g., \citealt{hanasusanto2015distributionally,hanasusanto2017ambiguous,xie2018deterministic,xie2019distributionally}).
			\item When the norm is $L_\infty$ (i.e., $\|\cdot\|=\|\cdot\|_\infty$) and the function $g_i(\bm x,{\rxi}) $ is monotone non-decreasing in $\rxi$ for any $\bm{x}\in \X$ and $i\in [I]$, we have 
			\begin{align*}
			\bar g_i(\bm x,\rzeta)  = g_i(\bm x,\rzeta+\theta\bm e),\forall i\in [I].
			\end{align*}
			This monotonicity structure has been studied in  the recent works \citep{zhang2021robust,xie2020tractable,chen2020regret}.
		\end{enumerate}
	}
	
	
	\subsection{Exactness of the Worst-case $\alsoxt$}
	Similar to \Cref{exact_el}, we are able to identify two sufficient conditions under which the worst-case $\alsox$ \eqref{drccp_alsox_formulation} can provide an optimal solution to DRCCP \eqref{eq_drccp}. Particularly, we consider the single DRCCP and elliptical reference distribution with the following condition.
	\begin{restatable}{proposition}{exacteldrccp}\label{exact_el_drccp} 
		{Suppose that the reference distribution $\Pr_{\tilde\rzeta}$ is elliptical $\Pr_{\mathrm{E}}(\bm{\mu},\bm{\mathrm{{\Sigma}}},\hat g)$, and the norm defining the Wasserstein distance is the Mahalanobis norm associated with the positive definite matrix $\bm{\mathrm{{\Sigma}}}$, i.e., $\left\|\bm{y}\right\|=\sqrt{\bm{y}^\top\bm{\mathrm{{\Sigma}}}^{-1}\bm{y}}$, for some $\bm y\in \Re^n$. If $I=1$ and the random function $g_1(\bm x,{\rxi}) ={\rxi}^\top\bm a_1(\bm x)-b_1(\bm x)$, then} the
		worst-case $\alsox$ \eqref{drccp_alsox_formulation} provides an optimal solution to DRCCP \eqref{eq_drccp} under $\infty-$Wasserstein ambiguity set if (i) $\mathcal{X}\subseteq\{\bm x:\sqrt{\bm{a}_1(\bm{x})^\top\bm{\mathrm{{\Sigma}}}\bm{a}_1(\bm{x})}=C\} $, where $C$ is a positive constant; or (ii) $\mathcal{X}\subseteq\{\bm x:b_1(\bm x)-\bm\mu^\top \bm{a}_1(\bm {x})=C\} $, where $C$ is an arbitrary constant.
	\end{restatable}
	\proof
	See Appendix~\ref{proof_exact_el_drccp}.
	\QEDA

	\section{Numerical Illustrations}
	\label{numerical_study}
	In this section, we conduct numerical studies to demonstrate the efficacy of the proposed methods. We evaluate the differences among $\CVaR$ approximation, $\alsox$, and $\alsox +$. 
	To evaluate the effectiveness of the proposed algorithms, we use ``\textrm{Improvement}'' to denote the percentage of differences between the value of a proposed algorithm and $\CVaR$ approximation, i.e.,
	\begin{align*}
	\textrm{Improvement} (\%) = \frac{ \CVaR \textrm{ approximation value}-\textrm{value of a proposed algorithm} }{|\CVaR \textrm{ approximation value}|}\times 100.
	\end{align*}
	
	All the instances in this section are executed in Python 3.6 with calls to solver Gurobi (version 8.1.1 with default settings) on a personal PC with 1.6 GHz Intel Core i5 processor and 8G of memory. We set the time limit of each instance to be 3600s. Codes of the numerical experiments are available at \url{https://github.com/jnan97/ALSO-X}.
	
	Now, we compare the performances of $\CVaR$ approximation, $\alsox$, and $\alsox +$ of solving the regular CCP with linear and nonlinear uncertain constraints. Particularly, we consider the number of data points $N=400,600,1000$, the risk level $\varepsilon=0.05, 0.1$, and the dimension of decision variables $n=20,40,100$. For each parametric setting, we generate $5$ random instances and report their average performance. 
	
	\noindent\textbf{Testing a Linear CCP. }
	Let us first consider the following linear CCP
	\begin{align*}
	v^*=\min _{\bm{x}} &\left\{ \bm{c}^\top \bm{x}\colon \bm{x}\in [0,1]^n, \frac{1}{N}\sum_{i\in[N]} \one \left[ \sum_{j\in[n]} \xi^i_j x_j\leq 100 \right] \geq 1-\varepsilon\right\}.
	\end{align*}
	Above, we generate the samples $\{\rxi^i\}_{i\in [N]}$ by assuming that the random parameters $\tilde{\rxi}$ are discrete and i.i.d. uniformly distributed between $1$ and $50$. We set $\delta_1=10^{-2}$ and $\delta_2=10^{-2}$ in $\alsox$ \Cref{alg_alsox} and $\alsox +$ \Cref{alg_alsox_+}. For each random instance, we generate the cost vector $\bm{c}$ as a random integer one with each entry uniformly distributed between $-10$ and $-1$. Since we have proven in \Cref{the1} that $\alsox$  {delivers} a better solution than that of the $\CVaR$ approximation, in $\alsox$ \Cref{alg_alsox} and $\alsox +$ \Cref{alg_alsox_+}, we use the optimal value from $\CVaR$ approximation as an initial upper bound $t_U$, and the quantile bound from \cite{ahmed2017nonanticipative,song2014chance} as an initial lower bound $t_L$. Besides, at each bisection iteration, we also use the upper bound of allowable increase or lower bound of allowable decrease of the current value $t$ to update its value for the next iteration. The numerical results for this linear CCP are displayed in \Cref{tab_linear_single}. We see that {although the computation time of $\alsox$ is longer than that of $\CVaR$ approximation,} $\alsox$ \Cref{alg_alsox} can be solved within seconds and enhance the solution of $\CVaR$ approximation by around 4-8\% improvement. The performance of $\alsox +$ \Cref{alg_alsox_+} is even more striking, which can improve the solution-quality of $\CVaR$ approximation by around 5-10\%. This demonstrates the correctness and effectiveness of our proposed algorithms. 
	
	

	\begin{table}[htbp]
		\centering
		\caption{Numerical Results for the Linear CCP}
		\renewcommand{\arraystretch}{1} 
		\label{tab_linear_single}
		\tiny
		\begin{center}
			\begin{tabular}{ c c  r  r r r r| r  r r r r }
				\hline
				\multicolumn{1}{c}{ \multirow{3}*{$N$} }&
				\multicolumn{1}{c}{ \multirow{3}*{$n$} }&
				\multicolumn{5}{c|}{$\varepsilon=0.05$} & \multicolumn{5}{c}{$\varepsilon=0.10$} \\
				\cline{3-12}
				& &	\multicolumn{1}{c}{$\CVaR$} &
				\multicolumn{2}{c}{$\alsox$} &
				\multicolumn{2}{c|}{$\alsox +$}& 	
				\multicolumn{1}{c}{$\CVaR$} &
				\multicolumn{2}{c}{$\alsox$} &
				\multicolumn{2}{c}{$\alsox +$}\\
				\cline{3-12}
				& &Time (s)&Time (s)& \makecell{Improve-\\ment (\%)} &Time (s)& \makecell{Improve-\\ment (\%)}&Time (s)&Time (s)& \makecell{Improve-\\ment (\%)} &Time (s)& \makecell{Improve-\\ment (\%)} \\
				\hline
				\multirow{3}{*}{400} 
				&20 & 0.04& 0.52&7.00 & 3.21&8.43& 0.04& 0.44&8.27 & 2.97&9.04\\
				\cline{2-12}
				&40 & 0.08& 0.87&4.66 & 6.94&6.27& 0.04& 0.71&6.85& 7.09&7.76\\
				\cline{2-12}
				&100 &0.19& 1.87&4.02& 22.02&5.38& 0.12 & 2.03 &4.95& 22.91&6.05 \\
				\cline{1-12}
				\multirow{3}{*}{600} &20 & 0.07&0.62 & 6.22 & 5.27&7.29& 0.04& 0.73&7.76& 4.04&8.39 \\
				\cline{2-12}
				&40 & 0.10& 1.29&5.36& 11.43&6.28 & 0.04& 1.09&6.56& 10.72&7.46\\
				\cline{2-12}
				&100 & 0.23& 3.17&3.20& 35.34&4.53& 0.13& 3.31&4.75& 36.84&5.63\\
				\cline{1-12}
				\multirow{3}{*}{1000} 	&20& 0.12& 1.07&5.86& 9.24&7.10 & 0.06&1.04&7.45& 7.75&8.33\\
				\cline{2-12}
				&40 & 0.21& 2.03 &4.65& 21.45&5.53& 0.16& 2.25 &5.84& 21.98&6.36\\
				\cline{2-12}
				&100 & 0.29& 4.95& 4.04&60.34&4.96& 0.34& 5.30&5.36& 60.58&6.62 \\
				\hline 
			\end{tabular}
		\end{center}
	\end{table}

	\noindent\textbf{Testing a Nonlinear CCP. }
	Following \cite{xie2018quantile, hong2011sequential, sun2014asymptotic}, let us consider the following chance constrained quadratic optimization problem as
	\begin{align*}
	v^*=\min _{\bm{x}} &\left\{ \bm{c}^\top \bm{x}\colon \bm{x}\in [0,1]^n, \frac{1}{N}\sum_{i\in[N]} \one \left[ \sum_{j\in[n]} \xi^i_j x_j^2\leq 100 \right] \geq 1-\varepsilon\right\}.
	\end{align*}
	Above, we generate the samples $\{\rxi^i\}_{i\in [N]}$ by assuming that the random parameters $\tilde{\rxi}$ are discrete and i.i.d. uniformly distributed between $1$ and $99$. We set $\delta_1=10^{-2}$ and $\delta_2=10^{-2}$ in $\alsox$ \Cref{alg_alsox} and $\alsox +$ \Cref{alg_alsox_+}. For each random instance, we generate its cost vector $\bm{c}$ as a random integer one with each entry uniformly distributed between $-10$ and $-1$. 
	For this nonlinear CCP, we run SD \Cref{alg_SD} to solve the hinge-loss approximation in \Cref{alg_alsox} as well as to solve the problem \eqref{basic_ap_1}, while we use Gurobi to directly solve the $\CVaR$ approximation. Note that we set the maximum number of iterations to be $50$. The $\CVaR$ approximation is time-consuming, in $\alsox$ \Cref{alg_alsox} and $\alsox +$ \Cref{alg_alsox_+}, we use the greedy method to find a feasible solution as an initial upper bound $t_U$, and the quantile bound from \cite{ahmed2017nonanticipative} as an initial lower bound $t_L$. 
	The numerical results are shown in \Cref{tab_nonlinear_single}. Notably, for this nonlinear CCP, we see that both $\alsox$ \Cref{alg_alsox} and $\alsox +$ \Cref{alg_alsox_+} provide better solutions than the $\CVaR$ approximation, and they are faster than the $\CVaR$ approximation, especially when the dimension of decision variables increases. {This might be because the off-the-shelf solvers often struggle in solving the large-scale second-order conic programs and the first-order method, on the contrary, is more effective given that the projection is relatively easy.}
	
	To demonstrate the effectiveness of our proposed method, we numerically compare the proposed $\alsox+$ \Cref{alg_alsox_+}  with the exact Big-M method. Interested readers are referred to Appendix~\ref{numerical_big_m} for the detailed numerical results, where Big-M method is often unable to find a better solution than $\alsox+$ especially when the dimension increases and $\alsox+$ can consistently find near-optimal solutions or even optimal solutions.

	\begin{table}[htbp]
		\centering
		\caption{Numerical Results for the Nonlinear CCP}
		\setlength{\tabcolsep}{2pt} 
		\renewcommand{\arraystretch}{1} 
		\label{tab_nonlinear_single}
		\tiny
		\begin{center}
			\begin{tabular}{ c c  r  r r r r| r  r r r r }
				\hline
				\multicolumn{1}{c}{ \multirow{3}*{$N$} }&
				\multicolumn{1}{c}{ \multirow{3}*{$n$} }&
				\multicolumn{5}{c|}{$\varepsilon=0.05$} & \multicolumn{5}{c}{$\varepsilon=0.10$} \\
				\cline{3-12}
				& &	\multicolumn{1}{c}{$\CVaR$} &
				\multicolumn{2}{c}{$\alsox$} &
				\multicolumn{2}{c|}{$\alsox +$}& 	
				\multicolumn{1}{c}{$\CVaR$} &
				\multicolumn{2}{c}{$\alsox$} &
				\multicolumn{2}{c}{$\alsox +$}\\
				\cline{3-12}
				& &Time (s)&Time (s)& \makecell{Improve-\\ment (\%)} &Time (s)& \makecell{Improve-\\ment (\%)}&Time (s)&Time (s)& \makecell{Improve-\\ment (\%)} &Time (s)& \makecell{Improve-\\ment (\%)} \\
				\hline
				\multirow{3}{*}{400} 
				&20 & 3.40& 2.40&2.23 &16.74 &2.60 & 1.64& 2.43&3.02 & 13.29&3.41\\
				\cline{2-12}
				&40 & 6.90& 2.61&1.99 & 16.84 &2.81	& 4.58& 3.02&2.27 & 16.53&2.72 \\
				\cline{2-12}
				&100 &43.14& 4.47&1.42&31.68 &1.93 & 55.32& 3.22&1.55& 25.25&1.90 \\
				\cline{1-12}
				\multirow{3}{*}{600} &20 & 4.22& 4.74&2.39& 17.61&2.84& 2.99& 3.55&2.89& 22.08&3.07 \\
				\cline{2-12}
				&40 & 11.57& 4.93&1.84&18.39 &2.03 & 8.94& 5.24&2.30& 22.88&2.62\\
				\cline{2-12}
				&100 & 70.85& 5.10&1.18& 28.21 &1.71 &68.44& 5.92&1.50& 31.60&1.78 \\
				\cline{1-12}
				\multirow{3}{*}{1000} 	&20& 7.81& 5.24&2.45& 20.24&2.61& 7.49& 7.75&2.84& 34.48&2.96\\
				\cline{2-12}
				&40 & 30.72& 6.94 &2.05& 21.74&2.21&21.84& 20.07 &2.24& 47.55&2.36\\
				\cline{2-12}
				&100 & 170.28& 6.89&1.30&33.50&1.54 & 130.08& 51.41&1.59& 63.30&1.73 
				\\
				\hline 
			\end{tabular}
		\end{center}
	\end{table}

	\noindent\textbf{Covering CCPs: Comparisons Between Relax-and-Scale Algorithm in \cite{xie2020bicriteria} and the Proposed Algorithms. }
	Although we have proven in \Cref{thm_ccp} that for the finite-support covering CCP, the proposed $\alsox$ \Cref{alg_alsox} has the same worst-case approximation ratio as the relax-and-scale algorithm (see, e.g., algorithm 2 in \citealt{ahmed2018relaxations} or algorithm 1 in \citealt{xie2020bicriteria}).
	In this subsection, we  numerically {compare} these two algorithms as well as the $\CVaR$ approximation and $\alsox +$ \Cref{alg_alsox_+}.
	
	We consider the following covering CCP as
	\begin{align*}
	v^*=\min _{\bm{x}} &\left\{ \bm{c}^\top \bm{x}\colon \bm{x}\in [0,1]^n, \frac{1}{N}\sum_{i\in[N]} \one \left[ \sum_{j\in[n]} \xi^i_j x_j\geq 40 \right] \geq 1-\varepsilon\right\}.
	\end{align*}
	Above, we generate the samples $\{\rxi^i\}_{i\in [N]}$ by assuming that the random parameters $\tilde{\rxi}$ are discrete and i.i.d. uniformly distributed between $1$ and $50$. We set $\delta_1=10^{-2}$ and $\delta_2=10^{-2}$ in the $\alsox$ \Cref{alg_alsox} and $\alsox +$ \Cref{alg_alsox_+}. For each random instance, we generate the cost vector $\bm{c}$ as a random integer one with each entry uniformly distributed between $1$ and $10$. In the $\alsox$ \Cref{alg_alsox} and $\alsox +$ \Cref{alg_alsox_+}, the continuous relaxation bound of covering CCP \eqref{covering_relax} is set as the initial lower bound $t_L$, and the approximation bound $(\floor{N\varepsilon}+1)t_L $ is set as the initial upper bound $t_U$. At each bisection iteration, we also incorporate the upper bound of allowable increase or lower bound of allowable decrease of the current value of $t$ to update its value in the next iteration. The numerical results are displayed in \Cref{tab_covering_comparison}. We see that the proposed $\alsox$ \Cref{alg_alsox} and $\alsox +$ \Cref{alg_alsox_+} are better than Relax-and-Scale algorithm in \citealt{ahmed2018relaxations} in terms of solution quality, while all the three algorithms dominate the results from the $\CVaR$ approximation.

	\begin{table}[htbp]
		\centering
		\caption{Numerical Result for Covering CCP}
		\setlength{\tabcolsep}{1pt} 
		\renewcommand{\arraystretch}{1} 
		\label{tab_covering_comparison}
		\tiny
		\begin{center}
			\begin{tabular}{ c c  r  r r r r r r| r  r r r r r r }
				\hline
				\multicolumn{1}{c}{ \multirow{3}*{$N$} }&
				\multicolumn{1}{c}{ \multirow{3}*{$n$} }& 
				\multicolumn{7}{c|}{$\varepsilon=0.05$} & \multicolumn{7}{c}{$\varepsilon=0.10$} 	\\
				\cline{3-16}
				& & 	\multicolumn{1}{c}{$\CVaR$} &
				\multicolumn{2}{r}{ \makecell{$\textrm{Relax-and-Scale}$\\ Algorithm} } &
				\multicolumn{2}{c}{$\alsox$} &
				\multicolumn{2}{c|}{$\alsox +$}&
				\multicolumn{1}{c}{$\CVaR$} &
				\multicolumn{2}{r}{ \makecell{$\textrm{Relax-and-Scale}$\\ Algorithm} } &
				\multicolumn{2}{c}{$\alsox$} &
				\multicolumn{2}{c}{$\alsox +$} \\
				\cline{3-16}
				\multicolumn{1}{r}{} & &Time (s)&Time (s)& \makecell{Improve-\\ment (\%)} &Time (s)& \makecell{Improve-\\ment (\%)} &Time (s)& \makecell{Improve-\\ment (\%)} &Time (s)&Time (s)& \makecell{Improve-\\ment (\%)} &Time (s)& \makecell{Improve-\\ment (\%)} &Time (s)& \makecell{Improve-\\ment (\%)}\\
				\hline
				\multirow{3}{*}{400} 
				&20 & 0.02 & 1.17 & 14.57 & 0.43 & 16.45 & 2.11 & 18.78 & 0.01 & 0.63 & 16.45 & 0.42 & 18.57 & 2.21 & 19.85 \\
				\cline{2-16}
				&40 & 0.02 &1.50 & 15.65 & 1.01 & 17.54 & 7.00 & 20.36 & 0.02 & 1.06 & 15.01 & 1.07 & 15.29 & 6.73 & 16.68\\
				\cline{2-16}
				&100 &0.06 & 2.79 & 7.33 & 2.44 & 10.48 & 18.53 & 12.10&0.05 & 2.33 & 9.86 & 2.01 & 11.44 & 15.59 & 13.05 \\
				\cline{1-16}
				\multirow{3}{*}{600} &20 &0.03 & 1.13 & 17.17 & 0.61 & 18.89& 4.48 & 20.22& 0.02 & 1.23 & 18.31 & 0.60 & 19.11 & 4.42 & 19.57\\
				\cline{2-16}
				&40 & 0.04 & 1.55 & 13.59 & 1.10 & 13.30& 7.58 & 15.50 & 0.03 & 1.58 & 15.58 & 1.11 & 16.16 & 7.41 & 17.82\\
				\cline{2-16}
				&100 & 0.09 & 3.26 & 7.83 & 2.88 & 10.80 & 19.22 & 11.55& 0.07 & 3.00 & 9.20 & 2.43 & 10.77 & 20.88 & 11.69 \\
				\cline{1-16}
				\multirow{3}{*}{1000} &20& 0.06 & 2.75 & 14.52 & 0.95 & 16.28 & 10.92 & 18.16 & 0.05 & 2.49 & 20.08 & 0.99 & 19.42 & 7.51 & 19.94\\
				\cline{2-16}
				&40 & 0.12 & 4.01 & 10.40 & 2.28 & 13.47 & 17.65 & 14.53 & 0.09 & 3.99 & 16.75 & 2.15 & 15.82 & 18.98 & 16.68\\
				\cline{2-16}
				&100 &0.26 & 6.42 & 9.65 & 5.76 & 10.12& 44.97 & 11.27&0.18 & 6.02 & 10.60 & 4.90 & 11.05 & 39.58 & 11.86\\
				\hline
			\end{tabular}
		\end{center}
	\end{table}

	\section{Conclusion} \label{sec_conclusion}
	In this paper, we studied and generalized the $\alsox$ algorithm for solving chance constrained programs (CCP). We showed that when uncertain constraints are convex, the $\alsox$ always outperforms $\CVaR$ approximation, the well-known best convex approximation in literature. We also showed several sufficient conditions under which $\alsox$ can return an optimal solution to CCP. We also provided an equivalent bilinear programming formulation of CCP, which allows us to enhance the $\alsox$ with a convergent alternating minimization scheme ($\alsox +$). We extended $\alsox$ to solve the distributionally robust chance constrained programs (DRCCPs) under {$\infty-$}Wasserstein ambiguity set. Our numerical study showed the effectiveness of the proposed algorithms.

{	\section*{Acknowledgment}
This research has been supported by the National Science Foundation grant 2046426. Valuable comments from the associate editor and two anonymous reviewers are gratefully acknowledged.}
	
	\bibliography{alsox}

	\newpage
	\titleformat{\section}{\large\bfseries}{\appendixname~\thesection .}{0.5em}{}
	\begin{appendices}
		
		\section{Proofs}\label{proofs}
		
		\vspace{10pt}
		
		\noindent\textbf{\large Proofs in Section~\ref{alsox_property}}
		
		\vspace{5pt}
		
		\subsection{Proof of \Cref{also_+_formulation_1}}\label{proof_also_+_formulation_1}\alsoxformulation*
		
	\begin{subequations}
		\begin{proof}
		We {prove} Formulation \eqref{eq_bilinear} and CCP \eqref{eq_bilinear_prop} are equivalent. Let $v_1, v_2$ be the optimal values of Formulation \eqref{eq_bilinear} and CCP \eqref{eq_bilinear_prop}, respectively. Then it remains to show that $v_1\leq v_2$ and $v_2\leq v_1$.
		\begin{enumerate}[]
			\item ($v_1\leq v_2$) Let $( \bm x^*, {z^*(\cdot)})$ be an optimal solution to CCP \eqref{eq_bilinear_prop}. Define $s^*(\rxi):=\max \{ g(\bm x^*,\rxi), 0\}$. {According to the properties of the measurable functions (see, e.g., section 3.1 in \citealt{royden1988real}), $s^*(\cdot)$ is measurable.} 
			{As in Formulation \eqref{eq_bilinear_prop}, the constraint $z^*(\trxi)\leq \I( g(\bm x^*,\trxi)\leq 0)$ holds a.s., we have 
				\begin{align*}
				0\leq 	\E[z^*(\trxi)s^*(\trxi)]  \leq 	\E\left[  \I( g(\bm x^*,\trxi)\leq 0) s^*(\trxi) \right],
				\end{align*} 
where the first inequality is due to the nonnegativity of $z^*(\rxi)s^*(\rxi)$ and the second one is because of monotonicity and nonnegativity of $s^*(\cdot)$.
				Since $s^*(\rxi):=\max \{ g(\bm x^*,\rxi), 0\}$, for any positive $t>0$,
we have
				\begin{align*}
					\Pr\left\{\trxi: \I( g(\bm x^*,\trxi)\leq 0)s^*(\trxi) \geq t  \right\}
=	\Pr\left\{\trxi:  s^*(\trxi) \geq t,  g(\bm x^*,\trxi)\leq 0\right\}=0,
				\end{align*}
which implies
				\begin{align*}
					\E\left[  \I( g(\bm x^*,\trxi)\leq 0) s^*(\trxi) \right]\leq 0.
				\end{align*}
		Thus, we must have  
		$\E[{z}^*(\trxi){s}^*(\trxi)]=0$.
			}
					Therefore, $(\bm x^*,{z^*(\cdot),s^*(\cdot)})$ is feasible to Formulation \eqref{eq_bilinear}, and hence $v_1\leq v_2$.
			\item ($v_2\leq v_1$) 	Let $(\bm x^*,{z^*(\cdot),s^*(\cdot)})$ be an optimal solution to Formulation \eqref{eq_bilinear} and suppose ${\hat{z}^*(\cdot)=\I\{ z^*(\cdot)>0\}}$.  {By the properties of the measurable functions (see, e.g., section 3.1 in \citealt{royden1988real}), the functional variable $\hat{z}^*(\cdot)$ is measurable.} 
			{	Thus, $\Pr\{\tilde{\bm\xi} \colon\hat z^*(\trxi)\in \{0,1\}\}=1$ and $\Pr\{\tilde{\bm\xi} \colon\hat z^*(\trxi)\geq z^*(\trxi)\}=1$. Together with the fact that $ \E[z^*(\trxi)]\geq 1-\varepsilon$, we have $ \E[\hat{z}^*(\trxi)]\geq 1-\varepsilon$ because of monotonicity}.  
			{According to constraints $ \E[z^*(\trxi)s^*(\trxi)]=0,{z}^*(\trxi)\in [0,1], s^*(\trxi)\geq 0$ almost surely, we have 
				\begin{align*}
					\Pr\left\{\tilde{\bm\xi} \colon z^*(\trxi)s^*(\trxi)=0, 0 \leq z^*(\trxi)\leq 1, s^*(\trxi)\geq 0\right\}=1.
				\end{align*}
			By the law of total probability (see, e.g., appendix A of \citealt{tijms2003first}), the above identity is equivalent to
				\begin{align}
					\Pr\left\{\tilde{\bm\xi} \colon 0 \leq z^*(\trxi)\leq 1, s^*(\trxi)=0\right\}+\Pr\left\{\tilde{\bm\xi} \colon  z^*(\trxi)=0, s^*(\trxi)> 0\right\}=1.\label{eq_two_terms}
				\end{align}
				Next, we bound two terms on the left-hand side separately. Since $g(\bm x^*, \trxi)\leq s^*(\trxi)$ holds almost surely, the conditional probability $\Pr\{\tilde{\bm\xi} \colon \I\{z^*(\trxi)> 0\} \leq \I \{ g(\bm x^*, \trxi)\leq 0\}\mid 0 \leq z^*(\trxi)\leq 1, s^*(\trxi)=0\}=1$. Hence, the first term on the left-hand side in \eqref{eq_two_terms} is equivalent to
				\begin{align*}
					& 			\Pr\left\{ \trxi \colon 0 \leq z^*(\trxi)\leq 1, s^*(\trxi)=0\right\}\\
					= &
					\Pr\left\{ \trxi \colon \I\left\{z^*(\trxi)> 0\right\} \leq \I \left\{ g(\bm x^*, \trxi)\leq 0 \right\}, 0 \leq z^*(\trxi)\leq 1, s^*(\trxi)=0\right\}.
				\end{align*}
	Since $\hat z^*(\trxi)=\I(z^*(\trxi)>0)$ and $0 \leq z^*(\trxi)\leq 1$ hold almost surely, we have
		\begin{align}
				\Pr\left\{ \tilde{\bm\xi} \colon 0 \leq z^*(\trxi)\leq 1, s^*(\trxi)=0\right\}  =\Pr\left\{ \trxi \colon \hat z^*(\trxi)\leq \I \left\{ g(\bm x^*, \trxi)\leq 0 \right\}, s^*(\trxi)=0\right\}.\label{eq_two_terms_left}
		\end{align}
		Similarly, the second term on the left-hand side in \eqref{eq_two_terms} can be written as
				\begin{align}
					\Pr\left\{ \trxi \colon z^*(\trxi)=0, s^*(\trxi)> 0\right\} = \Pr\left\{  \trxi \colon \hat z^*(\trxi) \leq \I \left\{ g(\bm x^*, \trxi)\leq 0\right\}, s^*(\trxi)> 0\right\}.\label{eq_two_terms_right}
				\end{align}
				Combining equalities \eqref{eq_two_terms}, \eqref{eq_two_terms_left} and \eqref{eq_two_terms_right} together, we have
				\begin{align*}
			\Pr\left\{ \trxi \colon \hat z^*(\trxi)\leq \I \left\{ g(\bm x^*, \trxi)\leq 0 \right\}, s^*(\trxi)=0\right\} + \Pr\left\{  \trxi \colon \hat z^*(\trxi) \leq \I \left\{ g(\bm x^*, \trxi)\leq 0\right\}, s^*(\trxi)> 0\right\}=1.
				\end{align*}
					By the law of total probability (see, e.g., appendix A of \citealt{tijms2003first}), the above equality can be simplified as
				\begin{align*}
					\Pr\left\{ \trxi \colon \hat z^*(\trxi) \leq \I \left\{ g(\bm x^*, \trxi)\leq 0\right\}\right\} = 1.
				\end{align*}
			}
			Thus, ($\bm x^*, {\hat{z}^*(\cdot)}$) satisfies the constraints in CCP \eqref{eq_bilinear_prop}. Therefore, $(\bm x^*,{\hat{z}^*(\cdot)})$ is feasible to CCP \eqref{eq_bilinear_prop}, and thus $v_2\leq v_1$. 	
		\end{enumerate}
		Therefore, Formulation \eqref{eq_bilinear} and CCP \eqref{eq_bilinear_prop} are equivalent. This concludes the proof. \QEDA
	\end{proof}
		\end{subequations} 	

		\subsection{Proof of \Cref{also_+_formulation_2}}\label{proof_also_+_formulation_2}\alsoxformulationbi*
		
			\begin{proof}
			Let $v_1, v_2$ be the optimal values of CCP \eqref{eq_bilinear} and Formulation \eqref{alsox_bilnear}, respectively. Then it remains to show that $v_1\leq v_2$ and $v_2\leq v_1$.
			\par
			\noindent ($v_1\leq v_2$) Let $(\bm x^*,{{z}^*(\cdot),{s}^*(\cdot)},t^*)$ be an optimal solution of Formulation \eqref{alsox_bilnear}. According to \eqref{alsox_bilnearc}, we have {$\Pr\{\trxi:g(\bm x^*,\tilde{\bm\xi})\leq 0\} \geq 1-\varepsilon$}, which implies $ \E_\Pr[\I(g(\bm x^*,\tilde{\bm\xi})\leq 0)] \geq 1-\varepsilon$. We can construct an optimal solution of Formulation \eqref{alsox_bilnear} as follows. { Let $\hat{s}({\rxi}) = \max\{ g(\bm x^*,{\rxi}),0\}$ and $\hat{z}({\rxi})=\I\{g(\bm x^*,{\rxi})\leq 0\}$. According to the properties of the measurable functions (see, e.g., section 3.1 in \citealt{royden1988real}), $\hat{s}({\trxi})$ and  $\hat{z}({\trxi})$ are measurable.}
			In this way, $(\bm x^*,\hat{z}({\rxi}),\hat{s}({\rxi}))$ satisfies the constraints in \eqref{alsox_bilnearb}. {Since $\hat{s}({\trxi})\geq 0$, $\hat{z}({\trxi})\geq 0$ hold almost surely,  $\E[\hat{z}({\trxi}) \hat{s}({\trxi})]$ is well defined. }
			 {From the proof in \Cref{also_+_formulation_1}, we have $\E[\hat{z}(\tilde{\rxi})\hat{s}(\tilde{\rxi})]=0$, indicating $(\bm x^*,\hat{z}(\cdot),\hat{s}(\cdot))$ solves the lower-level problem \eqref{alsox_bilnearb}}. Thus, $(\bm x^*,{\hat{z}(\cdot),\hat{s}(\cdot)},t^*)$ is another optimal solution of Formulation \eqref{alsox_bilnear}. 
			The fact that $(\bm x^*,{\hat{z}(\cdot),\hat{s}(\cdot)})$ is feasible to CCP \eqref{eq_bilinear} implies $v_1\leq v_2$.		
			\par
			\noindent ($v_2\leq v_1$) Let $(\bm x^*,{{z}^*(\cdot),{s}^*(\cdot)})$ be an optimal solution to CCP \eqref{eq_bilinear} and $t^*=\bm{c}^\top \bm{x}^*$. We have
			$\E[{z}^*(\tilde{\rxi}){s}^*(\tilde{\rxi})]=0$, which solves the lower-level problem \eqref{alsox_bilnearb}.
			Therefore, $(\bm x^*,{{z}^*(\cdot),{s}^*(\cdot)}, t^*)$ satisfies the constraints in Formulation \eqref{alsox_bilnear}. Thus, $v_2\leq v_1$.  \QEDA
		\end{proof}

		\subsection{Proof of \Cref{hinge_loss_truncated_el}}\label{proof_hinge_loss_truncated_el}\hingelosstruncatedel*
		
		\begin{proof}
			It is sufficient to prove that the hinge-loss approximation \eqref{eq_project_out} is equivalent to \eqref{eq_alsox_elb}. In fact, the objective function in the hinge-loss approximation \eqref{eq_project_out} can be calculated based on the definition of conditional expectation
			\begin{equation*}
			\E\left[[\tilde{\bm {\xi}}^\top\bm{a}_1(\bm{x})- b_1(\bm x)]_+\right]= \E\left[\tilde{\bm {\xi}}^\top \bm{a}_1(\bm{x}) - b_1(\bm x)\mid\tilde{\bm {\xi}}^\top \bm{a}_1(\bm{x})- b_1(\bm x)\geq 0\right]\Pr\left\{ \tilde{\bm {\xi}}^\top \bm{a}_1(\bm{x}) - b_1(\bm x)\geq 0\right\}.
			\end{equation*}
			Using the cumulative distribution function formula of an elliptical distribution, the objective function can be further simplified as
			\begin{equation*}
			\left(1-\mathrm{\Phi}(\frac{b_1(\bm x)-\bm\mu^\top \bm{a}_1(\bm{x})}{\sqrt{\bm{a}_1(\bm{x})^\top\bm{\mathrm{{\Sigma}}}\bm{a}_1(\bm{x})}})\right) \E\left[\tilde{\bm {\xi}}^\top \bm{a}_1(\bm{x}) - b_1(\bm x)\mid\tilde{\bm {\xi}}^\top \bm{a}_1(\bm{x})- b_1(\bm x)\geq 0\right].
			\end{equation*}
			According to the closed-form expression of the expectation of the truncated elliptical distribution (see, e.g., theorem 1 in \citealt{landsman2003tail}), the objective function is equivalent to
			\begin{equation*}
			\left(1-\mathrm{\Phi}(\frac{b_1(\bm x)-\bm\mu^\top \bm {x}}{\sqrt{\bm{a}_1(\bm{x})^\top\bm{\mathrm{{\Sigma}}}\bm{a}_1(\bm{x})}})\right)\left( \bm\mu^\top \bm{a}_1(\bm{x})-b_1(\bm x)+\sqrt{\bm{a}_1(\bm{x})^\top\bm{\mathrm{{\Sigma}}}\bm{a}_1(\bm{x})}\frac{\overline{G}(\frac{1}{2}(\frac{b_1(\bm x)-\bm\mu^\top \bm{a}_1(\bm{x}) }{\sqrt{\bm{a}_1(\bm{x})^\top\bm{\mathrm{{\Sigma}}}\bm{a}_1(\bm{x})}})^2)}{\left(1-\mathrm{\Phi}(\frac{b_1(\bm x)-\bm\mu^\top \bm{a}_1(\bm{x}) }{\sqrt{\bm{a}_1(\bm{x})^\top\bm{\mathrm{{\Sigma}}}\bm{a}_1(\bm{x})}})\right)} \right).
			\end{equation*}
			Let $\alpha = (b_1(\bm x)-\bm\mu^\top \bm{a}_1(\bm{x}))/\sqrt{\bm{a}_1(\bm{x})^\top\bm{\mathrm{{\Sigma}}}\bm{a}_1(\bm{x})} $. Then the objective can be simplified as
			\begin{equation*}
			\left(1-\mathrm{\Phi}(\alpha)\right)(\bm\mu^\top \bm{a}_1(\bm{x})-b_1(\bm x))+\sqrt{\bm{a}_1(\bm{x})^\top\bm{\mathrm{{\Sigma}}}\bm{a}_1(\bm{x})} \overline{G}(\alpha^2/2) =\sqrt{\bm{a}_1(\bm{x})^\top\bm{\mathrm{{\Sigma}}}\bm{a}_1(\bm{x})} \left(\overline{G}(\alpha^2/2)-\alpha+\alpha\mathrm{\Phi}(\alpha) \right).
			\end{equation*}
			%
			This concludes the proof.
			\QEDA
		\end{proof}

		\vspace{15pt}
		
		\noindent\textbf{\large Proofs in Section~\ref{alsox_strengths}}
		
		\vspace{5pt}

		\subsection{Proof of \Cref{the1}}\label{proof_the1}\alsoxbettercvarccp*
		\begin{proof}
			It is sufficient to show that for any given $t$, if an optimal solution of the lower-level problem \eqref{eq_also_x_2b} violates the chance constraint \eqref{eq_also_x_2c}, so does $\CVaR$ approximation. Next, we split the proof into two steps.\\
			\noindent\textbf{Step 1.} Recall that for any given $t$, the hinge-loss approximation \eqref{eq_also_x_2b} is
			\begin{subequations}
				\label{eq_also_x_2'}
				\begin{align}
				v^A(t)=\min_{\bm{x}\in \mathcal{X},{s(\cdot)}}\quad& \E[{s}(\tilde{\rxi})],\label{eq_also_x_2_a}\\
				\textup{s.t.}\quad & g(\bm x,\tilde{\bm {\xi}})\leq s(\tilde{\rxi}), \label{eq_also_x_2_b}\\
				& \bm{c}^\top \bm{x} \leq t,\label{eq_also_x_2_c}\\
				&s(\tilde{\rxi})\geq 0.\label{eq_also_x_2_d}
				\end{align}
			\end{subequations}
			And the lower-level problem \eqref{eq_cvar_t} is
			\begin{subequations}
				\label{eq_cvar_t'}
				\begin{align}
				v^{\CVaR}(t)=\min_{{\bm x\in \mathcal{X}},{s(\cdot)},\beta\leq 0}\ & \E[{s}(\tilde{\rxi})] - (1-\varepsilon)\beta,
				\label{eq_cvar_t_a}\\
				\textup{s.t.}\quad & g(\bm x,\tilde{\bm {\xi}})\leq s(\tilde{\rxi}), \label{eq_cvar_t_b}\\
				& \bm{c}^\top \bm{x} \leq t,\label{eq_cvar_t_c}\\
				& s(\tilde{\rxi})\geq \beta. \label{eq_cvar_t_d}
				\end{align}
			\end{subequations}
			Suppose that $(\bm x^*,{s^*(\cdot)})$ is an optimal solution to the hinge-loss approximation problem \eqref{eq_also_x_2'}. We would like to prove that if $\bm x^*$ is infeasible to CCP \eqref{eq_ccp}, i.e., $\Pr\{\tilde{\bm\xi}\colon s^*(\tilde{\bm\xi})>0\}>\varepsilon$, then we must have 
			$v^A(t)=v^{\CVaR}(t)>0$. Note that we already have $v^A(t)>0$ due the presumption that $\bm x^*$ is infeasible to CCP \eqref{eq_ccp}.
			\begin{subequations}
				To prove $v^A(t)=v^{\CVaR}(t)$, we let ${\alpha(\cdot),\pi, \mu(\cdot)}$ be the dual variables of constraints \eqref{eq_cvar_t_b}, \eqref{eq_cvar_t_c}, and \eqref{eq_cvar_t_d}, respectively. 
				The Lagrangian of the lower-level problem \eqref{eq_cvar_t'} is
				\begin{align*}
				\mathcal{L}\left(\bm{x},\beta,{s(\cdot),\mu(\cdot),\alpha(\cdot)},\pi\right):=&\E[{s}(\tilde{\rxi})]-(1-\varepsilon)\beta+\E\left[{\mu}(\tilde{\rxi})^\top[\beta-{s}(\tilde{\rxi})]\right]\\
				&+\pi(\bm{c}^\top \bm{x}-t)+\E\left[{\alpha}(\tilde{\rxi})^\top[g(\bm x,\tilde{\bm {\xi}})-{s}(\tilde{\rxi})]\right],
				\end{align*} 
				and its dual problem can be written as
				\begin{equation}
				v^{\CVaR}_D(t)=\max_{{\mu({\cdot}),\alpha({\cdot})},\pi}\, \min_{\bm{x},\beta,{s({\cdot})}} \mathcal{L}\left(\bm{x},\beta,{s(\cdot),\mu(\cdot),\alpha(\cdot)},\pi\right). \label{eq_cvar_t_dual}
				\end{equation}
				According to Assumptions \ref{A_1}-\ref{A_2}, the relaxed Slater condition holds, and thus theorem 1 in \cite{slater2014lagrange} implies that there is no duality gap between the lower-level problem \eqref{eq_cvar_t'} and its dual. Let $(\hat{\bm{x}},{\hat{s}(\cdot)},\hat{\beta})$ be an optimal solution of \eqref{eq_cvar_t'} and $({\hat{\alpha}(\cdot),\hat{\mu}(\cdot)},\hat{\pi})$ be an optimal solution of \eqref{eq_cvar_t_dual}. Then we have the following necessary and sufficient KKT conditions:
				\begin{align}
				&-\hat{\pi}\bm{c}^\top\in \partial_x\E\left[\hat{{\alpha}}(\tilde{\rxi})^\top[g(\hat{\bm {x}},\tilde{\bm {\xi}})]\right]+\N_{X}(\hat{\x}), \E[\hat{{\mu}}(\tilde{\rxi})]\leq 1-\varepsilon, \hat{{\mu}}(\tilde{\rxi})+\hat{{\alpha}}(\tilde{\rxi})=1, \nonumber\\
				& {0}\leq \hat{{\alpha}}(\tilde{\rxi}) \perp \left(\hat{{s}}(\tilde{\rxi})- g(\hat{\bm x},\tilde{\bm {\xi}})\right) \geq {0}, 0\leq\hat{\pi} \perp (t-\bm{c}^\top \hat{\bm{x}})\geq 0, {0}\leq \hat{{\mu}}(\tilde{\rxi}) \perp {\hat{{s}}(\tilde{\rxi}) -\hat\beta} \geq{0},\nonumber\\
				&\hat{\beta}\leq 0,\hat{\bm{x}}\in \mathcal{X}.\label{KKT1}\tag{KKT1}
				\end{align}
				Similarly, let ${\alpha(\cdot),\pi, \mu(\cdot)}$ be the dual variables of constraints \eqref{eq_also_x_2_b}, \eqref{eq_also_x_2_c}, and \eqref{eq_also_x_2_d}, respectively. The Lagrangian of the hinge-loss approximation \eqref{eq_also_x_2'} is 
				\begin{align*}
				\mathcal{L}\left(\bm{x},{s(\cdot),\mu(\cdot),\alpha(\cdot)},\pi\right):=\E[{s}(\trxi)]+\E\left[{\mu}(\trxi)^\top[-{s}(\trxi)]\right]
				+\pi(\bm{c}^\top \bm{x}-t)+\E\left[{\alpha}(\trxi)^\top[g(\bm x,\trxi)-{s}(\trxi)]\right],
				\end{align*} 
				and its dual program is
				\begin{equation} v^A_D(t)=\max_{{\mu({\cdot}),\alpha({\cdot})},\pi}\, \min_{\bm{x},{s({\cdot})}} \mathcal{L}\left(\bm{x},s({\rxi}),\mu({\rxi}),\alpha({\rxi}),\pi\right). \label{eq_alsox_dual}
				\end{equation}
				From the similar argument, the strong duality also holds, i.e., we must have $v^A_D(t)=v^A(t)$. 	Let $(\Bar{\bm{x}},{\Bar{s}(\cdot)})$ be an optimal solution of \eqref{eq_also_x_2'}, and $({\Bar{\alpha}({\cdot}),\Bar{\mu}({\cdot})},\Bar{\pi})$ be an optimal dual solution of \eqref{eq_alsox_dual}. Then we have the following necessary and sufficient KKT conditions:
				\begin{align}
				&-\Bar{\pi}\bm{c}^\top\in \partial_x\E\left[\Bar{{\alpha}}(\tilde{\rxi})^\top[g(\bm \Bar{x},\tilde{\bm {\xi}})]\right]+\N_{X}(\Bar{\x}), \Bar{{\mu}}(\tilde{\rxi})+\Bar{{\alpha}}(\tilde{\rxi})=1,\Bar{\bm{x}}\in \mathcal{X},\nonumber\\
				& 0\leq \Bar{{\alpha}}(\tilde{\rxi}) \perp \left(\Bar{{s}}(\tilde{\rxi})- g(\Bar{\bm x},\tilde{\bm {\xi}})\right) \geq 0, 0\leq\Bar{\pi} \perp (t-\bm{c}^\top \Bar{\bm{x}})\geq 0, 0\leq \Bar{{\mu}}(\tilde{\rxi}) \perp \Bar{{s}}(\tilde{\rxi}) \geq{0}.\label{KKT2}\tag{KKT2}
				\end{align}
				\noindent\textbf{Step 2.} To prove $v^{\CVaR}(t)= v^A(t)$ is equivalent to show that $v^{\CVaR}_D(t)=v^A_D(t)$. According to our presumption that $(\Bar{\bm x},{\Bar{s}({\cdot})})$ violates the chance constraint \eqref{eq_also_x_2c}, i.e., 
				$\Pr\{ \tilde{\rxi}\colon\Bar{{s}}(\tilde{\rxi})>0\} >\varepsilon$, which implies that $\Pr\{ \tilde{\rxi}:\Bar{{\mu}}(\tilde{\rxi})=0\} >\varepsilon$ from conditions \eqref{KKT2}.
				We also have
				\begin{equation*}
				\E\left[\Bar{\mu}(\tilde{\rxi})+\Bar{\alpha}(\tilde{\rxi})\right]=1,
				\end{equation*}
				and
				\begin{equation*}
				\E\left[\Bar{{\alpha}}(\tilde{\rxi})\right]\geq \E\left[\Bar{{\alpha}}(\tilde{\rxi})\I(\Bar{{\mu}}(\tilde{\rxi})={0})\right]=\Pr\left\{ \tilde{\rxi}\colon\Bar{{\mu}}(\tilde{\rxi})=0\right\} >\varepsilon.
				\end{equation*} 
				Since ${\Bar{{\alpha}}({\cdot})} \geq 0$, thus
				$\E[\Bar{{\mu}}(\tilde{\rxi})]< 1-\varepsilon$ must hold. This implies that the primal and dual pair, $(\Bar{\bm x},{\Bar{s}(\cdot)},\Bar{\beta}=0)$ and $({\Bar {\alpha}(\cdot)},\Bar \pi,{\Bar \mu(\cdot)})$, satisfies conditions \eqref{KKT1}. That is, $(\Bar{\bm x},{\Bar{s}(\cdot)},\Bar{\beta}=0)$ is optimal to the lower-level problem \eqref{eq_cvar_t_dual}. Hence, we have $v^A(t)= v^{\CVaR}(t)>0$.
				\QEDA
			\end{subequations}
		\end{proof}
		
		\subsection{Proof of \Cref{the4}}\label{proof_the4}\exactnessec*
		\begin{proof} 
			\begin{enumerate}[label=(\roman*)]
				\item We {prove} the ``only if" direction by contradiction.
				Suppose that the generalized nullspace property does not hold. Then there exists a solution $(\hat{\bm x},{\hat{s}(\cdot)})$ with ${{\trxi}}^\top(\bm A\hat{\bm x})-\bm B\hat{\bm x}-{\hat{s}}({{\trxi}})=0$  and $\hat{s}({{\trxi}})\not=0$ {a.s.} and a $\Pr-$measurable set $\mathcal{S} \subseteq\Xi$ such that $\Pr\{\tilde{\rxi}\colon \tilde{\rxi} \in \mathcal{S}\} \leq \varepsilon$ and $\E[|\hat{s}(\tilde{\rxi})|\I(\tilde{\rxi}\in \mathcal{S})]\geq \E[|\hat{s}(\tilde{\rxi})|\I(\tilde{\rxi}\notin \mathcal{S})]$. 
				Since
			$	{{\trxi}}^\top(\bm A\hat{\bm x})-\bm B\hat{\bm x}-	{\hat{s}}({{\trxi}})=0$ {holds a.s.}, we have $ 	{{\trxi}}^\top(\bm A\hat{\bm x})\I(	{{\trxi}}\in \mathcal{S})-\bm B\hat{\bm x}-{\hat{s}}(	{{\trxi})}\I({\trxi}\in \mathcal{S})=0$ and $	{{\trxi}}^\top (\bm A\hat{\bm x})\I(	{{\trxi}}\notin \mathcal{S})-\bm B\hat{\bm x}-	{\hat{s}}({{\trxi})}\I(	{{\trxi}}\notin \mathcal{S})=0$ 	{a.s.}.
				Then, $(\bm{\hat{x}}, \hat{ s}({\rxi})\I({\rxi}\in \mathcal{S}) )$ is not the unique optimal solution to the problem 
				\begin{align*}
				\min_{\bm{x},{s(\cdot)}}\left\{\E\left[|{s}(\tilde{\rxi})|\right]\colon
				\begin{array}{l}
				\displaystyle \tilde{\bm {\xi}}^\top(\bm A\bm x)-\bm b^\top\bm x-{s}(\tilde{\rxi})
				=\tilde{\bm {\xi}}^\top (\bm A\hat{\bm x})-\bm B\hat{\bm x}-{\hat{s}}(\tilde{\rxi})\I(\tilde{\rxi}\in \mathcal{S}),\\
				\displaystyle	 \bm{c}^\top \bm{x} = \bm{c}^\top \hat{\bm{x}}, {\bm U^\top \bm x=\bm U^\top \hat{\bm{x}}}
				\end{array}
				\right\}.
				\end{align*}
				Since $(\bm0,-\hat{s}({\rxi})\I({\rxi}\notin \mathcal{S})) $ is a different feasible solution to the problem, whose objective value is at least as good as $(\bm{\hat{x}}, \hat{s}({\rxi})\I({\rxi}\in \mathcal{S}) )$. This violates the presumption that all the optimal solution $(\bm{x}^*,{s^*(\cdot)})$ to the hinge-loss approximation \eqref{single_ccp_equality_1} has the same ${s^*(\cdot)}$ and satisfies $\Pr\{\tilde{\bm\xi} : s^*(\tilde{\rxi})=0\} \geq 1-\varepsilon$.
				
				To prove the ``if" direction: we let $(\bm{x},{s(\cdot)})$ be a feasible solution of the hinge-loss approximation \eqref{single_ccp_equality_1} that satisfies $\Pr\{\tilde{\rxi}: s(\tilde{\rxi})\neq 0\} \leq \varepsilon$ and $(\bm{\hat{x}},{\hat{s}(\cdot)})$ be an optimal solution of \eqref{single_ccp_equality_1} such that ${s(\cdot)\neq\hat{s}(\cdot)}$. Let us denote $\mathcal{S}=\{{\rxi}: s({\rxi})\neq 0\}$, $Z(\tilde{\rxi})=s({\rxi})-\hat{s}({\rxi})\neq 0$, and $ \bar{\bm{x}} = \bm{x}-\bm{\hat{x}} $. Then we have ${\trxi}^\top(\bm A\bar{\bm{x}} )-\bm B\bar{\bm{x}}-Z({\rxi})=0$ {a.s.}, $\bm{c}^\top\bar{\bm{x}}=\bm{c}^\top\bm{x}-\bm{c}^\top\hat{\bm{x}}=0$, and {$\bm{U}^\top\bar{\bm{x}}=\bm{U}^\top\bm{x}-\bm{U}^\top\hat{\bm{x}}=\bm 0$.}
				Thus,
				\begin{align*}
				&\E\left[|{s}(\tilde{\rxi})|\I(\tilde{\rxi}\in \mathcal{S})\right] - \E\left[|\hat{{s}}(\tilde{\rxi})|\I(\tilde{\rxi}\in \mathcal{S})\right]\leq \E\left[|{s}(\tilde{\rxi})-\hat{{s}}(\tilde{\rxi})|\I(\tilde{\rxi}\in \mathcal{S})\right] \\ =& \E\left[|Z(\tilde{\rxi})|\I(\tilde{\rxi}\in \mathcal{S})\right]< \E\left[|Z(\tilde{\rxi})|\I(\tilde{\rxi}\notin \mathcal{S})\right] = \E\left[|\hat{{s}}(\tilde{\rxi})|\I(\tilde{\rxi}\notin \mathcal{S})\right],
				\end{align*}
				the first inequality is due to the triangular inequality and the second inequality is based on the generalized nullspace property.
				
				Therefore, we get 
				\begin{equation*}
				\E\left[|{s}(\tilde{\rxi})|\right] = \E\left[|{s}(\tilde{\rxi})|\I(\tilde{\rxi}\in \mathcal{S})\right] < \E\left[|\hat{{s}}(\tilde{\rxi})|\right],
				\end{equation*}
				which is a contradiction to the optimality of ${\hat{s}(\cdot)}$.
				\item This follows directly from Part (i) by letting $t=v^*$. 	\QEDA
			\end{enumerate}
			
		\end{proof}
		
		\subsection{Proof of \Cref{exactness_set_covering}}\label{proof_exactness_set_covering}\exactnesssetcovering*
		
		\begin{proof}
			Since $v^A\geq v^*$, thus it suffices to show that $v^A\leq v^*$.
			In fact, we claim that for any $t\geq v^*$, there exists an optimal solution of the hinge-loss approximation \eqref{eq_also_x_2b} which satisfies the chance constraint \eqref{eq_also_x_2c}. We {prove} it by contradiction. Suppose the statement is not true. That is, there exists a $t\geq v^*$ such that an optimal solution $(\Bar{\bm x}, {\Bar{s}(\cdot)})$ of the hinge-loss approximation \eqref{eq_also_x_2b} violates the chance constraint (i.e., 
			$\Pr\{\tilde{\rxi}\colon \Bar{s}(\trxi)>0\}> \varepsilon$). According to the definition of random function $g(\cdot,\cdot)$, we know that {$\Pr\{\trxi:{\Bar{s}(\trxi)}\in \{0,M\}\}=1$}.
			{Let $\Bar{s}({\rxi})=M\I\{g(\Bar{\bm x},{\rxi})> 0\}$. According to the properties of the measurable functions (see, e.g., section 3.1 in \citealt{royden1988real}), $\Bar{s}({\trxi})$ is measurable. }
			 {Thus, we have}
			  $\E[\Bar{s}(\tilde{\rxi})]\geq M\Pr\{\Bar{s}(\tilde{\rxi})>0\}>M\varepsilon$. 
			
			On the other hand, let $(\bm x^*,{{s}^*(\cdot),{z}^*(\cdot)})$ be an optimal solution of CCP	\eqref{eq_bilinear} with optimal value $v^*$. Define $\hat{s}({\rxi})=M-M{z}^*({\rxi})$, and we have $$\E[\hat{s}(\tilde{\rxi})]=M-M\E[{z}^*(\tilde{\rxi})]\leq M-M(1-\varepsilon).$$ 
			Clearly, $(\bm x^*,{\hat{z}(\cdot)})$ is feasible to the hinge-loss approximation \eqref{eq_also_x_2b} with an objective value $\E[\hat{s}(\tilde{\rxi})]\leq M\varepsilon$, which contradicts the optimality of $(\Bar{\bm x}, {\Bar{s}(\cdot)})$.
			\QEDA
		\end{proof}

		\subsection{Proof of \Cref{exact_el}}\label{proof_exact_el}\exactel*
		
		\begin{proof}
			We split the proof into two parts by checking two conditions separately.
			\begin{enumerate} [label=(\roman*)]
				\item Suppose that $\mathcal{X}\subseteq\{\bm x:\sqrt{\bm{a}_1(\bm{x})^\top\bm{\mathrm{{\Sigma}}}\bm{a}_1(\bm{x})}=C\} $, where $C$ is a positive constant. Then for a given $t$, the hinge-loss approximation \eqref{eq_alsox_elb} becomes
				\begin{equation}
				\begin{aligned}
				(\bm{x}^*,\alpha^*)\in \argmin_{\bm{x}\in \mathcal{X},\alpha } \biggl\{ & C \left(\overline{G}(\alpha^2/2)-\alpha+\alpha\mathrm{\Phi}(\alpha) \right)\colon \bm{c}^\top \bm{x} \leq t,\\
				&(b_1(\bm x)-\bm\mu^\top \bm{a}_1(\bm {x}))/C =\alpha \biggr\}.\label{eq_alsox_elb1}
				\end{aligned}
				\end{equation}

				Let $f(\alpha)=\overline{G}(\alpha^2/2)-\alpha+\alpha\mathrm{\Phi}(\alpha)$. 
				Since its derivative is
				$$\partial f(\alpha)/\partial \alpha=-k\alpha \hat g(\alpha^2/2)-1+\mathrm{\Phi}(\alpha)+k\alpha \hat g(\alpha^2/2)=\mathrm{\Phi}(\alpha)-1<0,$$
				function $f(\alpha)$ is monotone decreasing over $\alpha \in \Re$. Thus, for any $t\geq v^*$, i.e., there exists an $\hat{\bm{x}}$ which is feasible to CCP \eqref{ccp_el} such that
				$$(b_1(\hat{\bm{x}})-\bm\mu^\top \bm{a}_1(\hat{\bm{x}}))/C\geq \mathrm{\Phi}^{-1}(1-\varepsilon), \bm{c}^\top \hat{\bm{x}}\leq t.$$
				
				Clearly, $(\hat{\bm{x}},\hat{\alpha}=\mathrm{\Phi}^{-1}(1-\varepsilon))$ is feasible to the hinge-loss approximation \eqref{eq_alsox_elb1}. Due to the monotonicity of the objective function, we must have $\alpha^*\geq \hat{\alpha}=\mathrm{\Phi}^{-1}(1-\varepsilon)$, i.e.,
				\[b_1(\bm {x}^*)-\bm{\mu}^\top\bm{a}_1(\bm {x}^*)\geq \mathrm{\Phi}^{-1}(1-\varepsilon) C.\]
				Hence, we must have $\bm {x}^*$ is also feasible to CCP \eqref{ccp_el}, i.e., $v^A\leq v^*$. On the other hand, we always have $v^A\geq v^*$. Thus, $v^A=v^*$.
				
				\item Suppose that $\mathcal{X}\subseteq\{\bm x:b_1(\bm x)-\bm\mu^\top \bm{a}_1(\bm{x})=C\} $, where $C$ is an arbitrary constant. Let us denote $\sigma =\sqrt{\bm{a}_1^\top(\bm {x})\bm{\mathrm{{\Sigma}}}\bm{a}_1(\bm {x})} $, and we have $\alpha=C/\sigma$. Then for a given $t$, the hinge-loss approximation \eqref{eq_alsox_elb} becomes
				\begin{equation}
				(\bm{x}^*,\sigma^*)\in \argmin_{\bm{x}\in \mathcal{X},\sigma} \left\{ \sigma f(C/\sigma)\colon \sqrt{\bm{a}_1(\bm{x})^\top\bm{\mathrm{{\Sigma}}}\bm{a}_1(\bm{x})}=\sigma, \bm{c}^\top \bm{x} \leq t\right\},\label{eq_alsox_elb2}
				\end{equation}
				where function $f(\cdot)$ is defined in the proof of Part (i). 
				Taking the derivative of the objective function with respect to $\sigma$, we have 
				\begin{equation*}
				\frac{\partial (\sigma f(C/\sigma))}{\partial \sigma} = f(C/\sigma)-C/\sigma \frac{\partial f(\alpha)}{\partial\alpha} \big|_{\alpha=C/\sigma}= \overline{G}(C^2/(2\sigma^2))> 0.
				\end{equation*}
				Thus, the objective function in the hinge-loss approximation is always monotone increasing with respect to $\sigma$. Thus, for any $t\geq v^*$, i.e., there exists a $\hat{\bm{x}}$ which is feasible to CCP \eqref{ccp_el} such that
				$$C\geq \mathrm{\Phi}^{-1}(1-\varepsilon) \sqrt{\bm{a}_1(\hat{\bm{x}})^\top\bm{\mathrm{{\Sigma}}}\bm{a}_1(\hat{\bm{x}})}, \bm{c}^\top \hat{\bm{x}}\leq t.$$
				
				Clearly, $(\hat{\bm{x}},\hat{\sigma}=\sqrt{\bm{a}_1(\hat{\bm{x}})^\top\bm{\mathrm{{\Sigma}}}\bm{a}_1(\hat{\bm{x}})})$ is feasible to the hinge-loss approximation \eqref{eq_alsox_elb1}. Due to the monotonicity of the objective function, we must have $\sigma^*\leq \hat{\sigma}\leq C/\mathrm{\Phi}^{-1}(1-\varepsilon)$, i.e.,
				\[b_1(\bm {x}^*)-\bm{\mu}^\top\bm{a}_1(\bm {x}^*):=C\geq \mathrm{\Phi}^{-1}(1-\varepsilon) \sqrt{\bm{a}_1(\bm{x}^*)^\top\bm{\mathrm{{\Sigma}}}\bm{a}_1(\bm{x}^*)}.\]
				Hence, we must have $\bm {x}^*$ is also feasible to CCP \eqref{ccp_el}, i.e., $v^A\leq v^*$. On the other hand, we always have $v^A\geq v^*$. Thus, $v^A=v^*$. 	\QEDA
				
			\end{enumerate}
			
		\end{proof}

%
		
		\subsection{Proof of \Cref{thm_ccp}}\label{proof_thm_ccp}\thmccp*
		
		\begin{proof}
			First, without loss of generality, we assume $v^{rel}>0$; otherwise, we must have $v^{rel}=v^{A}=v^*=0$ due to the covering structure. Next, we split the proof into three steps.
			
			\noindent{\textbf{Step 1.}} Let $\alpha,\{ \bm\beta_i\}_{i\in [N]}$ be the dual variables of the constraints of the continuous relaxation \eqref{covering_relax}, respectively. Then the dual of the continuous relaxation \eqref{covering_relax} is 
			\begin{subequations}
				\label{covering_relax_1_D}
				\begin{align}
				v^{rel}=\max_{\alpha\geq 0, \bm \beta_i\geq \bm 0,\forall i\in [N]}\quad& -\floor{N\varepsilon}\alpha +\sum_{i\in[N]}\bm\beta_i^\top \bm e,\label{covering_relax_1_Da}\\
				\text{s.t.}\ & \sum_{i\in[N]} {\bm{A}^i}^\top\bm\beta_i\leq \bm c, \label{covering_relax_1_Db}\\
				& \bm\beta_i^\top\bm e\leq\alpha,\label{covering_relax_1_Dc}
				\end{align}
			\end{subequations}
			where the strong duality holds since the continuous relaxation \eqref{covering_relax} is always feasible.
			
			For any $t\geq 0$, let $\gamma, \{\bm\omega_i\}_{i\in[N]}$ be the dual variables of the constraints of the hinge-loss approximation \eqref{covering_hinge}, respectively. The dual of the hinge-loss approximation \eqref{covering_hinge} is
			\begin{subequations}
				\label{covering_alsox_1_D}
				\begin{align}
				v^{A}(t)=\max_{\gamma\geq 0, \bm \omega_i\geq \bm 0,\forall i\in [N]}\quad& -t\gamma +\sum_{i\in[N]}\bm\omega_i^\top \bm e,\label{covering_alsox_1_Da}\\
				\text{s.t.}\ & \sum_{i\in[N]} {\bm{A}^i}^\top\bm\omega_i\leq \gamma \bm c, \label{covering_alsox_1_Db}\\
				& \bm\omega_i^\top\bm e\leq \frac{1}{N}, \label{covering_alsox_1_Dc}
				\end{align}
			\end{subequations}
			where the strong duality also holds since the hinge-loss approximation \eqref{covering_hinge} is always feasible.
			
			\noindent{\textbf{Step 2.}} Next, we prove the result by contradiction. Suppose that in the hinge-loss approximation \eqref{covering_hinge}, for a given $t\geq	v^{rel}(\floor{N\varepsilon}+1)>0$, there exists an optimal solution $(\bm x^*,\bm s^*)$, which is infeasible to the upper-level problem of $\alsox$ \eqref{covering}, i.e., $|\mathrm{supp}(\bm s^*)|\geq \floor{N\varepsilon}+1$. 
			Let $(\gamma^*,\{\bm\omega_i^*\}_{i\in[N]})$ denote an optimal dual solution of \eqref{covering_alsox_1_D}. Due to the complementary slackness, we must have
			\begin{equation}
			{\bm\omega_i^*}^\top\bm e=\frac{1}{N},\forall i\in [N]: s_i^*>0.\label{eq_omega_i_supp}
			\end{equation} 
			
			Let $(\hat{\bm x},\hat{\bm s})$ be an optimal solution to the continuous relaxation \eqref{covering_relax}. According to \cite{ahmed2018relaxations}, the scaled solution $((\floor{N\varepsilon}+1)\hat{\bm x},\min\{\ceil{(\floor{N\varepsilon}+1)\hat{\bm s}},\e\})$ is feasible to the covering CCP \eqref{covering_ccp} with objective value at most $(\floor{N\varepsilon}+1)v^{rel}$. Thus, the scaled solution $((\floor{N\varepsilon}+1)\hat{\bm x},\min\{\ceil{(\floor{N\varepsilon}+1)\hat{\bm s}},\e\})$ is also feasible to the hinge-loss approximation \eqref{covering_hinge} since $t\geq	v^{rel}(\floor{N\varepsilon}+1) >0$. Hence, we must have 
			\begin{equation}
			0<v^{A}(t)=\frac{1}{N}\sum_{i\in[N]}s^*_i=-t\gamma^* +\sum_{i\in[N]}{\bm\omega_i^*}^\top \bm e\leq \frac{\floor{N\varepsilon}}{N}.\label{eq_omega_i_opt}
			\end{equation}
			According to \eqref{eq_omega_i_supp}, we have 
			$$\sum_{i\in[N]}{\bm\omega_i^*}^\top \bm e\geq \frac{1}{N}|\mathrm{supp}(\bm s^*)|\geq \frac{\floor{N\varepsilon}+1}{N}.$$
			Together with the second inequality in \eqref{eq_omega_i_opt}, we must have $\gamma^*>0$.
			
			Also, the first inequality in \eqref{eq_omega_i_opt} implies that
			\[\gamma^*< \frac{\sum_{i\in[N]}{\bm\omega_i^*}^\top \bm e}{t}.\]
			
			\noindent{\textbf{Step 3.}} 
			Now, let us define $ \hat{\bm \beta}_i=\bm\omega_i^*/\gamma^*$ and $\hat{\alpha}=1/(N\gamma^*)$. Clearly, $(\hat{\alpha},\{\hat{\bm \beta}_i\}_{i\in [N]})$ is feasible to the dual \eqref{covering_relax_1_D} of the continuous relaxation, whose objective is equal to
			\begin{align*}
			-\floor{N\varepsilon}\hat{\alpha} +\sum_{i\in[N]}\hat{\bm \beta}_i^\top \bm e& = \frac{\sum_{i\in[N]}{\bm \omega_i^*}^\top \bm e-\floor{N\varepsilon}/N}{\gamma^*}> \frac{t(\sum_{i\in[N]}{\bm\omega_i^*}^\top \bm e-\floor{N\varepsilon}/N)}{\sum_{i\in[N]}{\bm\omega_i^*}^\top \bm e}\geq \frac{t}{\floor{N\varepsilon}+1}\geq v^{rel},
			\end{align*}
			where the first inequality is due to the fact that $\gamma^*< \sum_{i\in[N]}{\bm\omega_i^*}^\top \bm e/t$, the second inequality is because function $f(x)=(x-\floor{N\varepsilon}/N)/x$ is monotone increasing with respect to $x$ if $x\geq \floor{N\varepsilon}/N$, and the third one is due to the assumption that $t\geq v^{rel}(\floor{N\varepsilon}+1)>0$. This contradicts the weak duality that $ -\floor{N\varepsilon}\hat{\alpha} +\sum_{i\in[N]}\hat{\bm \beta}_i^\top \bm e\leq v^{rel}$.
			\QEDA
		\end{proof}

		\subsection{Proof of \Cref{prop_coveirng_ccp} }\label{proof_prop_coveirng_ccp}
		\propcoveirngccp*
		\begin{proof}
			Let us consider the following example.
			\begin{quote}
				\noindent\textbf{Example} Consider a CCP with $N$ equiprobable scenarios (i.e., $\Pr\{\tilde{\rxi}=\bm\xi^i\}=1/N$), risk level $\varepsilon>1/N$, set $\mathcal{X}=\Re_+^{\floor{N\varepsilon}+1}$, function $g(\bm{x},{\rxi})=1-{\rxi}^\top \bm{x}$, and 	$\bm\xi^i =\bm{e}_i$ for $ i\in[\floor{N\varepsilon}+1]$, $\bm\xi^i =\bm{e}$ for $ i\in[N]\setminus[\floor{N\varepsilon}+1]$. \QEDB
			\end{quote}
			In this example, the covering CCP \eqref{covering_ccp} reduces to
			\begin{equation*}
			v^*=	\min_{\bm{x}\in\Re^{\floor{N\varepsilon}+1}_+,\bm z \in\{0,1\}^N}\left\{ \sum_{i\in[\floor{N\varepsilon}+1]}x_i\colon
			\begin{aligned}
			& x_i\geq z_i, \forall i\in[\floor{N\varepsilon}+1],\\
			& \sum_{j\in[\floor{N\varepsilon}+1]}x_j\geq z_i,\forall i\in[N]\setminus[\floor{N\varepsilon}+1],\\
			& \sum_{i\in[N]}z_i \geq N-\floor{N\varepsilon}
			\end{aligned}
			\right\}
			\end{equation*}
			with the optimal value $v^*=1$.

			The corresponding $\alsox$ \eqref{covering} counterpart is
			\begin{align*}
			v^A= \min_{t} &\quad t,\\
			\text{s.t.}\quad&(\bm x^*, \bm s^*)\in	\argmin_{\bm{x}\in\Re^{\floor{N\varepsilon}+1}_+,\bm s\in\Re_+^{N}}\left\{ \frac{1}{N}\sum_{i\in[N]}s_i\colon
			\begin{aligned}
			& x_i\geq 1-s_i, \forall i\in[\floor{N\varepsilon}+1],\\
			& \sum_{j\in[\floor{N\varepsilon}+1]}x_j\geq 1-s_i,\forall i\in[N]\setminus[\floor{N\varepsilon}+1],\\
			& \sum_{i\in[\floor{N\varepsilon}+1]}x_i\leq t
			\end{aligned}
			\right\},\\
			& \sum_{i\in[N]}	\I({s^*_i}=0) \geq N-\floor{N\varepsilon}.
			\end{align*}
			For any $1\leq t<\floor{N\varepsilon}+1$, an optimal solution $(\bm x^*, \bm s^*)$ of the hinge-loss approximation is 
			$$x_j^*=\frac{t}{\floor{N\varepsilon}+1},\forall j\in [\floor{N\varepsilon}+1],\quad s_i^* =\begin{cases}
			\frac{\floor{N\varepsilon}+1-t}{\floor{N\varepsilon}+1}>0, &\text{ if }i\in [\floor{N\varepsilon}+1]\\
			0,&\text{otherwise}
			\end{cases}, \forall i \in [N],
			$$
			which violates the chance constraint since $ \sum_{i\in[N]}	\I({s^*_i}=0) =N-\floor{N\varepsilon}-1<N-\floor{N\varepsilon}$.
			On the other hand, if $t\geq \floor{N\varepsilon}+1$, the optimal solution $(\bm x^*, \bm s^*)$ of the hinge-loss approximation is $\bm x^*=\e,\bm{s}^*=\bm{0}$, which satisfies the chance constraint. Thus, in this example, we have $v^A=\floor{N\varepsilon}+1=(\floor{N\varepsilon}+1)v^*$. This completes the proof.
			\QEDA
			
		\end{proof}
		
		\vspace{15pt}
		
		\noindent\textbf{\large Proofs in Section~\ref{ap}}
		
		\vspace{5pt}

		\subsection{Proof of \Cref{alsox_convergence}}\label{proof_alsox_convergence}\alsoxconvergence*
		
		\begin{proof}
			At iteration $k+1$ of $\AM$ \Cref{alg_am}, the optimality condition of problem \eqref{basic_ap_z} implies that
			\begin{align*}
			\E\left[z^{k+1}(\tilde{\rxi})s^{k}(\tilde{\rxi})\right] \leq \E\left[z^k(\tilde{\rxi})s^{k}(\tilde{\rxi})\right].
			\end{align*}	
			Similarly,	the optimality condition of problem \eqref{basic_ap_1} implies that 
			\begin{align*}
			\E\left[z^{k+1}(\tilde{\rxi})s^{k+1}(\tilde{\rxi})\right] \leq \E\left[z^{k+1}(\tilde{\rxi})s^k(\tilde{\rxi})\right].
			\end{align*}
			Therefore, we have 
			\begin{align*}
			\E\left[z^{k+1}(\tilde{\rxi})s^{k+1}(\tilde{\rxi})\right] \leq \E\left[z^{k+1}(\tilde{\rxi})s^k(\tilde{\rxi})\right]\leq \E\left[z^k(\tilde{\rxi})s^{k}(\tilde{\rxi})\right].
			\end{align*}
			Thus, the sequence of $\E[z^k(\tilde{\rxi})s^k(\tilde{\rxi})]$ is monotonically nonincreasing. The fact that both $s^k({\rxi})$ and $z^k({\rxi})$ are nonnegative implies that $\E[z^k(\tilde{\rxi})s^k(\tilde{\rxi})]\geq 0$ at any iteration $k$. Hence, Monotone Convergence Theorem (see, e.g., thereon 7.16 in \citealt{rudin1964principles}) implies that the sequence of objective values $\{\E[z^{k}(\tilde{\rxi})s^{k}(\tilde{\rxi})]\}_{k\in \Ze_+}$ is convergent. \QEDA
		\end{proof}
		
		\vspace{15pt}
		
		\noindent\textbf{\large Proofs in Section~\ref{drccp}}
		
		\vspace{5pt}

		{\subsection{Proof of \Cref{the_infinity_drccp}}\label{proof_the_infinity_drccp}\theinfinitydrccp*}
		\begin{proof}
			We first prove the following claim.
			\begin{claim}\label{claim1}
				{For any $\Pr-$measurable function $f(\bm{\xi}):\Xi\rightarrow \Re$ with $\Pr\in \P_{\infty}^W$,} 
				we must have
				\[\sup _{\Pr\in\mathcal{P}^W_{\infty}} \E_\Pr[f(\tilde{\rxi})]=\E_{\Pr_{\tilde{\rzeta}}}\left[\sup_{\rxi}\left\{f(\rxi): \|\rxi-\tilde{\rzeta}\|\leq \theta\right\}\right].\]
			\end{claim}
			\proof
			It is sufficient to prove that 
			$$\lim_{q\rightarrow \infty}\sup _{\Pr\in\mathcal{P}^W_{q}} \E_\Pr[f(\tilde{\rxi})]=\E_{\Pr_{\tilde{\rzeta}}}\left[\sup_{\rxi}\left\{f(\rxi): \|\rxi-\tilde{\rzeta}\|\leq \theta\right\}\right].$$
			For any $q\geq 1$, according to theorem 1 in \cite{gao2016distributionally} or theorem 1 in \cite{blanchet2019quantifying}, $\sup _{\Pr\in\mathcal{P}^W_{q}} \E_\Pr[f(\tilde{\rxi})]$ can be reformulated as
			\begin{equation*}
			\sup _{\Pr\in\mathcal{P}^W_{q}} \E_\Pr[f(\tilde{\rxi})]= \min_{\lambda\geq 0} \lambda\theta^q+\E_{\Pr_{\tilde\rzeta}}\left[\sup_{{\bm{\xi}}}\left\{f(\rxi)-\lambda\left\|{\bm{\xi}}-\tilde{\rzeta}\right\|^q \right\} \right].
			\end{equation*}
			First of all, interchanging the expectation and inner supremum operator with the outer minimum operator, we arrive at the lower bound as
			\begin{equation*}
			\sup _{\Pr\in\mathcal{P}^W_{q}} \E_\Pr[f(\tilde{\rxi})]\geq \E_{\Pr_{\tilde\rzeta}}\left[\sup_{{\bm{\xi}}}\left\{f(\rxi)+\min_{\lambda\geq 0}\left\{\lambda\theta^q-\lambda\left\|{\bm{\xi}}-\tilde{\rzeta}\right\|^q \right\}\right\} \right]:=\E_{\Pr_{\tilde{\rzeta}}}\left[\sup_{\rxi}\left\{f(\rxi): \|\rxi-\tilde{\rzeta}\|\leq \theta\right\}\right].
			\end{equation*}
			Since the $q-$Wasserstein distance is monotone nondecreasing as $q$ increases (according to Jensen's inequality), thus $\{\sup _{\Pr\in\mathcal{P}^W_{q}} \E_\Pr[f(\tilde{\rxi})]\}_{q\geq 1}$ is a monotone nonincreasing sequence and is bounded from below. Thus, its limit exists and thus, we have
			\begin{equation*}
			\sup _{\Pr\in\mathcal{P}^W_{\infty}} \E_\Pr[f(\tilde{\rxi})]:=\lim_{q\rightarrow \infty}\sup _{\Pr\in\mathcal{P}^W_{q}} \E_\Pr[f(\tilde{\rxi})]\geq \E_{\Pr_{\tilde{\rzeta}}}\left[\max_{\rxi}\left\{f(\rxi): \|\rxi-\tilde{\rzeta}\|\leq \theta\right\}\right].
			\end{equation*}
			On the other hand, let {us} define a random vector $\tilde{\rxi}$ as $\tilde{\rxi}\in\arg\max_{\rxi}\{f(\rxi): \|\rxi-\tilde{\rzeta}\|\leq \theta\}$. According to the definition, we have the $\infty-$Wasserstein distance $W_{\infty}(\Pr_{\tilde{\rxi}}, \Pr_{\tilde{\rzeta}})$ no larger than $\theta$. That is, $\Pr_{\tilde{\rxi}}\in \mathcal{P}^W_{\infty}$ and 
			\[ \E_{\Pr_{\tilde{\rxi}}}[f(\tilde{\rxi})]=\E_{\Pr_{\tilde{\rzeta}}}\left[\max_{\rxi}\left\{f(\rxi): \|\rxi-\tilde{\rzeta}\|\leq \theta\right\}\right].\]
			Thus, the lower bound $\E_{\Pr_{\tilde{\rzeta}}}[\max_{\rxi}\{f(\rxi): \|\rxi-\tilde{\rzeta}\|\leq \theta\}]$ of $\sup _{\Pr\in\mathcal{P}^W_{\infty}} \E_\Pr[f(\tilde{\rxi})]$ is attainable. This concludes the proof.
			\QEDB
			
			\noindent For the distributionally robust chance constraint in \eqref{eq_drccp}, according to \Cref{claim1}, we have
			\begin{equation*}
			\inf_{\Pr\in \mathcal{P}_\infty^W} \Pr\left\{\tilde{\bm\xi}\colon g_i(\bm x,\tilde{\bm\xi})\leq 0, \forall i\in[I]\right\}= \Pr_{\tilde\rzeta}\left\{\tilde\rzeta\colon g_i(\bm x,\trxi)\leq 0,\forall \|\tilde{\bm\xi}-\tilde{\rzeta}\|\leq \theta, \forall i\in[I] \right\}.
			\end{equation*}
			{Recall the definition of $\bar g_i(\bm x,\rzeta)$, we have the following equivalent representation of \eqref{eq_drccp}
				\begin{equation*}
				\Pr_{\tilde\rzeta}\left\{{\tilde\rzeta}\colon \bar g_i(\bm x,\trzeta)\leq 0, \forall i \in [I]\right\} \geq 1-\varepsilon.
				\end{equation*}}
			This completes the proof.
			\QEDA
		\end{proof}

		{\subsection{Proof of \Cref{the_infinity_alsox}}\label{proof_the_infinity_alsox}\theinfinityalsox*}
		\begin{proof}
			{	We split the proof into two parts by proving the reformulations of the worst-case $\alsox$ and the worst-case $\CVaR$ approximation, separately.}
			\begin{enumerate} [label=(\roman*)]
				\item 	According to \Cref{claim1} in Appendix~\ref{proof_the_infinity_drccp}, we can rewrite the objective function of the worst-case hinge-loss approximation \eqref{drccp_alsox} under $\infty-$Wasserstein ambiguity set as
				\begin{align*}
				\sup _{\Pr\in\mathcal{P}^W_{\infty}} \E_\Pr\left[\max_{i\in [I]}g_i(\bm x,\rxi) _+\right]=	\E_{\Pr_{\tilde{\rzeta}}}\left\{ \max_{{\rxi}}\max_{i\in[I]}\left\{g_i(\bm x,\rxi)\right\}_+\colon {\left\|{\rxi}-\tilde\rzeta\right\|\leq \theta}\right\} .
				\end{align*}
				Interchanging the maximum operators, we have 
				\begin{align*}
				\sup _{\Pr\in\mathcal{P}^W_{\infty}} \E_\Pr\left[\max_{i\in [I]}g_i(\bm x,\rxi) _+\right]= 	\E_{\Pr_{\tilde{\rzeta}}}\left\{\max_{i\in[I]} \max_{{\rxi}}\left\{g_i(\bm x,\rxi)\right\}_+\colon {\left\|{\rxi}-\tilde\rzeta\right\|\leq \theta}\right\}.	
				\end{align*}
				{According to the definition of functions $\{\bar g_i(\cdot,\cdot)\}_{i\in [I]}$, we can further rewrite the objective function as
					\begin{align*}
					\sup _{\Pr\in\mathcal{P}^W_{\infty}} \E_\Pr\left[\max_{i\in [I]}g_i(\bm x,\rxi) _+\right]= 	\E_{\Pr_{\tilde{\rzeta}}}\left[\max_{i\in[I]}\bar g_i(\bm x,\rzeta)_+  \right].	
					\end{align*}
					Thus, the worst-case hinge-loss approximation \eqref{drccp_alsox} is equivalent to
					\begin{align*}
					\bm x^*\in \argmin _{\bm {x}\in \mathcal{X}}\left\{ \E_\Pr\left[\max_{i\in [I]}\bar g_i(\bm x,\tilde{\bm\xi})_+\right]\colon \bm{c}^\top \bm{x} \leq t \right\}.
					\end{align*}
					According to \Cref{the_infinity_drccp}, the worst-case chance constraint is equivalent to the regular chance constraint \eqref{eq_drccp_infty}. Therefore, we arrive at the reformulation \eqref{drccp_alsox_infity_eq}.}
				
				\item In the worst-case $\CVaR$ approximation  \eqref{worse_case_cvar_q}, we can interchange the supremum operator with the minimum one, since $\Pr_{\tilde\rzeta}$ is sub-Gaussian, the definition of $\infty-$Wasserstein ambiguity set shows that for any $t\geq 0$ and $\Pr\in\P_{\infty}^W$, we have
				$$\Pr\{\tilde{\bm\xi}: \|\tilde{\bm\xi}\|\geq t+\theta\} =\Pr_{\tilde{\rzeta}}\{\tilde{\rzeta}: \|\tilde{\bm\xi}\|\geq t+\theta,\|\tilde{\bm\xi}-\tilde{\rzeta}\|\leq \theta\}\leq \Pr_{\tilde{\rzeta}}\{\tilde{\rzeta}: \|\tilde{\rzeta}\|\geq t\}\leq C_1e^{-C_2t^2},$$
				for some positive constants $C_1,C_2>0$, and thus $\P_{\infty}^W$ is weakly compact. Thus, {according to corollary \cite{terkelsen1972some}, the worst-case $\CVaR$ approximation is equivalent to}
				\begin{equation*}
				v_\infty^\CVaR=\min _{\bm {x}\in \mathcal{X}}\left\{\bm c^\top\bm x\colon \min_{\beta}\left[ \beta+\frac{1}{\varepsilon}\sup_{\Pr\in \mathcal{P}^W_{\infty}}\E_{\Pr}\left\{\max_{i\in[I]}\left\{g_i(\bm x,\rxi)\right\} -\beta\right\}_+\right]\leq 0 \right\}.
				\end{equation*}
				According to \Cref{claim1} in Appendix~\ref{proof_the_infinity_drccp}, the worst-case $\CVaR$ approximation becomes
				\begin{align*}
				v_\infty^\CVaR=\min _{\bm {x}\in \mathcal{X}}\left\{\bm c^\top\bm x\colon\min_{\beta}\left[ \beta+\frac{1}{\varepsilon} \E_{\Pr_{\tilde\rzeta}}\left[\max_{\rxi}	\left\{ \max_{i\in[I]}\left\{g_i(\bm x,\rxi)\right\} -\beta\right\}_+\colon {\left\|{\rxi}-\tilde{\rzeta}\right\|\leq \theta}\right]\right]\leq 0 \right\}.
				\end{align*}
				Interchanging the maximum operators and taking the optimization over $\rxi$, we have 
				{\begin{align*}
					v_\infty^\CVaR=\min _{\bm {x}\in \mathcal{X}}\left\{\bm c^\top\bm x\colon \min_{\beta}\left[ \beta+\frac{1}{\varepsilon} \E_{\Pr_{\tilde\rzeta}}\left\{\max_{i\in[I]} \left\{\bar g_i(\bm x,\rzeta)\right\}-\beta\right\}_+\right]\leq 0 \right\}.
					\end{align*}}
				This completes the proof.
				\QEDA
			\end{enumerate}
		\end{proof}

		\subsection{Proof of \Cref{the_infinity}}\label{proof_the_infinity}\theinfinitytheorem*
		\begin{proof}
			According to {\Cref{the_infinity_alsox}}, the worst-case $\alsox$ and the worst-case $\CVaR$ approximation correspond to the same regular chance constrained program. Based on \Cref{the1}, we know that for a regular CCP, $\alsox$ yields a better 
			solution than that of $\CVaR$ approximation. Thus, the worst-case $\alsox$ can return a better solution than that of the worst-case $\CVaR$ approximation for DRCCP under $\infty-$Wasserstein ambiguity set.
			\QEDA
		\end{proof}

		\subsection{Proof of \Cref{exact_el_drccp}}\label{proof_exact_el_drccp}\exacteldrccp*
		
		\begin{proof}	\begin{subequations}
				{	According to \Cref{hinge_loss_truncated_el}  and \Cref{the_infinity_alsox},}	for a given $t$,  the worst-case hinge-loss approximation under $\infty-$Wasserstein ambiguity set is equivalent to
				{	\begin{equation}\label{eq_hinge_ellitptical2}
					(\bm{x}^*,\alpha^*) \in	\argmin_{\bm{x}\in \mathcal{X},\alpha }\biggl\{\sqrt{\bm{a}_1(\bm{x})^\top\bm{\mathrm{{\Sigma}}}\bm{a}_1(\bm{x})} \left[ f({\alpha}-\theta)\right]\colon \bm{c}^\top \bm{x} \leq t ,\frac{b_1(\bm x)-\bm\mu^\top \bm{a}_1(\bm {x})}{\sqrt{\bm{a}_1(\bm{x})^\top\bm{\mathrm{{\Sigma}}}\bm{a}_1(\bm {x})}} =\alpha \biggr\},
					\end{equation}}
				where 
				$f(\alpha)=\overline{G}(\alpha^2/2)-\alpha+\alpha\mathrm{\Phi}(\alpha)$.  
				
				Next, we split the proof into two parts by checking two sufficient conditions separately.
				\begin{enumerate} [label=(\roman*)]
					\item Suppose that $\mathcal{X}\subseteq\{\bm x:\sqrt{\bm{a}_1(\bm{x})^\top\bm{\mathrm{{\Sigma}}}\bm{a}_1(\bm{x})}=C\} $, where $C$ is a positive constant. Then, the worst-case hinge-loss approximation can be simplified as
					{	\begin{equation}\label{eq_hinge_ellitptical3}
						(\bm{x}^*,\alpha^*) \in	\argmin_{\bm{x}\in \mathcal{X},\alpha }\biggl\{F_1({\alpha}):=C\left[ f(\alpha-\theta)\right]\colon \bm{c}^\top \bm{x} \leq t,
						\frac{b_1(\bm x)-\bm\mu^\top \bm{a}_1(\bm {x})}{C} =\alpha\biggr\}.
						\end{equation}}
					Since the first-order derivative of $F_1(\cdot)$ is 
					\begin{equation*}
					\frac{\partial F_1({\alpha}) }{\partial {\alpha}} = {C} (\mathrm{\Phi}(\alpha-\theta)-1) <0,
					\end{equation*}
					{	function $	F_1(\alpha)$ is monotone decreasing over ${\alpha}\in\Re$. }
					
					{According to problem \eqref{ccp_el} and \Cref{the_infinity_drccp}}, we can rewrite 
					DRCCP  \eqref{eq_drccp} as
					{	\begin{align}
						v_\infty^*  =\min _{\bm {x}\in \mathcal{X}} \left\{\bm{c}^\top\bm{x}\colon b_1(\bm x)-\bm{\mu}^\top\bm{a}_1(\bm{x})\geq (\Phi^{-1}(1-\varepsilon)+\theta) \sqrt{\bm{a}_1(\bm{x})^\top\bm{\mathrm{{\Sigma}}}\bm{a}_1(\bm{x})}\right\}.
						\label{eq_drccp_infty1}
						\end{align}}
					Thus, for any $t\geq v_\infty^*$, there exists a feasible solution $\bar{\bm{x}}$ to DRCCP \eqref{eq_drccp_infty1}, such that
					{	\begin{align*}
						\bm{\mu}^\top\bm{a}_1(\bar{\bm{x}})-b_1(\bar{\bm{x}})+(\Phi^{-1}(1-\varepsilon)+\theta) {C}\leq 0 , \bm{c}^\top \bar{\bm{x}}\leq t.
						\end{align*}
						Let $\alpha'=1/{C}(b_1(\bar{\bm{x}})-	\bm{\mu}^\top\bm{a}_1(\bar{\bm{x}})) \geq \Phi^{-1}(1-\varepsilon)+\theta$.} Then, $(\bar{\bm{x}},\alpha')$ is feasible to the worst-case hinge-loss approximation \eqref{eq_hinge_ellitptical3}.
					
					Due to the monotonicity of the objective function $F_1(\cdot)$, we must have {$\alpha^*\geq \alpha'\geq \Phi^{-1}(1-\varepsilon)+\theta$}, i.e., 
					\begin{align*}
					\bm{\mu}^\top\bm{a}_1({\bm{x}}^*)-b_1({\bm{x}^*})+{C}{\alpha}^*\leq 0.
					\end{align*}
					Hence, we must have $\bm{x}^*$ is also feasible to DRCCP
					\eqref{eq_drccp_infty1}. This implies that the optimal value of the worst-case $\alsox$ \eqref{drccp_alsox_formulation} must be $v_\infty^A\leq v_\infty^*$. On the hand, we always have $v_\infty^A\geq v^*$. Thus, $v_\infty^A=v_\infty^*$.
					
					\item Suppose that $\mathcal{X}\subseteq\{\bm x:b_1(\bm x)-\bm\mu^\top \bm{a}_1(\bm {x})=C\} $, where $C$ is an arbitrary constant. Let us denote $\sigma = \sqrt{\bm{a}_1(\bm{x})^\top\bm{\mathrm{{\Sigma}}}\bm{a}_1(\bm{x})} $.
					Then the worst-case hinge-loss approximation can be simplified as 
					{	\begin{equation}\label{eq_hinge_ellitptical4}
						(\bm{x}^*,\hat{\alpha}^*,\sigma^*) \in	\argmin_{\bm{x}\in \mathcal{X},\hat{\alpha},\sigma }\biggl\{F_2(\sigma,\hat{\alpha}):=\sigma f(\hat{\alpha})\colon \bm{c}^\top \bm{x} \leq t ,
						\hat\alpha=\frac{C}{\sigma},
						\sigma = \sqrt{\bm{a}_1(\bm{x})^\top\bm{\mathrm{{\Sigma}}}\bm{a}_1(\bm{x})}\biggr\}.
						\end{equation}}
					The first-order derivative of $F_2(\cdot,\cdot)$ with respect to $\sigma$ is 
					{\begin{equation*}
						\frac{\partial F_2(\sigma,\hat{\alpha}) }{\partial \sigma} =  f(\hat{\alpha})+ \sigma \frac{ f(\hat{\alpha})}{\partial \hat{\alpha}} \frac{\partial \hat{\alpha} }{\partial \sigma} = f(\hat{\alpha})+ (1-\mathrm{\Phi}(\hat{\alpha}))\frac{C}{\sigma}>0.
						\end{equation*}}
					Thus, function $F_2(\sigma,\hat{\alpha})$ is monotone increasing over $\sigma\in\Re_+$. Thus, for any $t\geq v_\infty^*$, i.e., there exists a feasible solution $\hat{\bm{x}}$ to DRCCP \eqref{eq_drccp_infty} such that
					{	\begin{align*}
						C\geq (\Phi^{-1}(1-\varepsilon)+\theta)\sqrt{\bm{a}_1(\hat{\bm{x}})^\top\bm{\mathrm{{\Sigma}}}\bm{a}_1(\hat{\bm{x}})} , \bm{c}^\top \hat{\bm{x}}\leq t.
						\end{align*}}
					Let $\hat\sigma=\sqrt{\bm{a}_1(\hat{\bm{x}})^\top\bm{\mathrm{{\Sigma}}}\bm{a}_1(\hat{\bm{x}})}$ and $\hat{\alpha}'=(C/\hat\sigma)$. Then, $(\hat{\bm{x}},\hat{\alpha}',\hat\sigma)$ is feasible to the worst-case hinge-loss approximation \eqref{eq_hinge_ellitptical4}. Due to the monotonicity of the objective function $F_2(\cdot,\cdot)$ with respect to $\sigma$, we must have {${\sigma}^*\leq \hat{\sigma}\leq C/(\Phi^{-1}(1-\varepsilon)+\theta)$}, i.e., 
					{	\begin{align*}
						b_1({\bm{x}}^*)-\bm{\mu}^\top\bm{a}_1({\bm{x}}^*):=C\geq 
						(\Phi^{-1}(1-\varepsilon)+\theta){\sigma}^*=(\Phi^{-1}(1-\varepsilon)+\theta)
						\sqrt{\bm{a}_1(\bm{x}^*)^\top\bm{\mathrm{{\Sigma}}}\bm{a}_1(\bm{x}^*)}.
						\end{align*}}
					Hence, we must have ${\bm{x}}^*$ is also feasible to DRCCP
					\eqref{eq_drccp_infty}, i.e., $v_\infty^A\leq v_\infty^*$. On the other hand, we always have $v_\infty^A\geq v_\infty^*$. Thus, $v_\infty^A=v_\infty^*$.
					\QEDA
				\end{enumerate}
			\end{subequations}
		\end{proof}

		\newpage
		\section{Examples}\label{sec_append_example}
		\begin{example}
			\rm
			\label{example_alsox_cvar}
			Consider a CCP with 3 equiprobable scenarios (i.e., $N=3$ and $\Pr\{\tilde{{\xi}}=\xi^i\}=1/N$), risk level $\varepsilon=1/2$, set $\mathcal{X}=\Re_+$, function $g(\bm{x},{\xi})=-x+{\xi}$,
			and $\xi^1=3$, $\xi^2=2$, $\xi^3=1$. Then the optimal solution of this CCP \eqref{eq_ccp} can be obtained by solving the following mixed-integer linear program
			\begin{equation*}
			v^*=\min_{x\geq 0,\bm z}\left\{x\colon x\geq 3z_1, x\geq 2z_2, x\geq z_3, \sum_{i\in[3]}z_i\geq 2, \bm z\in\{0,1\}^3 \right\}.
			\end{equation*}
			Its $\alsox$ counterpart admits the following form
			{
				\begin{align*}
				v^A =\min _{ {t}}\, &\biggl\{ t\colon  \sum_{i\in[3]}	\I({s^*_i}=0) \geq 2,\\
				\quad \quad&( x^*,\bm s^*)\in\argmin _{x\geq 0,\bm s\geq \bm{0} }\biggl\{\frac{1}{3} \sum_{i\in[3]}s_i\colon x\geq 3-s_1, x\geq 2-s_2, x\geq 1-s_3,x\leq t\biggr\}\biggr\}.
				\end{align*} }
			%
			The $\CVaR$ approximation is
			\begin{equation*}
			v^\CVaR=\min_{x\geq 0,\beta\leq 0,\bm s}\left\{x\colon 3-x\leq s_1,2-x\leq s_2, 1-x\leq s_3, (s_1+s_2+s_3)/3-\beta/2\leq 0, s_i\geq \beta,\forall i\in[3] \right\}.
			\end{equation*}
			By the straightforward calculation, we obtain $v^*=2$, $v^A=2$, and $v^{\CVaR}=8/3$. \Cref{alsox_cvar_example_figure} illustrates their relationships, where the dotted line segment on the $x$-axis represents the feasible region.
			When $t=8/3$, the optimal solution from the hinge-loss approximation \eqref{eq_also_x_2b} is $s_1^*=1/3, s_2^*=s_3^*=0$, which means the second constraint and the third constraint are satisfied, while the first constraint is violated. The support size of $\bm{s}^*$ is 1, i.e., $|\mathrm{supp}(\bm{s}^*)|=1$, so the current solution is feasible to the upper-level problem in $\alsox$ \eqref{eq_also_x_2} and we {decrease} $t$. Finally, we can show that $\alsox$ has an optimal $t^*=2$. 
			\begin{figure}[h]\centering
				\begin{tikzpicture}
				\begin{axis}[
				domain=-1:4,
				xmin=0.5, xmax=3.55,
				samples=200,
				axis x line=middle,
				axis y line=none,
				xlabel={$x$},
				]
				\addplot[densely dotted,line width=2pt, domain = 2:3.5,-]
				{0};
				
				\filldraw[red] (1,0) circle (2pt);
				\filldraw[red] (2,0) circle (2pt);
				\filldraw[red] (3,0) circle (2pt);
				
				\draw[dashed,-] (8/3,0.3)--(8/3,-0.2) node[below] {$v^{\CVaR}$};
				\node [draw=none,-] at (2.8,0.4) {$t=\frac{8}{3}$} ;
				\draw[dashed,-] (2,0.3)--(2,-0.2) node[below] {$v^*,v^A$};
				\node [draw=none] at (2.1,0.4) {$t=2$} ;
				\draw [decorate,decoration={brace,amplitude=4pt},-]
				(8/3+0.01,0) -- (3-0.01,0) node [black,midway,above] 
				{\footnotesize $s_1$};
				\end{axis}
				\end{tikzpicture} 
				\caption{Illustration of \Cref{example_alsox_cvar}}. 
				\label{alsox_cvar_example_figure}
			\end{figure}
			\QEDB	
		\end{example}

		\begin{example}
			\rm
			\label{exactness_set_covering_example}
			Consider a CCP with 3 equiprobable scenarios (i.e., $N=3$, $\Pr\{\tilde{\rxi}=\bm\xi^i\}=1/N$), risk level $\varepsilon=1/3$, set $\mathcal{X}=\{0,1\}^2$, function $g(\bm{x},{\rxi})=1-{\rxi}^\top \bm{x}, $ and 
			$\bm\xi^1 =(1,0)^\top$, $\bm\xi^2 =(0,1)^\top$, $\bm\xi^3 =(1,1)^\top$.
			The optimal solution of this CCP can be obtained by solving the following mixed-integer linear program
			\begin{equation*}
			v^*=\min_{\bm{x}\in \{0,1\}^2,\bm{z}\in \{0,1\}^3}\left\{ x_1+2x_2\colon x_1\geq z_1,x_2\geq z_2,x_1+x_2\geq z_3,\sum_{i\in[3]}z_i\geq 2 \right\}
			\end{equation*}	
			with the optimal value $v^*=1$. \\
			The corresponding $\alsox$ \eqref{eq_also_x_2} is
			{
				\begin{align*}
				v^A =&\min _{t}\,\biggl\{ t\colon  \sum_{i\in[3]}	\I({s^*_i}=0) \geq 2, \\
				&( \bm x^*,\bm s^*)\in \argmin_{\bm{x}\in \{0,1\}^2,\bm{s}\geq \bm 0}\biggl\{\frac{1}{3} \sum_{i\in[3]}s_i\colon x_1\geq 1-s_1, x_2\geq 1-s_2, x_1+x_2\geq 1-s_3, x_1+2x_2\leq t\biggr\}\biggr\}.
				\end{align*}}
			For any $1<t<3$, the optimal solution is $x_1^*=1$, $ x_2^*=0$, $s_1^*=s_3^*=0$, $s_2^*=1>0$, the support size of $\bm{s}^*$ is $1$. Thus, this solution is feasible to CCP and we can decrease $t$ until $t=1$. Thus, the optimal value of $\alsox$ is also $v^A=1=v^*$. 
			\QEDB
		\end{example}
		
				\begin{example}
			\rm
			\label{exactness_el_example}
			Consider a single linear CCP with a Gaussian distribution (i.e., $\tilde{\rxi}\thicksim\mathcal{N}(\bm{\Bar{\mu}},\bm{\Bar{\mathrm{\Sigma}}})$) with $n=2,\bm{\Bar{\mu}}= [2,1]^\top$, $ \bm{\Bar{\mathrm{\Sigma}}}=\left[ \begin{smallmatrix} 1 & 0\\0 & 1 \end{smallmatrix} \right]$, 
			risk level $\varepsilon=0.05$, set $\mathcal{X}=\Re^2$ and function $g(\bm{x},{\rxi})=-1+{\rxi}^\top \bm{x}$. This example violates both conditions in \Cref{exact_el}. We {show} that $v^A>v^*$.
			
			First, in this example, CCP \eqref{ccp_el} becomes 
			\begin{equation*}
				v^* = \min_{\bm x\in\Re^2} \left\{-x_1-3x_2:1-2x_1-x_2\geq\mathrm{\Phi}^{-1}(1-\varepsilon)\sqrt{x_1^2+x_2^2} \right\},
			\end{equation*}
			and its approximate optimal value is $v^* =-1.55432$ with error bound $[-10^{-7}, 10^{-7}]$.
			
			The corresponding $\alsox$ \eqref{eq_alsox_el} is
			\begin{align*}
				v^A=\min_{t}\, & \biggl\{t\colon 1-2x_1^*-x_2^*\geq\mathrm{\Phi}^{-1}(1-\varepsilon)\sqrt{(x^*_1)^2+(x^*_2)^2}\\
				&\quad (x_1^*,x_2^*)\in\argmin_{\bm x\in\Re^2} \biggl\{ \sqrt{x_1^2+x_2^2} \bigg[ \varphi(\frac{1-2x_1-x_2}{\sqrt{x_1^2+x_2^2}}) -\frac{1-2x_1-x_2}{\sqrt{x_1^2+x_2^2}}\\ \quad\quad\quad\quad\quad\quad\quad&\quad+\frac{1-2x_1-x_2}{\sqrt{x_1^2+x_2^2}}\mathrm{\Phi}(\frac{1-2x_1-x_2}{\sqrt{x_1^2+x_2^2}}) \bigg] \colon -x_1-3x_2\leq t\biggr\}
				\biggr\}.
			\end{align*}
			
			Suppose that $t = -1.42$, an approximate optimal solution of the hinge-loss approximation \eqref{eq_alsox_elb} is $x_1^* \approx -0.375511$ with error bound $[-10^{-6},10^{-7}]$ and
			$	x_2^*\approx 0.598504 $ with error bound $[-3\times 10^{-7},10^{-7}]$. We see that any possible solution within the error box, i.e.,
			\begin{align*}
				&\max_{x_1,x_2}\left\{1-2x_1-x_2:x_1-x_1^*\in [-10^{-6},10^{-7}],x_2-x_2^*\in [-3 \times 10^{-7},10^{-7}] \right\}< 1.1526\\ <1.1620<&\max_{x_1,x_2}\left\{\mathrm{\Phi}^{-1}(1-\varepsilon)\sqrt{(x_1)^2+(x_2)^2}:x_1-x_1^*\in [-10^{-6},10^{-7}],x_2-x_2^*\in [-3\times10^{-7},10^{-7}] \right\}.
			\end{align*} 
			Therefore, we must have $v^A\geq -1.42>-1.55\geq v^*$, i.e., the solution from $\alsox$ \eqref{eq_alsox_el} is not exactly optimal to the CCP.
			\QEDB
		\end{example}
		
		\begin{example}
			\rm
			\label{alsox_not_exact_example}
			Consider a CCP with $3$ equiprobable scenarios (i.e., $\Pr\{\tilde{\rxi}=\bm\xi^i\}=1/3$ for each $i\in [3]$), risk level $\varepsilon=1/3$, set $\mathcal{X}=\Re_+^2$, function $g(\bm{x},{\rxi})=1-{\rxi}^\top \bm{x}$, and 	$\bm\xi^1 =(1,0)^\top$, and $\bm\xi^2=\bm\xi^3 =(1,1)^\top$. In this case, 
			the CCP is equivalent to the following mixed-integer linear program
			\begin{equation*}
			v^*=	\min_{\bm{x}\in\Re^2_+,\bm z \in\{0,1\}^3}\left\{ 3x_1+2x_2\colon x_1\geq z_1, x_1+x_2\geq z_2,x_1+x_2\geq z_3, z_1+z_2+z_3\geq 2	\right\}
			\end{equation*}
			with optimal value $v^*=2$.
			
			The corresponding $\alsox$ \eqref{eq_also_x_2} is
			{
				\begin{align*}
				&v^A= \min_{t} \, \biggl\{ t\colon \sum_{i\in[3]}	\I({s^*_i}=0) \geq 2,\\
				&(\bm x^*, \bm s^*)\in	\argmin_{\bm{x}\in\Re^2_+,\bm s \in\Re^3_+}\biggl\{ \frac{1}{3}\sum_{i\in[3]}s_i\colon 
				\begin{array}{l}
				\displaystyle  x_1\geq 1-s_1,x_1+x_2\geq 1-s_2,\\
				\displaystyle x_1+x_2\geq 1-s_3,3x_1+2x_2\leq t \end{array}\biggr\}\biggr\},
				\end{align*}}
			with the optimal $v^A=3$.
			
			For any $t\in [2,3)$, an optimal solution to the hinge-loss approximation is $x_1^*=t/3,x_2^*=0,s_1^*=s_2^*=s_3^*=1-t/3$. Invoking the $\AM$ \Cref{alg_am} with initial $s^0_i=s_i^*$ for each $i\in [3]$, we see that $(\bm x^*,\bm s^*,\bm z^*)$ with $z^*_1=1,z^*_2=1,z^*_3=0$ {is} a stationary point of the $\AM$ \Cref{alg_am}. Therefore, in this example, $\alsox +$ \Cref{alg_alsox_+} with the tolerance $\delta_1=0$ provides the same solution as $\alsox$ \eqref{eq_also_x_2}, and both fail to find an optimal solution of the CCP.
			\QEDB
		\end{example}

		\newpage
		\section{An Illustration of $\alsoxt$, $\CVaRt$ Approximation, and $\alsoxt+$ Algorithm}\label{Illustration_section}
		We use \Cref{am_dc_example} to provide a simple illustration of $\alsox$, $\CVaR$ approximation, and $\alsox+$ algorithm. The results are shown in \Cref{am_dc_example_figure}, where the non-convex shaded region denotes the feasible region of the corresponding CCP studied in \Cref{am_dc_example}, and points D and E are its optimal solutions with the optimal value $v^* = 0.5$. Point F is the best solution from $\CVaR$ approximation, which is quite far away from the true optimal solution. 
		
		Given $t=0.6$, the hinge-loss approximation \eqref{eq_also_x_2b} is to minimize the average of violations for all constraints. Due to the symmetry of the random parameters, we see that the interval between point A and point C is the set of its optimal solutions. If one were unlucky and picked any solution inside the interval (e.g., point B) rather than the boundary points, such a choice would end up with an infeasible solution to the CCP. 
		On the contrary, $\alsox +$ \Cref{alg_alsox_+} with the tolerance $\delta_1=0$ {breaks} the symmetry by circumventing the infeasible solutions like point B, and provide a better solution. For example, when $t= 0.6$ and $\alsox +$ \Cref{alg_alsox_+} starts at point B, it {selects} the two smallest constraint violations and move the solution to either point A or point C, which is feasible to the CCP. 
		
		Hence, in \Cref{am_dc_example}, we show that $\alsox+$ algorithm can find an optimal solution, while both the $\CVaR$ approximation and $\alsox$ \eqref{eq_also_x_2} may not.
		\begin{figure}[h]\centering
			\begin{tikzpicture}[scale=1]
			\begin{axis}[scale=1,
			unit vector ratio*=1 1 1,
			domain=-0.19:1.10,
			xmin=-0.19, xmax=1.10,
			ymin=-0.19, ymax=1.10,
			samples=200,
			axis y line=center,
			axis x line=middle,
			xtick={0,0.2,0.4,0.6,0.8,1.0},
			ytick={0,0.2,0.4,0.6,0.8,1.0},
			xlabel={$x_1$},
			ylabel={$x_2$}
			]
			\filldraw [line width=0pt,fill=gray!20](0,1.08) -- (0,0.5) -- (1/3,1/3)-- (1/2,0)-- (1.08,0)-- (5,5);
			\addplot+[mark=none,blue,-][domain=0:0.5] {1-2*x} node[pos=0.1](blue1){};
			\node [right,color=blue,-] at (blue1) {}; 
			\draw [blue, -> ] (0.1,0.8) -- (0.2,0.8) node [right] {\footnotesize\footnotesize\footnotesize{$2x_1 + x_2 \geq 1$}};
			\addplot+[mark=none,blue,-][domain=0:1] {0.5-1/2*x} node[pos=0.7](blue2){};
			\node [right,color=blue] at (blue2) {};
			\draw [blue, -> ] (0.8,0.1) -- (0.8,0.2) node [above] {\footnotesize\footnotesize{$x_1 + 2x_2 \geq 1$}};
			
			\addplot+[mark=none,blue,-][domain=0:0.5] {1/3-2/3*x} node[pos=0.7](blue3){};
			\node [left,color=blue] at (blue3) {};
			\addplot+[mark=none,black,style=dashed,-][domain=0:0.2] {0.6-x} node[pos=0](black1){};
			\addplot+[mark=none,black,style=dashed,-][domain=0.4:0.6] {0.6-x} node[pos=0](black1){};
			\draw [blue, -> ] (0.2,0.2) -- (0.18,0.1) node [below] {\footnotesize{$2x_1 + 3x_2 \geq 1$}};
			

			
			
			\addplot+[mark=none,black,style=dashed,-] [domain=0:2/3]{2/3-x} node[pos=0](black1){};
			
			\addplot+[mark=none,black,style=dashed,-][domain=0:0.5] {1/2-x} node[pos=0.15](black1){};
			
			\draw [red, -> ] (0.35,0.15) -- (0.3,-0.1) node [below] {\footnotesize$t=0.5$};
			
			\draw [red, -> ] (0.55,0.05) -- (0.52,-0.1) node [below] {\footnotesize$t=0.6$};
			
			\draw [red, -> ] (2/3,0) -- (0.75,-0.1) node [below] {\footnotesize$t=2/3$};
			
			
			
			\draw [very thick,dash dot,-] (0.2,0.4) -- (0.4,0.2);
			
			\filldraw[red] (0.2,0.4) circle (2pt) node[anchor=east] {A}; 
			\filldraw[red] (0.3,0.3) circle (2pt) node[anchor=east] {B}; 
			\filldraw[red] (0.4,0.2) circle (2pt) node[anchor=east] {C}; 
			\filldraw[red] (0.5,0) circle (0pt) node[anchor=south] {D}; 
			\filldraw[red] (1/3,1/3) circle (2pt) node[anchor=south west] {F};
			\filldraw[red] (0,0.5) circle (0pt) node[anchor=west] {E}; 
			\node[mark size=2pt,color=olive] at (0.5,0) {\pgfuseplotmark{square*}};
			\node[mark size=2pt,color=olive] at (0,0.5) {\pgfuseplotmark{square*}};
			\end{axis}	
			\end{tikzpicture}
			\caption{Illustration of $\alsox$, $\CVaR$ Approximation, and $\alsox +$ Algorithm using \Cref{am_dc_example}. Point D and point E, marked by solid square, denote the optimal solutions of the CCP, one of which {is} found by the $\alsox +$ algorithm. Point F shows the solution found by $\alsox$ and $\CVaR$. Three dashed lines denote objective function lines of the CCP with values equal to $t=2/3,0.6,0.5 $ (from top to bottom), respectively. In $\alsox$, when $t=0.6$, point A, point B, and point C are three distinct optimal solutions.}
			\label{am_dc_example_figure}
		\end{figure}
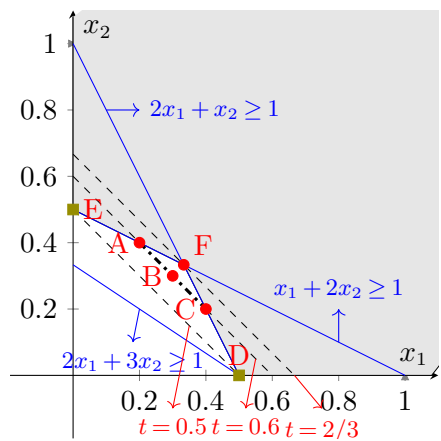

		\newpage
{	\section{The Closedness of the Feasible Region of Chance Constraint in CCP \eqref{eq_ccp}}\label{sec_append_lower_continuous_closed}
}
	{	\begin{proposition}
			\label{prop_lower_continuous_closed}
		Suppose set $\X\subseteq \Re^n$ is closed and function $g(\bm x,\trxi)$ is lower semi-continuous with respect to $\bm x$ with probability 1, then the feasible region of CCP \eqref{eq_ccp}
			\begin{align*}
				\X_1 = \left\{ \bm x\in \X\colon \Pr\left\{ g(\bm x,\trxi)\leq 0 \right\}\geq 1-\varepsilon\right\}
			\end{align*}
			is closed. 
		\end{proposition}
		\begin{proof}
			For any sequence $\{\bm x_i\}\in\X_1$ converging to $\bm x_0$, we want to prove that $\bm x_0\in\X_1$.
			Since $\Pr\{ g(\bm x,\trxi)\leq 0 \} = \E[\I(g(\bm x,\trxi)\leq 0)]$,  then we can write set  $	\X_1$ as
			\begin{align*}
				\X_1 =  \left\{ \bm x\in \X\colon \E[\I(g(\bm x,\trxi)\leq 0)] \geq 1-\varepsilon\right\}.
			\end{align*} 
			Since the sequence $\{\bm x_i\}\subseteq\X_1$, we have 
			\begin{align*}
				\limsup_{i\to\infty}  \E[\I(g(\bm x_i,\trxi)\leq 0)]   \geq 1-\varepsilon.
			\end{align*}
			According to Fatou's lemma (see, e.g., section 4 in \citealt{royden1988real}), we have 
			\begin{align*}
				\E\left[	\limsup_{i\to\infty} \I(g(\bm x_i,\trxi)\leq 0)\right] \geq 	\limsup_{i\to\infty}  \E[\I(g(\bm x_i,\trxi)\leq 0)].
			\end{align*}
			Since the indicator function is upper semi-continuous, we have 
			\begin{align*}
				\E\left[	\I( \limsup_{i\to\infty}  g(\bm x_i,\trxi)\leq 0)\right] \geq  \E[	\limsup_{i\to\infty} \I(g(\bm x_i,\trxi)\leq 0)].
			\end{align*}
			Since the indicator function is nonincreasing and the fact that $\limsup_{i\to\infty} g(\cdot,\cdot)\geq \liminf_{i\to\infty} g(\cdot,\cdot)$, we have
			\begin{align*}
				\E[	\I( \liminf_{i\to\infty}  g(\bm x_i,\trxi)\leq 0)] \geq   \E[	\I( \limsup_{i\to\infty}  g(\bm x_i,\trxi)\leq 0)].
			\end{align*}
			According to the assumption that function $g(\bm x,\rxi)$ is lower semi-continuous  and the fact that the indicator function is nonincreasing, we have
			\begin{align*}
				\E[	\I(   g(\bm x_0,\trxi)\leq 0)] \geq  \E[	\I( \liminf_{i\to\infty}  g(\bm x_i,\trxi)\leq 0)],
			\end{align*}
			which implies that  $ \E[	\I(   g(\bm x_0,\trxi)\leq 0)]\geq 		 	\limsup_{i\to\infty}  \E[\I(g(\bm x_i,\trxi)\leq 0)]  \geq 1-\varepsilon$. Thus,  $\bm x_0\in\X_1$, which completes the proof.
			\QEDA 
		\end{proof} 
			We remark that \Cref{prop_lower_continuous_closed}  generalizes proposition 1.7. of \citealt{kall1994stochastic}, where the authors showed that when function $g(\bm x,\rxi)$ is continuous,  the feasible region of CCP \eqref{eq_ccp}  is closed. 
	}
		
		\newpage
		{\section{Tractability of $\alsoxt$ Under Discrete Support or Elliptical Distributions}
			\label{alsox_tractable}}
		{\subsection{Tractability of $\alsoxt$ Under Discrete Support}
			\label{alsox_discrete_tractable}
			If the underlying probability distribution is finite-support with $N$ scenarios, i.e., the random vector $\trxi$ has a finite support $\Xi=\{\rxi^1,\cdots,\rxi^N\}$ with $\Pr\{\trxi=\rxi^i\}=p_i$ for all $i\in[N]$, then CCP \eqref{eq_ccp} reduces to
			\begin{align}
			v^* = \min_{\bm {x}\in \mathcal{X}}\left\{\bm{c}^\top \bm{x} \colon \sum_{i\in [N]}p_i\I(g(\bm x,\bm\xi^i)\leq 0) \geq 1-\varepsilon\right\},\label{eq_ccp_general}
			\end{align}
			and by projecting out functional variable $s(\cdot)$, $\alsox$ \eqref{eq_also_x_2} admits the following form
			\begin{align}
			v^A =\min _{ {t}}\quad & t,\nonumber\\
			\text{s.t.}\quad& \bm x^*\in\argmin_{\bm {x}\in \mathcal{X}}\left\{ \sum_{i\in[N]}p_i[g(\bm x,\bm\xi^i)]_+\colon \bm{c}^\top \bm{x} \leq t \right\},\label{eq_also_x_discrete_general}\\
			& \sum_{i\in[N]}p_i	\I(g(\bm x^*,\bm\xi^i)\leq 0) \geq 1-\varepsilon.\nonumber
			\end{align}
			As a direct application of theorem A.3.3. in \citealt{ben2009robust}, the following corollary shows that under mild conditions, the hinge-loss approximation \eqref{eq_also_x_discrete_general} can be tractable.
			\begin{corollary}
				\label{polynomial_support_tract}
				(theorem A.3.3. in \citealt{ben2009robust}) Suppose that (i) the encoding length of $t$ is polynomial in that of CCP \eqref{eq_ccp_general}; and (ii) the feasible region of the hinge-loss approximation is contained in a Euclidean ball with radius $R$ and is containing a Euclidean ball with radius $r$. Then there exists an efficient algorithm to solve the hinge-loss approximation \eqref{eq_also_x_discrete_general} to $\hat\varepsilon>0$ accuracy, whose running time is polynomial in 
				$n,m,I,N,\ln(R/r),\ln(1/\hat\varepsilon)$, and the encoding length of CCP \eqref{eq_ccp_general}.
			\end{corollary}
		}
		
		{\subsection{Tractability of $\alsoxt$ Under Elliptical Distributions}
			\label{alsox_el_tractable}
			For the single linear CCP (1), i.e., $I = 1$ and $g(\bm x,{{\rxi}}) =  \rxi^\top \bm a_1(\bm x)-b_1(\bm x)$ with affine functions $\bm a_1(\bm x)$, $ b_1(\bm x)$, if the random parameters $\tilde{\bm {\xi}}$ follow a joint elliptical distribution with $\tilde{\bm {\xi}} \thicksim \Pr_{\mathrm{E}}(\bm{\mu},\bm{\mathrm{{\Sigma}}},\hat g)$,  CCP \eqref{eq_ccp} reduces to 
			\begin{align}
			v^*=	\min_{\bm{x}\in \mathcal{X}} \left\{\bm{c}^\top \bm{x} \colon b_1(\bm x)-\bm{\mu}^\top\bm{a}_1(\bm {x})\geq \mathrm{\Phi}^{-1}(1-\varepsilon) \sqrt{\bm{a}_1(\bm{x})^\top\bm{\mathrm{{\Sigma}}}\bm{a}_1(\bm {x})}\right\}, \label{ccp_el_general}
			\end{align}
			and by projecting out variable $\alpha$ in \eqref{eq_alsox_el}, $\alsox$ admits the following form 
			\begin{align}
			v^A = &  \min _{ {t}}\quad t\nonumber\\
			&	\begin{aligned}
			\text{s.t.}\quad & \bm x^*\in\argmin_{\bm{x}\in \mathcal{X}} \biggl\{ \biggl(1-\mathrm{\Phi}(\frac{b_1(\bm x)-\bm\mu^\top \bm {x}}{\sqrt{\bm{a}_1(\bm{x})^\top\bm{\mathrm{{\Sigma}}}\bm{a}_1(\bm{x})}})\biggr)\biggl( \bm\mu^\top \bm{a}_1(\bm{x})-b_1(\bm x)\biggr) \\
			& \quad +  \sqrt{\bm{a}_1(\bm{x})^\top\bm{\mathrm{{\Sigma}}}\bm{a}_1(\bm{x})}\overline{G}\biggl(\frac{1}{2}(\frac{b_1(\bm x)-\bm\mu^\top \bm{a}_1(\bm{x}) }{\sqrt{\bm{a}_1(\bm{x})^\top\bm{\mathrm{{\Sigma}}}\bm{a}_1(\bm{x})}})^2\biggr) \colon \bm{c}^\top \bm{x} \leq t\biggr\}, \label{alsox_el_general}
			\end{aligned}  \\
			&\quad\quad\quad b_1(\bm x^*)-\bm{\mu}^\top\bm{a}_1(\bm {x}^*)\geq \mathrm{\Phi}^{-1}(1-\varepsilon) \sqrt{\bm{a}_1(\bm{x}^*)^\top\bm{\mathrm{{\Sigma}}}\bm{a}_1(\bm {x}^*)}.\nonumber
			\end{align}
			Similarly, the following corollary shows that under mild conditions, the hinge-loss approximation \eqref{alsox_el_general} can be tractable.
			\begin{corollary}
				\label{el_tract}
				(theorem A.3.3. in \citealt{ben2009robust}) Suppose that (i) the encoding length of $t$ is polynomial in that of CCP \eqref{ccp_el_general}; and (ii) the feasible region of the hinge-loss approximation is contained in a Euclidean ball with radius $R$ and is containing a Euclidean ball with radius $r$. Then there exists an efficient algorithm to solve the hinge-loss approximation \eqref{alsox_el_general} to $\hat\varepsilon>0$ accuracy, whose running time is polynomial in 
				$n,m,\ln(R/r),\ln(1/\hat\varepsilon)$, and the encoding length of CCP \eqref{ccp_el_general}.
		\end{corollary}}
		
		\newpage
		
		{\section{An Example when $\alsoxt$ Fails to Find any Feasible Solution}\label{sec_append_example2}
			\begin{example}
				\label{alsox_infeasible_example}
				\rm
				Consider a CCP with 3 equiprobable scenarios (i.e., $N=3$, $\Pr\{\tilde{\rxi}=\bm\xi^i\}=1/N$), risk level $\varepsilon=1/3$, set $\mathcal{X}=\Re_+^2$, function $g(\bm{x},{\rxi})=-{\rxi}_1^\top \bm{x} +\xi_2$, and 
				$\bm\xi_1^1 =(1,0)^\top$, $\bm\xi_1^2 =(0,1)^\top$, $\bm\xi_1^3 =(1,1)^\top $, $\xi_2^1=\xi_2^2=1, \xi_2^3=-1$.
				The optimal value of this CCP can be found by solving the following mixed-integer linear program
				\begin{equation*}
				v^*=	\min_{\bm{x}\in\Re_+^2}\left\{ x_1+x_2\colon  \I(x_1\geq 1)+ \I(x_2\geq 1)+ \I(-x_1-x_2\geq -1)\geq 2 \right\},
				\end{equation*}	
				i.e., $v^*=1$. \\
				$\alsox$  of this example may be infeasible, which can be formulated as
				\begin{align*}
				&v^A =\min _{ {t}}\,\biggl\{ t\colon \sum_{i\in[3]}	\I({s^*_i}=0) \geq 2,\\
				&
				( \bm x^*,\bm s^*)\in\argmin_{\bm{x}\in\Re^2_+,\bm{s}\in\Re^3_+}\biggl\{\frac{1}{3} \sum_{i\in[3]}s_i\colon  x_1\geq 1-s_1,x_2\geq 1-s_2,-x_1-x_2\geq -1-s_3,x_1+x_2\leq t  \biggr\} \biggr\}.
				\end{align*}
				When $t\geq 1$, the hinge-loss approximation returns a solution with 
				$x_1^*=1/2, x_2^*=1/2, s^*_1=1/2,s^*_2=1/2, s_3^*=0$, and the support size of $\bm s^*$ is greater than $1$, then we have to increase the objective bound $t$ to the infinity. Therefore, in this example, $\alsox$ cannot return any feasible solution. Simple calculations show that $\CVaR$ approximation is also infeasible in this example.
				\QEDB
		\end{example}}

		\newpage
		{
			\section{Complexity of CCP \eqref{ccp_el} when $\varepsilon\in(0.5,1)$}\label{proof_el_hard}
			\begin{restatable}{proposition}{elhard}\label{el_hard} 
				When $\varepsilon\in(0.5,1)$, CCP \eqref{ccp_el} in general is NP-hard.
			\end{restatable}
			\begin{proof}
				Let us first consider the NP-hard problem - optimization of a general binary program \citep{garey1979guide}, which asks
				\begin{quote}\it
					\textbf{Optimization of a general binary program.} Given an integer matrix $\bm D \in \Ze^{m\times n}$, and integer vector $\bm d\in\Ze^m$, what is an optimal solution of the problem $\min_{\bm x\in\{0,1\}^n}\{ \bm c^\top \bm x : \bm D\bm x\geq \bm d\}$? 
				\end{quote}
				Consider a special case of CCP \eqref{ccp_el}, where set $\X=\{(\bm x,\bm y):  \bm D\bm x\geq \bm d,\bm x+\bm y=\bm e,  \bm x,\bm y\in[0,1]^n\}$, affine functions $b_1(\bm x,\bm y)-\bm{\mu}^\top\bm{a}_1(\bm {x},\bm y) =\mathrm{\Phi}^{-1}(1-\varepsilon)\sqrt{n}$ and $\bm{a}_1(\bm x,\bm y)=(\bm x,\bm y)$, the covariance matrix $\bm{\mathrm\Sigma}=\bm I_{2n}$. In this case, CCP \eqref{ccp_el}  can be rewritten as
				\begin{align*}
				v^*=	\min_{\bm{x},\bm y} \left\{\bm{c}^\top \bm{x} \colon \mathrm{\Phi}^{-1}(1-\varepsilon) \sqrt{n} \geq \mathrm{\Phi}^{-1}(1-\varepsilon) \sqrt{\sum_{i\in[n]}(x_i^2+y_i^2)},\bm x+\bm y=\bm e, \bm D\bm x\geq \bm d,  \bm x,\bm y\in[0,1]^n \right\}.
				\end{align*}
				Since $\varepsilon\in(0.5,1)$, we must have $\mathrm{\Phi}^{-1}(1-\varepsilon)<0$ and CCP \eqref{ccp_el} is
				\begin{align*}
				v^*=	\min_{\bm{x},\bm y} \left\{\bm{c}^\top \bm{x} \colon n \leq \sum_{i\in[n]}(x_i^2+y_i^2),\bm x+\bm y=\bm e, \bm D\bm x\geq \bm d,  \bm x,\bm y\in[0,1]^n \right\}.
				\end{align*}
				Since for each $i\in [n]$, the maximization problem $\max_{x_i,y_i\in [0,1]}(x_i^2+y_i^2)=1$ has two optimal solutions $x_i=0,y_i=1$ or $x_i=1,y_i=0$, the constraint $n \leq \sum_{i\in[n]}(x_i^2+y_i^2)$ is satisfied if and only if $\bm{x}\in \{0,1\}^n$ and $\bm x+\bm y=\bm e$. Thus, projecting out variables $\bm{y}$, CCP \eqref{ccp_el} can be further reduced to
				\begin{align*}
				v^*=	\min_{\bm{x}} \left\{\bm{c}^\top \bm{x} \colon  \bm D\bm x\geq \bm d,  \bm x\in\{0,1\}^n \right\}.
				\end{align*}
				which is exactly the desirable binary program.  This completes the proof.	\QEDA
			\end{proof}
		}

		\newpage
		\section{Comparing $\alsoxt +$ \Cref{alg_alsox_+} and Exact Big-M Model}\label{numerical_big_m}
		Big-M model is known to work well for solving a CCP \citep{ahmed2017nonanticipative}. Albeit being a heuristic, the proposed $\alsox +$ \Cref{alg_alsox_+} can effectively identify better feasible solutions than the exact Big-M model with a much shorter solution time. To illustrate this, we use $``\textrm{UB}" $ and $``\textrm{LB}" $ to denote the best upper bound and the best lower bound found by the $\text{Big-M}$ model. Since we may not be able to solve the $\text{Big-M}$ model to optimality within the time limit, we use GAP to denote its optimality gap as
		\begin{align*}
		\textrm{GAP} (\%)= \frac{| \textrm{UB} -\textrm{LB} | }{|\textrm{LB}|}\times 100,
		\end{align*}
		while we use the term ``Improvement'' to denote the solution quality of $\alsox +$ \Cref{alg_alsox_+}
		\begin{align*}
		\textrm{Improvement} (\%)= \frac{\textrm{UB} -\textrm{value of the $\alsox +$ \Cref{alg_alsox_+}} }{|\textrm{UB} |}\times 100.
		\end{align*}
		The numerical results are shown in \Cref{tab_nonlinear_comparisons_single_small} and \Cref{tab_nonlinear_comparisons_single_large}. 
		It is seen that {for most instances, especially for those with a larger problem dimension, the $\text{Big-M}$ model cannot be solved to optimality, while $\alsox +$ \Cref{alg_alsox_+} can provide better solutions than the best upper bounds found by the $\text{Big-M}$ model in a much shorter time, and $\alsox+$ can consistently find near-optimal solutions or even optimal solutions,}
		which further validates the efficacy of our proposed methods.
		\begin{table}[htbp]
			\centering
			\caption{Comparisons Between the Exact Big-Model and $\alsox+$ \Cref{alg_alsox_+} for Solving the Nonlinear CCP with Small Instances}
			\setlength{\tabcolsep}{2pt} 
			\renewcommand{\arraystretch}{1} 
			\label{tab_nonlinear_comparisons_single_small}
			\tiny
			\begin{center}
				\begin{tabular}{c c  r r r r| r r  r r }
					\hline
					\multicolumn{1}{c}{ \multirow{3}*{$N$} }&
					\multicolumn{1}{c}{ \multirow{3}*{$n$} }&
					\multicolumn{4}{c|}{$\varepsilon=0.05$} & \multicolumn{4}{c}{$\varepsilon=0.10$} \\
					\cline{3-10}
					& &
					\multicolumn{2}{c}{$\text{Big-M}$ Model} &
					\multicolumn{2}{c|}{$\alsox +$} &
					\multicolumn{2}{c}{$\text{Big-M}$ Model} &
					\multicolumn{2}{c}{$\alsox +$} 	\\
					\cline{3-10}
					\multicolumn{1}{c}{} & &Gap (\%)&Time (s) & \makecell{Improve-\\ment (\%)} &Time (s) &Gap (\%)&Time (s) & \makecell{Improve-\\ment (\%)} &Time (s)\\
					\hline	
					\multirow{3}{*}{{30}} 
					&{20} & 0.00& 5.87&0.00& 5.58& 0.00 & 8.93&0.00& 7.31\\
					\cline{2-10}
					&{40} & 0.00& 12.62&0.00&8.46& 0.00 & 19.53&0.00&9.89\\
					\cline{2-10}
					&{100} & 0.00&2076.30&-0.21&8.98 & 0.00 &3454.75&-0.22&10.03 \\
					\cline{1-10}
					\multirow{3}{*}{{40}} 
					&{20} & 0.00& 7.68 &0.00 & 8.67  &0.00 & 12.53 &0.00 & 11.37  \\
					\cline{2-10}
					&{40} & 0.00& 23.45&0.00& 11.97  & 0.00& 263.56 &0.00 & 15.84  \\
					\cline{2-10}
					&{100} &2.44& 3600&-0.30& 17.39&6.86& 3600&-0.92& 20.53\\
					\cline{1-10}
					\multirow{3}{*}{{50}} 
					&{20} & 0.00 & 15.22 & -0.23 &10.87& 0.00 & 98.01 &-0.33  & 11.93  \\
					\cline{2-10}
					&{40} & 0.00& 66.36&-0.32& 12.46&0.00& 2190.43&-0.18& 13.71\\
					\cline{2-10}
					&{100} &3.42& 3600&-0.81& 18.53&19.67& 3600&0.05& 25.49\\
					\hline
				\end{tabular}
			\end{center}
		\end{table}
		
		\begin{table}[htbp]
			\centering
			\caption{Comparisons Between the Exact Big-Model and $\alsox+$ \Cref{alg_alsox_+} for Solving the Nonlinear CCP with Large Instances}
			\setlength{\tabcolsep}{2pt} 
			\renewcommand{\arraystretch}{1} 
			\label{tab_nonlinear_comparisons_single_large}
			\tiny
			\begin{center}
				\begin{tabular}{c c  r r r r| r r  r r }
					\hline
					\multicolumn{1}{c}{ \multirow{3}*{$N$} }&
					\multicolumn{1}{c}{ \multirow{3}*{$n$} }&
					\multicolumn{4}{c|}{$\varepsilon=0.05$} & \multicolumn{4}{c}{$\varepsilon=0.10$} \\
					\cline{3-10}
					& &
					\multicolumn{2}{c}{$\text{Big-M}$ Model} &
					\multicolumn{2}{c|}{$\alsox +$} &
					\multicolumn{2}{c}{$\text{Big-M}$ Model} &
					\multicolumn{2}{c}{$\alsox +$} 	\\
					\cline{3-10}
					\multicolumn{1}{c}{} & &Gap (\%)&Time (s) & \makecell{Improve-\\ment (\%)} &Time (s) &Gap (\%)&Time (s) & \makecell{Improve-\\ment (\%)} &Time (s)\\
					\hline	
					\multirow{3}{*}{400} 
					&20 & 14.48& 3600 &1.31 &16.74 & 14.83& 3600 &0.47 &13.29 \\
					\cline{2-10}
					&40 & 16.72& 3600&1.93 & 16.84 & 29.81& 3600&1.99 & 16.53\\
					\cline{2-10}
					&100 &198.71& 3600&2.81& 31.68&58.20& 3600&2.04& 25.25\\
					\cline{1-10}
					\multirow{3}{*}{600} &20 & 12.37& 3600&1.95& 17.61& 27.51 & 3600&2.19& 22.08\\
					\cline{2-10}
					&40 & 25.74& 3600&1.86&18.39 & 35.14 & 3600&1.69&22.88\\
					\cline{2-10}
					&100 & 251.84&3600&1.33&28.21 & 78.06 &3600&2.63&31.60 \\
					\cline{1-10}
					\multirow{3}{*}{1000} &20 & 17.70& 3600&2.55& 20.24& 27.76& 3600&1.52& 34.48\\
					\cline{2-10}
					&40 &31.85& 3600 &3.11& 21.74&55.97& 3600 &2.87& 47.55\\
					\cline{2-10}
					&100 & 421.88& 3600&2.18&33.50& 345.89& 3600&2.71&63.30 \\
					\hline
				\end{tabular}
			\end{center}
		\end{table}

	\end{appendices}
\end{document}